\documentclass[times]{oupau}

\usepackage{amsmath}
\usepackage{amsthm}
\usepackage{amsfonts}
\usepackage{amssymb}
\usepackage{graphics}

\newcommand\R{{\mathbb{R}}}
\newcommand\C{{\mathbb{C}}}

\newcommand\Z{{\mathbf{Z}}}

\newcommand\Q{{\mathbf{Q}}}

\renewcommand\P{{\mathbf{P}}}
\newcommand\E{{\mathbf{E}}}

\newcommand\Var{\mathbf{Var}}
\renewcommand\Im{{\operatorname{Im}}}
\renewcommand\Re{{\operatorname{Re}}}
\newcommand\eps{{\varepsilon}}

\newcommand\bv{\mathbf v}
\newcommand\bw{\mathbf w}

\theoremstyle{plain}

\theoremstyle{definition}
  \newtheorem{remark}[theorem]{Remark}

\begin{document}

\title{Local universality of zeroes of random polynomials}
\shorttitle{Universality of zeroes of polynomials}

\author{Terence Tao\affil{1} and Van Vu\affil{2}}
\abbrevauthor{T. Tao and V. Vu}

\address{\affilnum{1}Department of Mathematics, UCLA, Los Angeles CA 90095-1555 and \affilnum{2} Department of Mathematics, Yale, New Haven, CT 06520}

\correspdetails{tao@math.ucla.edu}

\begin{abstract} 
In this paper, we establish some local universality results concerning the correlation functions of the zeroes of random polynomials with independent coefficients. More precisely, consider two random polynomials 
$f =\sum_{i=1}^n c_i \xi_i z^i$ and $\tilde f =\sum_{i=1}^n c_i \tilde \xi_i z^i$, where the $\xi_i$ and $\tilde \xi_i$ are iid random variables that match moments to second order, the coefficients $c_i$ are deterministic, and the degree parameter $n$ is large.  Our results show, under some light conditions on the coefficients 
 $c_i$ and the tails of $\xi_i, \tilde \xi_i$, that the correlation functions of the zeroes of $f$ and $\tilde f$ are approximately the same. 
As an application, we give some answers to the classical  question

\centerline{  ``{\it How many zeroes of a random polynomials are real ?}''  } 
\noindent  for several classes of random polynomial models.

Our analysis relies on  a general {\it replacement principle}, motivated by some recent work in random matrix theory.
This principle enables one to compare the correlation functions of two random functions $f$ and $\tilde f$ if their log magnitudes $\log |f|, \log|\tilde f|$ are close in distribution, and if some non-concentration bounds are obeyed.
\end{abstract}

\received{}
\revised{}
\accepted{}
\communicated{}

\maketitle


\setcounter{tocdepth}{1}
\tableofcontents


\section{Introduction} \label{section:introduction} 

\subsection{Models of random polynomials} 

In this paper we study the distribution of the zeroes of a random polynomial $f=f_n$ when the degree parameter $n$ is large (or goes asymptotically to infinity).  For sake of exposition,  we will focus on a simple model of random polynomials in which the coefficients are independent and derived from a common atom distribution, although several of our results extend to more general models. 

\begin{definition}[Random polynomials]\label{Rand-def}  Let $n$ be a positive integer, let $c_0,\ldots,c_n$ be deterministic complex numbers, and let $\xi$ be a complex random variable (which we call the \emph{atom distribution}) of mean zero and finite non-zero variance.
Given the coefficients $c_0,\ldots,c_n$ and atom distribution $\xi$, we associate the random polynomial $f=f_n = f_{n,\xi}: \C \to \C$ defined by the formula
$$ f(z) := \sum_{i=0}^n c_i \xi_i z^i,$$
where $\xi_0,\ldots,\xi_n$ are jointly independent copies of $\xi$.
\end{definition}

In practice, we will usually normalize the atom distribution $\xi$ to have unit variance; note that this normalization clearly does not affect the zeroes of $f$.  In some literature, one replaces either the bottom coefficient $\xi_0$ or the top coefficient $\xi_n$ with the constant $1$. This generally has a negligible impact on the distribution of the zeroes in the large $n$ limit; however,  we shall avoid such normalizations here (although this has the consequence that the polynomial $f$ may occasionally have degree less than $n$, or even vanish entirely).  Our focus in this paper will primarily be on the \emph{universality phenomenon} in the context of zeroes of such random polynomials, which roughly speaking asserts that the (appropriately normalized) asymptotic behavior of these zeroes as $n \to \infty$ should become independent of the choice of atom distribution.

We isolate three specific choices of coefficients $c_i$ that have been studied for a long time: 

\begin{itemize} 
\item[(i)] \emph{Flat polynomials} or \emph{Weyl polynomials} are polynomials associated to the coefficients $c_i := \sqrt {\frac{ 1}{ i!} } $. 
\item[(ii)] \emph{Elliptic polynomials} or \emph{binomial polynomials} are polynomials associated to the coefficients $c_i:= \sqrt{\binom{n}{i}} $. 
\item[(iii)] \emph{Kac polynomials} are polynomials associated to the coefficients $c_i := 1$.
\end{itemize}  

One can view Kac polynomials as the special case $L=1$ of \emph{hyperbolic polynomials} in which $c_i:= \sqrt {\frac{ L(L+1) \cdots (L+i-1)}{i!} }$ for some parameter $L > 0$, but for simplicity we will focus on the classical Kac polynomial case as a proxy for the more general hyperbolic case.

These polynomials have been intensively studied, particularly in the case when the atom distribution $\xi$ is either the real gaussian $N(0,1)_\R$ or the complex gaussian $N(0,1)_\C$; see \cite{kab, manju}.  As we shall recall later, the situation with the Kac polynomials is somewhat special when compared to the other models described above (its zeroes tend to cluster around the unit circle, instead of being distributed throughout a two-dimensional region in the plane).
 
If $f=f_n$ is a random polynomial of the form described in Definition \ref{Rand-def}, then $f$ has degree at most $n$.  If $f$ is not identically zero, then from the fundamental theorem of algebra it has $\operatorname{deg}(f)$ zeroes in the complex plane $\C$ (counting multiplicity).  We adopt the convention that $f$ also has $n - \operatorname{deg}(f)$ zeroes at infinity, and when $f$ is identically zero we adopt the convention that $f$ has $n$ zeroes at infinity and no zeroes in $\C$. With these (admittedly artificial) conventions, $f$ thus always has $n$ zeroes $\zeta_1,\ldots,\zeta_n$ (ordered in some arbitrarily chosen fashion, e.g. lexicographically) in the Riemann sphere $\C \cup \{\infty\}$ (counting multiplicity), so that $\{\zeta_1,\ldots,\zeta_n\}$ may be viewed as a point process in $\C \cup \{\infty\}$.  We will sometimes refer to the set $\{\zeta_1,\ldots,\zeta_n\}$ as the \emph{spectrum} of $f$.

\subsection{Number of real zeroes} 

With $f=f_n$ as above, and any subset $\Omega$ of $\C$, write
$$ N_\Omega := |\{ 1\leq i \leq n: \zeta_i \in \Omega \} |$$
for the number of zeroes of $f$ in $\Omega$ (counting multiplicity); in particular, $N_\R$ is the number of real zeroes. This is a random variable taking values in $\{0,\ldots,n\}$.  The issue of understanding the typical size of $N_\R$ was already raised by Waring as far back as 1782 (\cite[page 618]{To}, \cite{Kostlan}), and has generated a huge amount of literature, of which we now pause to give a (incomplete and brief) survey.  The statistic $N_\R$ is of interest primarily in the case when the atom distribution $\xi$ and the coefficients $c_i$ are both real-valued, since in the genuinely complex case one usually expects that none of the zeroes of the associated polynomial $f$ will be real.

Most earlier works focused on the case of Kac polynomials, which are easier to analyze but have an atypical behavior compared to other random polynomial models.  One of the first results in this context is by \cite{BP}, who  studied the case of Kac polynomials with $\xi$ uniformly distributed in $\{-1,0,1\}$, and established the somewhat weak upper bound
$$ \E N_\R \ll n^{1/2}$$
where we use the usual asymptotic notation $X=O(Y)$ or $X \ll Y$ to denote the bound $|X| \leq CY$ where $C$ is independent of $Y$.  This bound is not sharp, and Kac polynomials actually have a remarkably small number of real zeroes.  Indeed, in a series of papers \cite{lo,  lo-3, lo-4, lo-2}, Littlewood and Offord proved that 
 for Kac polynomials with  many basic atom distributions (such as gaussian, Bernoulli or uniform on $[-1,1]$), one has the bounds
$$ \frac{\log n}{\log \log \log n} \ll N_\R \ll \log^2 n $$
with probability $1-o(1)$, where we use $o(1)$ to denote a quantity that goes to $0$ as $n \to \infty$.
  
Later, \cite{kac-0} found an exact formula for $\E N_\R$ in the case that $\xi_i$ are real gaussians, and showed that
$$ \E N_\R = \left(\frac{2}{\pi} +o(1)\right) \log n$$
in this case (see \cite{Wa}, \cite{EK} for  more precise asymptotics).
 
This asymptotic has been extended to Kac polynomials with more general atom distributions.  In a subsequent paper \cite{kac-1}, the result was extended to the case when $\xi$ has the uniform distribution on $[-1,1]$.  Erd\"os and Offord  \cite{EO} extended the result to the Bernoulli distribution case (i.e. when $\xi$ is uniform on $\{-1,+1\}$).  \cite{Stev}  extended the asymptotics  for a wide class of distributions, and finally \cite{IM1, IM2}  extended 
 the result to all mean-zero distributions in the domain of attraction of the normal law, with the extra assumption that 
$\P(\xi=0 )=0$.  In \cite{Mas1, Mas2} it was also proved that if $\P( \xi=0) =0$ and $\E |\xi|^{2 +\eps} < \infty$ for some constant $\eps >0$ then the variance of $N_\R$ is $(\frac{4}{\pi} (1- \frac{2}{\pi}) + o(1)) \log n$, and furthermore established a central limit theorem for $N_\R$. In \cite{Dembo}, the probability that $N_\R=k$ for any fixed $k$ was computed.  There are also some non-trivial deterministic bounds bounds for the maximum value of $N_\R$ (when the coefficients are, say, drawn from $\{-1,0,1\}$) which we will not describe in detail here, but see e.g. \cite{erd} for some recent results in this direction.
 
For non-Kac models such as the flat or elliptic polynomial ensembles, the behavior of $N_\R$ changes considerably.  These types of random polynomial models were already studied to some extent in the classical papers of Littlewood and Offord, but most of the work on these models appeared later, partially motivated by connections to physics \cite{BBL} and random analytic functions  \cite{manju} or problems in numerical analysis and   computation theory \cite{Kostlan, SSm}. In particular, many researchers consider the elliptic (or binomial) polynomial the most ``natural'' random polynomial \cite[Section 1]{EK}, \cite{Kostlan, SSm}; one reason for this is that in the case when the atom distribution $\xi$ is the complex Gaussian $N(0,1)_\C$, the distribution of the zeroes of the associated random polynomial is invariant with respect to rotations of the Riemann sphere $\C \cup \{\infty\}$ (see e.g. \cite[Proposition 2.3.4]{manju}).
 
It is known that when the atom distribution $\xi$ is the real Gaussian $N(0,1)_\R$, one has
$$ \E N_\R = \left(\frac{2}{\pi} + o(1)\right) \sqrt{n}$$
for flat (Weyl) polynomials and
$$ \E N_\R =  \sqrt{n}$$
for elliptic (binomial) polynomials; see \cite{EK} for a nice geometric proof of these facts.  Thus one has substantially more zeroes for such polynomials than in the Kac case when $n$ is large.  However, unlike the situation with Kac polynomials, the extension of these results to more general (non-gaussian)
 distributions was not fully understood.  The reader is referred to  \cite{IM3, IM4, LS1, LS2, Wil, bleher, shiff2} and the books \cite{BSbook, Far} for several results and further discussion. 

\subsection{Distribution of zeroes: Correlation functions} 

We now turn to a  popular way to study the distribution of zeroes of random polynomials, namely by investigating  their \emph{correlation functions}.  To define these functions, let us first consider the complex case in which the coefficients $c_i$ and the atom distribution $\xi$ are not required to be real valued.  In this case the point process $\{\zeta_1,\ldots,\zeta_n\}$ of zeroes of a random polynomial $f=f_n$ can be described using the (complex) \emph{$k$-point correlation functions} $\rho^{(k)} = \rho^{(k)}_f: \C^k \to \R^+$, defined for any fixed natural number $k$ by requiring that
\begin{equation}\label{cord}
 \E \sum_{i_1,\ldots,i_k \hbox{ distinct}} \varphi(\zeta_{i_1},\ldots,\zeta_{i_k}) = \int_{\C^k} \varphi(z_1,\ldots,z_k) \rho^{(k)}(z_1,\ldots,z_k)\ dz_1\ldots dz_k
 \end{equation}
for any continuous, compactly supported, test function $\varphi: \C^k \to \C$, with the convention that $\varphi(\infty)=0$;
see e.g. \cite{manju, AGZ}.  This definition of $\rho^{(k)}$ is clearly independent of the choice of ordering $\zeta_1,\ldots,\zeta_n$ of the zeroes. Note that if the random polynomial $f$ has a discrete law rather than a continuous one, then $\rho^{(k)}_{f}$ needs to be interpreted as a measure\footnote{We point out one subtlety in the discrete case: the summation \eqref{cord} requires the \emph{indices} $i_1,\ldots,i_k$ to be distinct, but allows the \emph{zeroes} $\zeta_{i_1},\ldots,\zeta_{i_k}$ to be repeated; thus for instance if $f$ is the deterministic polynomial $z^n$ then $\rho^{(k)}$ is $\frac{n!}{(n-k)!}$ times the Dirac mass at the origin in $\C^k$.  This convention allows for identities such as $(n-k) \rho^{(k)}(z_1,\ldots,z_k) = \int_{\C \cup \infty} \rho^{(k+1)}(z_1,\ldots,z_{k+1})\ dz_{k+1}$ to be extended to the discrete setting (after being interpreted in an appropriate distributional sense); it also ensures that the distribution functions $\rho^{(k)}$ vary continuously (in the vague topology) with respect to perturbations of law of the random polynomial $f$ (again measured in the vague topology).  Of course, when $f$ has a continuous distribution, the zeroes are almost surely simple, and this subtlety becomes irrelevant.} rather than as a function.  

\begin{remark} When $\xi$ has a continuous complex distribution, the $c_i$ are non-zero, then the zeroes are almost surely simple.  In this case if  $z_1,\ldots,z_k$ are distinct, and one can interpret $\rho^{(k)}(z_1,\ldots,z_k)$ as the unique quantity such that the probability that there is a zero in each of the disks $B(z_i,\eps)$ for $i=1,\ldots,k$ is $(\rho^{(k)}(z_1,\ldots,z_k)+o_{\eps \to 0}(1)) (\pi \eps^2)^k$ in the limit $\eps \to 0$.
\end{remark}

When the random polynomials $f$ have real coefficients,  the zeroes $\zeta_1,\ldots,\zeta_n$ are symmetric around the real axis, and one expects several of the zeroes to lie on this axis.  Because of this, it is not as natural to work with the complex $k$-point correlation functions $\rho^{(k)}_f$, as they are likely to become singular on the real axis. Instead, we divide the complex plane $\C$ into three pieces $\C = \R \cup \C_+ \cup \C_-$, with $\C_+ := \{ z \in \C: \Im(z)>0\}$ being the upper half-plane and $\C_- := \{z \in \C: \Im(z)<0\}$ being the lower half-plane.  By the aforementioned symmetry, we may restrict attention to the zeroes in $\R$ and $\C_+$ only.  For any natural numbers $k,l \geq 0$, we then define the \emph{mixed $(k,l)$-correlation function} $\rho^{(k,l)} = \rho^{(k,l)}_{f}: \R^k \times (\C_+ \cup \C_-)^l \to \R^+$ of a random polynomial $f$ to be the function defined by the formula
\begin{eqnarray} \label{cord1}
& \E \sum_{i_1,\ldots,i_k \hbox{ distinct}} \sum_{j_1,\ldots,j_l \hbox{ distinct}} \varphi(\zeta_{i_1,\R},\ldots,\zeta_{i_k,\R},\zeta_{j_1,\C_+},\ldots,\zeta_{j_l,\C_+})  \\  \nonumber
&\quad = \int_{\R^k} \int_{\C_+^l} \varphi(x_1,\ldots,x_k,z_1,\ldots,z_l) \rho^{(k,l)}_{f}(x_1,\ldots,x_k,z_1, \ldots,z_l)\ dz_1\ldots dz_l dx_1 \ldots dx_k 
\end{eqnarray}	
for any continuous, compactly supported test function $\varphi: \R^k \times \C^l \to \C$ (note that we do not require $\varphi$ to vanish at the boundary of $\C_+^l$), $\zeta_{i,\R}$ runs over an arbitrary enumeration of the real zeroes of $f_n$, and $\zeta_{j,\C_+}$ runs over an arbitrary enumeration of the zeroes of $f_n$ in $\C_+$.  This defines $\rho^{(k,l)}$ (in the sense of distributions, at least) for $x_1,\ldots,x_k \in \R$ and $z_1,\ldots,z_l \in \C_+$; we then extend $\rho^{(k,l)}(x_1,\ldots,x_k,z_1,\ldots,z_l)$ to all other values of $x_1,\ldots,x_k \in \R$ and $z_1,\ldots,z_l \in \C_+ \cup \C_-$ by requiring that $\rho^{(k,l)}$ is symmetric with respect to conjugation of any or all of the $z_1,\ldots,z_l$ parameters.  Again, we permit $\rho^{(k,l)}$ to be a measure\footnote{As in the complex case, we allow the real zeros $\zeta_{i_1,\R},\ldots,\zeta_{i_k,\R}$ or the complex zeroes $\zeta_{j_1,\C_+},\ldots,\zeta_{j_l,\C_+}$ to have multiplicity; it is only the indices $i_1,\ldots,i_k,j_1,\ldots,j_l$ that are required to be distinct.  In particular, in the discrete case it is possible for $\rho^{(0,2)}(z_1,z_2)$ (say) to have non-zero mass on the diagonal $z_1=z_2$ or the conjugate diagonal $z_1 = \overline{z_2}$, if $f$ has a repeated complex eigenvalue with positive probability.} instead of a function when the random polynomial $f_n$ has a discrete distribution.

In the case $l=0$, the correlation functions $\rho^{(k,0)}$ for $k \geq 1$ provide (in principle, at least) all the essential information about the distribution of the real zeroes, which as mentioned previously, was   the original motivation of the very first papers studying random polynomials.   For instance, one easily verifies the identity
\begin{equation}\label{enr}
 \E N_\R = \int_\R \rho^{(1,0)}(x)\ dx
\end{equation}
and similarly
\begin{equation}\label{varr}
 \Var N_\R = \int_\R \int_\R \rho^{(2,0)}(x,y) -\rho^{(1,0)}(x) \rho^{(1,0)}(y) \ dx \ dy  + \int_\R \rho^{(1,0)}(x)\ dx
\end{equation}

\begin{remark} When $\xi$ has a continuous real distribution, the $c_i$ are non-zero real, then the zeroes are almost surely simple. If  the $x_1,\ldots,x_k \in \R$ are distinct, and the $z_1,\ldots,z_l \in \C_+$ are distinct, then  one can interpret $\rho^{(k,l)}(x_1,\ldots,x_k,z_1,\ldots,z_l)$ as the unique quantity such that the probability that there is a zero in each of the intervals $[x_i-\eps,x_i+\eps]$ and disks $B(z_j,\eps)$ for $i=1,\ldots,k$ and $j=1,\ldots,l$ is $(\rho^{(k,l)}(x_1,\ldots,x_k,z_1,\ldots,z_l)+o_{\eps \to 0}(1)) (2\eps)^k (\pi \eps^2)^l$ in the limit $\eps \to 0$.
\end{remark}

\begin{remark}  In principle, one could express the complex correlation functions in a distributional sense in terms of the real correlation functions, for instance we have
$$ \rho^{(1)}(z) = \rho^{(0,1)}(z) + \rho^{(1,0)}(\Re z) \delta(\Im z)$$
in the sense of distributions, where $\delta$ is the Dirac distribution at $0$, with similar (but significantly more complicated) identities for $\rho^{(k)}$ when $k>1$, reflecting the many combinatorial possibilities for $k$ complex zeroes to lie on the real line, or to be complex conjugates of each other.  We will however not use such identities in this paper.
\end{remark}

\subsection{Universality}  

In the case when the atom distribution $\xi$ is a real or complex gaussian, the correlation functions $\rho^{(k,l)}$ (in the real case) or $\rho^{(k)}$ (in the complex case) can be computed explicitly using tools such as the Kac-Rice formula; see \cite{manju} or Lemma \ref{krf} below.  When the atom distribution is not gaussian, the Kac-Rice formula is still available,  but is considerably less tractable.  Nevertheless, it has been widely believed that the asymptotic behavior of the correlation functions in the non-gaussian case should match that of the gaussian case once one has performed appropriate normalizations, at least if the atom distribution $\xi$ is sufficiently short-tailed. This type of meta-conjecture is commonly referred to as the \emph{universality phenomenon}. 

At macroscopic (or \emph{global}) scales (comparable to the diameter of the bulk of the set of zeroes), universality results for polynomials given by Definition \ref{Rand-def} were established recently in \cite{kab}.  For instance, they established the analogue of the circular law for Weyl polynomials given only a mild log-integrability condition for the atom distribution (see \cite[Theorem 2.3]{kab}), as well as many other results of this nature; see \cite{kab} for full details.

In this paper we will be concerned primarily with universality of correlation functions at the microscopic (or \emph{local}) scale, comparable to the mean spacing between zeroes; through formulae such as \eqref{enr}, this also can lead to some partial universality results for quantities such as $N_\R$.
  At such microscopic scales, the most general previous result we found concerning universality is due to \cite[Theorem 7.2]{bleher}, who considered binomial polynomials in which the atom distribution $\xi$ was real-valued, with unit variance, and was sufficiently smooth and rapidly decaying (see \cite[Theorem 7.2]{bleher} for the precise technical conditions required on $\xi$).  With these hypotheses, they showed that the pointwise limit of the normalized correlation function $n^{-k/2} \rho^{(k,0)}(a+\frac{x_1}{\sqrt{n}},\ldots,a+\frac{x_k}{\sqrt{n}})$ for any fixed $k,a,x_1,\ldots,x_k$ (with $a \neq 0$) as $n \to \infty$ was independent of the choice of $\xi$ (with an explicit formula for the limiting distribution).  Again, see \cite[Theorem 7.2]{bleher} for a precise statement.  One of the main tools used in that argument was the Kac-Rice formula.

In this paper, we  first introduce a  new method to prove universality, which makes no distinction between continuous and discrete random variables (and in particular, avoids the use of the Kac-Rice formula, except when verifying a certain technical level repulsion estimate in the real gaussian case). As a matter of fact, we will only require some bounded moment assumption on the atom distribution.  This approach relies on a general \emph{replacement principle}, which we will present in the next section. This   reduces the task of establishing universality for zeroes of a random polynomial $f$ to that of establishing universality for the log-magnitude $\log |f(z)|$ of that polynomial evaluated at various points $z$, together with that of verifying some technical eigenvalue repulsion bounds.  This principle was implicitly introduced in our previous paper \cite{tv-iid} in the context where $f$ was the characteristic polynomial of a random matrix (and is thus can be viewed as a microsopic analogue to the macroscopic replacement principle in \cite[Theorem 2.1]{TVcir} to establish the circular law for various ensembles of random matrices), but applies for more general random matrix models, and is in fact particularly easy to use for the models in Definition \ref{Rand-def} since $f(z)$ is just the sum of independent random variables for each given $z$ in this case.
As applications of this principle we will establish universality results for all the classical ensembles listed above.  We would like to emphasize here 
that while in this paper we focus on random polynomials with independent coefficients, our replacement principle {\it does not}  require this assumption. For example, it can be applied to characteristic polynomials 
of random matrices \cite{tv-iid}.

\subsection{Notation} 

We use $1_E$ to denote the indicator of $E$, thus $1_E$ equals $1$ when $E$ is true and $0$ when $E$ is false.  We also write $1_\Omega(x)$ for $1_{x \in \Omega}$.

We use $\sqrt{-1}$ to denote the unit imaginary, in order to free up the symbol $i$ as an index of summation.  As we will be using two-dimensional integration on the complex plane $\C := \{ z = x+\sqrt{-1} y: x,y \in \R \}$ far more often than we will be using contour integration, we use $dz := dx dy$ to denote Lebesgue measure on the complex numbers, rather than the complex line element $dx + \sqrt{-1} dy$.  For $z \in \C$ and $r>0$, we use $B(z,r) := \{ w \in\C: |z-w| <r \}$ to denote the open disk of radius $r$ centered at $z$.

If $G: \R^k \to \C$ is a function and $a \geq 0$, we use $\nabla^a G$ to denote the tensor $(\frac{\partial^a}{\partial x_{i_1} \ldots \partial x_{i_a}} G)_{1 \leq i_1,\ldots,i_a \leq k}$; in particular,
$$ |\nabla^a G| := \left(\sum_{1 \leq i_1,\ldots,i_a \leq k} \left|\frac{\partial^a}{\partial x_{i_1} \ldots \partial x_{i_a}} G\right|^2\right)^{1/2}.$$

 Following \cite{TVlocal}, we  say that two complex random variables $\xi, \xi'$ \emph{match moments to order $m$} if one has
$$ \E \Re(\xi)^a \Im(\xi)^b = \E \Re(\xi')^a \Im(\xi')^b$$
for all natural numbers $a,b \geq 0$ with $a+b \leq m$.

\section{Replacement principle, complex case} 

Our replacement principle asserts, roughly speaking, that the $k$-correlation functions of the zeroes of two random polynomials $f$ and $\tilde f$ are asymptotically the same provided that 

\begin{itemize} 

\item[(i)] (Comparability of log-magnitudes) The joint distribution of $\log|f|$ at a 
few values is close  to the joint distribution of $\log|\tilde f|$ at those same values; and

\item[(ii)] (Non-clustering property) $f$ and $\tilde f$  do not have too many zeroes  concentrating in a small  region.  

\end{itemize} 

\vskip2mm

We will also need a mild non-degeneracy condition that prevents $f$ or $\tilde f$ from vanishing identically too often, but this hypothesis is easily verified in practice.

Moreover, we can show that the non-clustering property  holds if the variables  $\log |f(z)|$ and $\log |\tilde f(z)|$ are strongly concentrated around 
a suitable deterministic function $G(z)$ (see Proposition \ref{critloc} below).  So, in order to compare the distribution of the zeroes of $f$ and $\tilde f$, all we need is to study the distribution of the log-magnitude functions $\log |f(z)|$ and $\log |\tilde f(z)|$ for various choices of parameter $z$.

When the random polynomials $f$ and $\tilde f$ have real coefficients, we  can  prove a similar replacement principle for  the mixed $(k,l)$-point correlation functions involving $k$ real numbers and $l$ strictly complex numbers,  provided we assume an additional level repulsion estimate on at least one of $f,\tilde f$. In practice, this estimate will be easy to verify for many random polynomials with real gaussian coefficients.   

We now give the formal statement of the replacement principle in the complex case.

\begin{theorem}[Replacement principle, complex case]\label{replace}  Let $C, r_0 \ge    1 \ge c_0 > 0$ be real constants and $k, a_0 \geq 1$ be integer constants, and set
\begin{equation}\label{adef}
 A := \frac{100k a_0}{c_0}.
\end{equation}
Let $n \geq 1$ be a natural number, and let $f=f_n, \tilde f = \tilde f_n$ be random polynomials of degree at most $n$ (not necessarily of the form in Definition \ref{Rand-def}) and  $z_1,\ldots,z_k$ be complex numbers that are allowed to depend on $n$.  Assume the following axioms.

\begin{itemize}

\item[(i)] (Non-degeneracy)  With probability at least $1-C n^{-A}$, $f$ is not identically zero, and similarly for $\tilde f$.

\item[(ii)] (Non-clustering property)  For  $r \geq 1$,  one has $N_{B(z_i,r)}(f) \leq C n^{1/A} r^2$ with probability at least $1- C n^{-A}$.   Similarly for $\tilde f$.

\item[(iii)] (Comparability of log-magnitudes) Given any $1 \leq k' \leq n^{c_0}$, any complex numbers $z'_1,\ldots,z'_{k'} \in \bigcup_{i=1}^k B(z_i,20r_0)$, and any smooth function $F: \C^{k'}\to \C$ obeying the derivative bounds
$$ |\nabla^a F(w)| \leq n^{c_0}$$
for all $0\leq a \leq a_0$ and $w \in \C^{k'}$, we have 
\begin{equation}\label{leah}
\Big|  \E \Big( F(\log |f(z'_1)|, \ldots, \log |f(z'_{k'})|) - F(\log |\tilde f(z'_1)|, \ldots, \log |\tilde f(z'_{k'})|)  \Big) \Big|  \le C n^{-c_0}
\end{equation}
with the convention that $F$ vanishes when one or more of its arguments are undefined.
\end{itemize}

Let $G: \C^k\to \C$ be a smooth function supported on the polydisc $B(0,r_0)^k$ that obeys the bounds
\begin{equation}\label{gsmooth}
|\nabla^a G(w)| \leq M
\end{equation}
for all $0 \leq a \leq a_0 + 2k + 1$, all $w \in \C^k$, and some $M>0$.  Then
\begin{align*}
&\Big| \int_{\C^k} G(w_1,\ldots,w_k) \rho^{(k)}_{f}(z_1 + w_1,\ldots,z_k + w_k)\ dw_1 \ldots dw_k \\
&\quad 
- \int_{\C^k} G(w_1,\ldots,w_k) \rho^{(k)}_{\tilde f}(z_1 + w_1,\ldots,z_k + w_k)\ dw_1 \ldots dw_k  \Big| \le \tilde C M n^{-c_0/4},
\end{align*}
where $\tilde C$ depends only on the quantities $C, r_0, c_0, k, a_0$.
\end{theorem}

We will prove this theorem in Section \ref{tmt-sec}.

\begin{remark}  
In the applications in this paper, we will always take $a_0=3$, because our statistics will ultimately only depend on the first two moments of the atom distribution, and this can be exploited by Taylor expansions with a third order error.  In applications to random matrices such as \cite{tv-iid}, it is more convenient to take $a_0=5$, because random matrix statistics may be sensitive to the first four moments of the atom distribution (thanks to the \emph{Four Moment Theorem}), which require Taylor expansions with a fifth order error to exploit.
\end{remark} 

\begin{remark} One can view the above result as a local version of the replacement principle in \cite[Theorem 2.1]{TVcir}, which assumed a much weaker non-clustering bound (which, in the context of characteristic polynomials of random matrices, was formulated as a Frobenius norm bound on the relevant random matrices) and which assumed asymptotic comparability of $\frac{1}{n} \log |f(z)|$ and $\frac{1}{n} \log |\tilde f(z)|$ for each complex number $z$, rather than local comparability (the relationship between the two is roughly analogous to the relationship between the law of large numbers and the central limit theorem), but only gave conclusions about the global distribution of zeroes, rather than the local correlation functions.  Versions of this latter principle were used in the recent work of \cite{kab} on global universality for random polynomials.
\end{remark}

\begin{remark}
The theorem requires some smoothness bounds \eqref{gsmooth} on the test function $G$, but if one is willing to replace the quantitative bound
$\tilde C M n^{-c_0/4}$ in the conclusion of the theorem by weaker upper and lower bounds with error terms that go to zero as $n \to \infty$, one can extend the result to functions $G$ that are merely assumed to be continuous rather than smooth, by using tools such as the Stone-Weierstrass theorem to approximate continuous $G$ above and below by smooth $G$; we omit the details.
\end{remark}

\begin{remark}\label{land} Theorem \ref{tmt} is adapted to the situation in which the mean spacing between zeroes is expected to be comparable to $1$ (so that the correlation functions $\rho^{(k)}$ are also expected to have average magnitude comparable to $1$).  In practice, we may employ a rescaling in order to allow Theorem \ref{tmt} to meaningfully apply to settings in which the mean spacing is at some other scale (e.g. $1/\sqrt{n}$ or $1/n$).  One could also develop more general versions of Theorem \ref{tmt} in which the mean spacing near each reference point $z_j$ varies with $j$, but we will not detail such generalizations here in order to simplify the exposition.
\end{remark}

\section{Replacement principle: real case} 

Now we give the analogue of Theorem \ref{replace} in the case of polynomials with real coefficients, which is slightly more complicated and has slightly worse constants, but is otherwise very similar to the complex replacement principle.

\begin{theorem}[Replacement principle, real case]\label{replace-real}  Let $C, r_0 \ge    1 \ge c_0 > 0$ be real constants, and $a_0 \geq 1$ and $k,l \geq 0$ be integer constants with $k+l > 0$, and set
\begin{equation}\label{adef-real}
 A := \frac{200 (k+l)^2 (a_0+2)}{c_0}.
\end{equation}
Let $n \geq 1$ be a natural number, let $f=f_n, \tilde f=\tilde f_n$ be random polynomials of degree at most $n$ with real coefficients (not necessarily of the form in Definition \ref{Rand-def}) and let $x_1,\ldots,x_k \in \R$ and $z_1,\ldots,z_l \in \C$ be numbers that are allowed to depend on $n$.  Assume the following axioms.

\begin{itemize}

\item[(i)] (Non-degeneracy)  With probability at least $1-Cn^{-A}$, $f$ is not identically zero, and similarly for $\tilde f$.

\item[(ii)] (Non-clustering property)  For  $r \geq 1$, we have $N_{B(x_i,r)}(f), N_{B(z_j,r)}(f) \leq C n^{1/A} r^2$ for $1 \leq i \leq k, 1 \leq j \leq l$ with probability at least $1- Cn^{-A}$, and similarly for $\tilde f_n$.

\item[(iii)] (Comparability of log-magnitudes) Given any $1 \leq k' \leq n^{c_0}$, any complex numbers 
$$z'_1,\ldots,z'_{k'} \in \bigcup_{i=1}^k B(x_i,100r_0) \cup \bigcup_{j=1}^l B(z_j,100r_0),$$ 
and any smooth function $F: \C^{k'}\to \C$ obeying the derivative bounds
$$ |\nabla^a F(w)| \leq n^{c_0}$$
for all $0\leq a \leq a_0$, we have 
$$\Big|  \E \Big( F(\log |f(z'_1)|, \ldots, \log |f(z'_{k'})|) - F(\log |\tilde f(z'_1)|, \ldots, \log |\tilde f(z'_{k'})|)  \Big) \Big|  \le C n^{-c_0}.$$

\item[(iv)] (Weak level repulsion)  For  $x,y$ real, $z$ complex in the region
$$ \bigcup_{i=1}^k B(x_i,100r_0) \cup \bigcup_{j=1}^l B(z_j,100r_0)$$
with $|x-y|, |\Im(z)| \leq 1/C$, we have the pointwise bounds
\begin{equation}\label{rho-20}
\rho^{(2,0)}_{\tilde f}(x,y)  \le C
\end{equation}
and 
\begin{equation}\label{rho-01}
\rho^{(0,1)}_{\tilde f}(z)  \le C.
\end{equation}
\end{itemize}

Let $G: \R^k \times \C^l \to \C$ be a smooth function supported on $[-r_0,r_0]^k \times B(0,r_0)^l$ that obeys the bounds
\begin{equation}\label{gsmooth-2}
|\nabla^a G(w)| \leq M
\end{equation}
for all $0 \leq a \leq a_0 + 2(k+l) + 1$, all $w \in \R^k \times \C^l$, and some $M>0$.  Then
\begin{align*}
& \Big| \int_{\R^k} \int_{\C^l} G(y_1,\ldots,y_k,w_1,\ldots,w_l) \\
&\quad\quad \rho^{(k,l)}_{f}(x_1+y_1,\ldots,x_k+y_k,z_1 + w_1,\ldots,z_l + w_l)\ dw_1 \ldots dw_l dy_1 \ldots dy_l \\
&\quad - 
\int_{\R^k} \int_{\C^l} G(y_1,\ldots,y_k,w_1,\ldots,w_l) \\
&\quad\quad\quad \rho^{(k,l)}_{\tilde f}(x_1+y_1,\ldots,x_k+y_k,z_1 + w_1,\ldots,z_l + w_l)\ dw_1 \ldots dw_l dy_1 \ldots dy_l \Big|  \\
&\le \tilde C M n^{-\frac{c_0}{200(a_0+2)(k+l)}}
\end{align*}
where $\tilde C$ depends only on $C, r_0, c_0, k, l, r_0, a_0$.
\end{theorem}

We prove this theorem in Section \ref{real-sec}.

\begin{remark}
Notice that we require the weak repulsion \eqref{rho-20}, \eqref{rho-01} just for $\tilde f$ and not for $f$. In practice, we can  often choose $\tilde f$ to have gaussian coefficients and verify this axiom by a direct (although not entirely trivial) computation (using tools such as the Kac-Rice formula), while $f$ is permitted to have a discrete distribution (which would be very unlikely to obey \eqref{rho-20}, \eqref{rho-01} in a pointwise sense).  In our applications one can usually establish a stronger level repulsion bound than \eqref{rho-20}, \eqref{rho-01}, namely a bound which decays linearly in $|x-y|$ or $|\Im z|$, but we will not need this stronger bound here.  The reader may notice the difference in the exponent in the final bound, compared to the complex case ($\frac{c_0}{200(a_0+2)(k+l)}$ instead of $\frac{c_0}{4}$). This is due to the fact that we will need to apply the 
result in the complex case for a function $G$ with derivatives that can be polynomially large.   We make no attempt to 
optimize these constants whatsoever. 
\end{remark}

\section{Guaranteeing  the assumptions in the replacement principle}  \label{section:guarantee} 

We now present some tools to verify the various axioms in the replacement principle (Theorem \ref{replace} or Theorem \ref{replace-real}).  In order to use asymptotic notation such as $O()$ and $o()$, it will be convenient to phrase these tools in the asymptotic setting in which $n$ is going to infinity (rather than being large and fixed), although one could easily rewrite the propositions below in the non-asymptotic language of a single fixed $n$ if desired.

Let us say that an event depending on $n$ occurs \emph{with overwhelming probability} if it occurs with probability $1-O(n^{-A})$ for any fixed $A$ (independent of $n$), where the implied constant is allowed to depend on $A$.

We now give a general result (which implicitly appears in our previous paper \cite{tv-iid}; see also \cite{byy} for a closely related argument) that guarantees the non-degeneracy and non-clustering axioms (i), (ii) in Theorem \ref{replace} or Theorem \ref{replace-real} if one can obtain concentration result for the log-magnitude $\log |f|$.  

\begin{proposition}[Criterion for non-clustering]\label{critloc}
Let $n \geq 1$ be a natural number, and let $f=f_n$ be a random polynomial of degree at most $n$.  Let $z_0$ be a complex number depending on $n$, and let $0 < c \leq r$ be quantities that are permitted to depend on $n$, with the polynomial size bounds $r \ll n^{O(1)}$ and $c \gg n^{-O(1)}$.  Assume the following axiom:
\begin{itemize}
\item[(i)] (Concentration of logarithm)  For any $z \in B(z_0,r+c) \backslash B(z_0,r-c)$, one has
$$ \log |f(z)| = G(z) + O(n^{o(1)})$$
with overwhelming probability, where $G: \C \to \R$ is a (deterministic) smooth function (that can depend on $n$) obeying the polynomial size bound
\begin{equation}\label{go}
 \sup_{z \in B(z_0,r+c) \backslash B(z_0,r-c)} |G(z)| \ll n^{O(1)},
\end{equation}
and we adopt the convention $\log|0| = -\infty$.
\end{itemize}
Then one has with overwhelming probability that $f$ is non-vanishing and obeys the bound
\begin{equation}\label{nub}
 N_{B(z_0,r)}(f) =\frac{1}{2\pi} \int_{B(z_0,r)} \Delta G(z)\ dz + O( n^{o(1)} c^{-1} r ) + O\left( \int_{B(z_0,r+c) \backslash B(z_0,r-c)} |\Delta G(z)|\ dz \right).
 \end{equation}
Furthermore, the implied constants in the conclusions depend only on the implied constants in the hypotheses.
\end{proposition}

This proposition can be viewed as a variant of the classical Jensen formula linking the zeroes of a holomorphic function to a certain integral of the log-magnitude of this function.  We will prove it in Section \ref{sample-sec}.  To use this proposition, we now specialize to the case of polynomials $f = f_\xi$ of the form of Definition \ref{Rand-def}.  We will normalize the atom distribution $\xi$ to have unit variance.  A short computation then reveals that for any complex number $z$, the random variable $f(z)$ has mean zero and variance
\begin{equation}\label{V-def}
 V(z) := \E (f(z) \bar f(z) )=  \sum_{i=0}^n |c_i|^2 |z|^{2i}.
\end{equation}
Note that this quantity is independent of the atom distribution $\xi$ (once it has been normalized as above).  It is then natural to expect the concentration result
$$ \log |f(z)| = \frac{1}{2} \log V(z) + O(n^{o(1)})$$
with overwhelming probability, which would give the hypothesis of Proposition \ref{critloc} with $G(z) := \frac{1}{2} \log V(z)$.  The following proposition makes this prediction rigorous, provided that the coefficients $c_i$ contain a sufficiently long and non-trivial lacunary subsequence:

\begin{lemma} \label{lemma:concentration-general} Let $n \geq 1$, and let $f = f_n$ be a random polynomial of the type in Definition \ref{Rand-def} whose atom distribution $\xi$ has mean zero and variance one; suppose further that we have the moment condition $\E |\xi|^{2+\eps} \leq M$ for some $\eps > 0$ and $M < \infty$. Let $z$ be a complex number (that can depend on $n$), and let $V(z)$ be defined by \eqref{V-def}.  Assume that there are indices $i_1, \dots, i_m \in \{0,\ldots,n\}$ for some $m = \omega(\log n)$ (thus $m \geq C(n) \log n$ for some $C(n)$ that goes to infinity as $n \to \infty$) such that we have the lacunarity property
$$ |c_{i_{j}}  z^{i_j}|  \ge  2 |c_{i_{j+1} } z^{i_{j+1} }| $$
for all $1 \leq j < m$, and the lower bound
$$
 |c_{i_m} z^{i_m}|  \ge   V(z)^{1/2} \exp (- n^{o(1)} ).
$$
 Then with overwhelming probability we have
$$\log |f(z)| = \frac{1}{2}  \log V(z) + O(n^{o(1)} ). $$
The implied constants in the conclusion depend on those in the hypotheses, and also on $\eps$ and $M$, but are otherwise uniform in $\xi$.
\end{lemma} 

We establish this lemma in Section \ref{conc-sec}.

The following simple lemma is useful in proving the existence of  the subsequence $i_j$ in the above lemma.  

\begin{lemma} \label{lemma:length} 
Assume that $b_0 \ge b_1 \ge \dots \ge b_l >0$ and $ b_i /b_{i+1} \le C$ for some $C \geq 2$ then the sequence $b_i$ contains a subsequence $b_{i_1},\ldots,b_{i_m}$ of length $m \gg \log_C b_0/b_l$ that obeys the lacunarity property $b_{i_j} \geq 2 b_{i_{j+1}}$ for all $1 \leq j < m$. 
\end{lemma} 

\begin{proof}  This is immediate from the greedy algorithm.
\end{proof}

\begin{remark} Proposition \ref{critloc} combined with Lemma \ref{lemma:concentration-general} suggests that the first intensity $\rho^{(1)}(z)$ of a random polynomial $f$ of the form in Definition \ref{Rand-def} should be approximately equal to $\frac{1}{4\pi} \Delta \log V(z)$ in some weak sense.  In the case that the atom distribution $\xi$ was complex gaussian, this approximation was in fact shown to be exact in \cite{EK} (see also \cite{sodin}); this can also be derived from the Kac-Rice formula.  The results in \cite{kab} can be viewed as a verification of this approximation $\rho^{(1)}(z) \approx \frac{1}{4\pi} \Delta\log V(z)$ at global scales.
\end{remark}

\subsection{Comparability of log-magnitudes}

Next, we present a  two moment theorem for the log-magnitude, which assures assumption (iii) of the 
replacement principle. 

\begin{theorem}[Two moment theorem for log-magnitude]\label{tmp-general}  Let $\xi, \tilde \xi$ be two complex random variables of mean zero, variance one, which match moments to second order, and which obey the moment bound $\E |\xi|^{2+\eps}, \E |\tilde \xi|^{2+\eps} < M$ for some finite $\eps,M$.  Let $n \geq 1$, and suppose that $f_{n,\xi}, f_{n,\tilde \xi}$ are random polynomials of the form in Definition \ref{Rand-def} with atom distributions $\xi,\tilde \xi$ respectively, and some choices of coefficients $c_0,\ldots,c_n$. Let $k \geq 1$ be a natural number with $k \leq n^{\alpha_0}$ for some $\alpha_0>0$, and let $z_1,\ldots,z_k$ be complex numbers obeying the delocalization bounds
\begin{equation} \label{maxbound} | c_i z_j^i  | \le n^{-\alpha_1} V(z_j)^{1/2}. \end{equation} 
for $0 \leq i \leq n$ and $1 \leq j \leq k$, where $V$ is defined in \eqref{V-def}.  
Let $G: \C^k \to \C$ be a smooth function (possibly depending on $n$) obeying the derivative bounds
$$ |\nabla^a G(w)| \leq n^{ \alpha_0}$$
for all $0\leq a \leq 3$.  Then if $\alpha_0$ is sufficiently small depending only on $\alpha_1,\eps$, we have
\begin{equation}\label{geo}
 \E G(\log |f_{n,\xi}(z_1)|, \ldots, \log |f_{n,\xi}(z_k)|) - G(\log |f_{n,\tilde \xi}(z_1)|, \ldots, \log |f_{n,\tilde \xi}(z_k)|) = O(n^{-\alpha_0})
\end{equation}
where the implied constant depends only on $\alpha_0,\alpha_1,\eps,M$.
\end{theorem}

We prove this theorem in Section \ref{log-sec}. The $2+\eps$ moment bound is needed to obtain a polynomial decay rate $O(n^{-\alpha_0})$ in \eqref{geo}, which in turn is needed in our version of the replacement principle.  It may however be possible through a more careful analysis to obtain local universality results for polynomials that do not obey this bound, at the cost of replacing $O(n^{-\alpha_0})$ type error terms in the final universality bounds with qualitative decay terms $o(1)$.  We will not pursue this issue here.  Note that if $\xi,\tilde \xi$ are both real valued, then the hypothesis of matching moments to second order is automatic since the $\xi,\tilde \xi$ are normalized to have mean zero and variance one. The leaing idea is to  use Lindeberg replacement trick, originated in \cite{Lin} (see also \cite{PR, SC} for more  recent 
discussions). The arguments  we will use  follow the spirit of \cite{TVlocal, tv-iid}, where characteristic polynomials of 
random matrices were considered. 

\subsection{A sufficient condition for the repulsion bounds}  \label{section:repulsion bound} 

Finally, we give a lemma for verifying the repulsion axiom (iv) of the real replacement principle in the case when the atom distribution is gaussian.  We use the usual exterior product $\wedge: \C^{n+1} \times \C^{n+1} \to \bigwedge^2 \C^{n+1}$ on the vector space $\C^{n+1}$, in particular
$$ |v \wedge w| = (\sum_{0 \leq i < j \leq n} |v_i w_j - v_j w_i|^2)^{1/2}$$
for any $v = (v_0,\ldots,v_n)$ and $w = (w_0,\ldots,w_n)$ in $\C^{n+1}$.

\begin{lemma}[Repulsion of zeroes]\label{lemma:RB1}  Let $n \geq 1$, and let $f = f_n$ be a random polynomial of the type in Definition \ref{Rand-def}, with real coefficients $c_0,\ldots,c_n$ and with atom distribution $\xi$ given by the real gaussian $N(0,1)_\R$.  
Let $x$ be a real number, let $C>1$ and $r_0>0$.
Let $R: B(x,r_0) \to \C$ be a  holomorphic function that is nonvanishing in $B(x,r_0)$, and let $v: \C \to \C^{n+1}$ be the vector valued holomorphic function
$$ v(z) := (R(z) c_i z^i)_{i=0}^n.$$
Assume the axiom
\begin{equation}\label{vo}
|v(z)| \leq C
\end{equation}
for all $z \in B(x,r_0)$, as well as the axiom
\begin{equation}\label{vos}
\left|v(x) \wedge \frac{d}{dx} v(x)\right| \geq C^{-1}.
\end{equation}
Then, if $\delta$ is sufficiently small depending on $r_0$ and $C$, one has the real repulsion estimate
\begin{equation}\label{rhofxy}
\rho^{(2,0)}_f(x,x+\delta) =  O(\delta).
\end{equation}
and the complex repulsion estimate
\begin{equation}\label{rhofz}
\rho^{(0,1)}(x+\sqrt{-1}\delta) =O(\delta).
\end{equation}
Here the implied constants are allowed to depend on $C$ and $r_0$.
\end{lemma} 

We prove this lemma in Section \ref{repulse-sec}.  Our main tools will be the Kac-Rice formula from \cite{kac-0}, \cite{kac}, \cite{rice}, the Cauchy integral formula, and certain geometric arguments.  The holomorphic factor $R(z)$ should be viewed as a normalization factor that one is free to choose in order to make the two hypotheses \eqref{vo}, \eqref{vos} of the lemma hold simultaneously.

\section{Universality of the correlation functions of the classical ensembles} 

We now specialize the above results to the classical ensembles mentioned in the introduction, namely the flat, elliptic, hyperbolic, and Kac polynomials.

\subsection{Flat polynomials}

We begin with the case of flat polynomials (or Weyl polynomials), i.e. random polynomials of the form in Definition \ref{Rand-def} with $c_i := \frac{1}{\sqrt {i!} } $.   Under extremely mild assumptions on the atom distribution\footnote{In \cite{kab}, $\xi$ does not even need to have finite mean or variance; the hypothesis $\E \log(1+|\xi|) < \infty$ suffices.  It is unlikely however that weak hypotheses continue to suffice for local universality.  For instance, in \cite{LS1} it was shown that the number of real zeroes of a Kac polynomial changes significantly when one takes $\xi$ to be drawn from the Cauchy distribution rather than from a distribution of finite variance; see also \cite{IZ} for some stronger and more general results in this direction.} $\xi$, it was shown in \cite[Theorem 2.3]{kab} that the zeroes of such polynomials obeyed the circular law, thus for instance for any Jordan-measurable subset $\Omega$ of the complex plane (e.g. a ball or a rectangle), one has
$$ \frac{1}{n} N_\Omega \to \int_\Omega \frac{1}{\pi} 1_{B(0,1)}(z/\sqrt{n})\ dz$$
both in probability and in the almost sure sense as $n \to \infty$ (assuming the atom distribution $\xi$ is independent of $n$).  Thus, in particular, the bulk of the zeroes should lie inside the disk $B(0,\sqrt{n})$ and be uniformly distributed within that disk at global scales (i.e. at scales comparable to $\sqrt{n}$).  As such, we expect the mean eigenvalue spacing to be comparable to $1$.

We now can present our main universality results for flat polynomials at local scales.

\begin{theorem}[Two moment theorem for flat polynomials; complex case]\label{tmt}  Let $k \geq 1$, $\eps > 0$, and $C>0$ be constants.  Let $n$ be a natural number, Let $f_{n,\xi}, f_{n,\tilde \xi}$ be flat polynomials with atom distributions $\xi, \tilde \xi$ being complex random variables of mean zero and variance one, matching moments to second order and also obeying the bounds $\E |\xi|^{2+\eps}, \E |\tilde \xi|^{2+\eps} \leq C$.
Let  $z_1,\ldots,z_k \in \C$ be quantities depending on $n$ with $n^\eps \leq |z_i| \leq \sqrt{n} + C$ for all $i=1,\ldots,k$. 

Let $G: \C^k\to \C$ be a smooth function supported on the polydisc $B(0,C)^k$ that obeys the bounds
$$
|\nabla^a G(w)| \leq C
$$
for all $0 \leq a \leq 5k + 1$ and all $w \in \C^k$.  Then
\begin{align*}
&\Big| \int_{\C^k} G(w_1,\ldots,w_k) \rho^{(k)}_{f_{n,\xi}}(z_1 + w_1,\ldots,z_k + w_k)\ dw_1 \ldots dw_k \\
&\quad 
- \int_{\C^k} G(w_1,\ldots,w_k) \rho^{(k)}_{f_{n,\tilde \xi}}(z_1 + w_1,\ldots,z_k + w_k)\ dw_1 \ldots dw_k  \Big| \le \tilde C n^{-c_0}
\end{align*}
for some $\tilde C$ depending only on $k,\eps,C$, and some $c_0>0$ depending only on $\eps$.
\end{theorem}

Informally, this theorem establishes local universality of the zeroes in the bulk and edge of the spectrum, except when one is near the origin.  Note that we do not expect universality near the origin, since $\P( f_{n,\xi}(0) = 0 ) = \P(\xi=0)$ is clearly not universal in $\xi$; see \cite{bleher} for further discussion of this issue (in the context of elliptic polynomials rather than flat ones).  Away from the disk $B(0,\sqrt{n})$, one expects very few zeroes, which suggests that Theorem \ref{tmt} should also hold in this case, but our methods do not cover this regime.

We also remark that a result similar to Theorem \ref{tmt} has recently been established (by a rather different method) by \cite{ledoan}.  In our language, the results in \cite{ledoan} establish universality for the distribution of the random variable
$$ \sum_{i=1}^n \phi(\zeta_i - z)$$
where $\phi$ is a continuous, compactly supported function independent of $n$,  and $z$ is close to the boundary of the spectrum.  The main idea is to establish a central limit theorem for a normalized partial Taylor series expansion of $f$ around $z$.  Their argument is simpler than the one given here, but does not appear to give a uniform polynomial rate of convergence as in Theorem \ref{tmt} (or Theorem \ref{tmt-real} below), which is needed in some of our applications.

\vskip2mm

In the case when the coefficients are real, we obtain the following two moment theorem for the mixed correlation functions. 

\begin{theorem}[Two moment theorem for flat polynomials; real case]\label{tmt-real}  Let $k,l \geq 0$, $\eps > 0$, and $C>0$ be constants with $k+l > 0$.  Let $n$ be a natural number, Let $f_{n,\xi}, f_{n,\tilde \xi}$ be flat polynomials with atom distributions $\xi, \tilde \xi$ being real random variables of mean zero and variance one obeying the bounds $\E |\xi|^{2+\eps}, \E |\tilde \xi|^{2+\eps} \leq C$.
Let  $x_1,\ldots,x_k \in \R$ and $z_1,\ldots,z_l \in \C$ be quantities depending on $n$ with $n^\eps \leq |x_i|, |z_j| \leq \sqrt{n} + C$ for all $i=1,\ldots,k$, $j=1,\ldots,l$.  
Let $G: \R^k \times \C^l \to \C$ be a smooth function supported on $[-C,C]^k \times B(0,C)^l$ that obeys the bounds
$$
|\nabla^a G(w)| \leq C
$$
for all $0 \leq a \leq 5(k+l) + 1$ and all $w \in \R^k \times \C^l$. Then
\begin{align*}
& \Big| \int_{\R^k} \int_{\C^l} G(y_1,\ldots,y_k,w_1,\ldots,w_l) \\
&\quad\quad \rho^{(k,l)}_{f_{n,\xi}}(x_1+y_1,\ldots,x_k+y_k,z_1 + w_1,\ldots,z_l + w_l)\ dw_1 \ldots dw_l dy_1 \ldots dy_l \\
&\quad - 
\int_{\R^k} \int_{\C^l} G(y_1,\ldots,y_k,w_1,\ldots,w_l) \\
&\quad\quad\quad \rho^{(k,l)}_{f_{n,\tilde \xi}}(x_1+y_1,\ldots,x_k+y_k,z_1 + w_1,\ldots,z_l + w_l)\ dw_1 \ldots dw_l dy_1 \ldots dy_l \Big|  \\
&\le \tilde C n^{-c_0}
\end{align*}
where $\tilde C$ depends only on $C, k, l$, and $c_0>0$ depends only on $k,l$.
\end{theorem}

We prove these theorems as consequences of the previously stated results in Section \ref{section:flat}.  As an application of these results we are able to establish some new results about the number $N_\R$ of real eigenvalues of flat polynomials:

\begin{theorem}[Number of real zeroes of polynomials]\label{tmt-zeroes}  Let $n$ be a natural number, Let $f_{n,\xi}$ be a flat polynomial with atom distributions $\xi$ being a real random variables of mean zero and variance one obeying the bound $\E |\xi|^{2+\eps} \leq C$ for some $C,\eps>0$.  Then one has 
$$\E N_\R = \frac{2}{\pi} \sqrt{n} + O(n^{1/2-c}),$$
where the implied constant in the $O()$ notation depends only on $C,\eps$.  More generally, for any interval $I \subset \R$, one has 
$$ \E N_I = \frac{1}{\pi} |I \cap [-\sqrt{n},\sqrt{n}]| + O(n^{1/2-c}).$$
\end{theorem}

We establish this result in Section \ref{section:flat} also.  With some additional calculation that we will sketch in that section, one can also obtain the bound $\Var N_I = O(n^{1-c})$, which by Chebyshev's inequality then tells us that
$$ N_I = \frac{1}{\pi} |I \cap [-\sqrt{n},\sqrt{n}]| + O(n^{1/2-c})$$
with probability $1-O(n^{-c})$.  Informally, this asserts that a global scales, the real zeroes of a flat real polynomial are asymptotically uniformly distributed in $[-\sqrt{n}, \sqrt{n}]$ with intensity $1/\pi$.  As a  matter of fact, our
local  universality results allow us to consider the number of real zeros in intervals of length $O(1)$. 

\subsection{Elliptic polynomials} 

We turn now to the elliptic polynomials 
 $$f'_{n,\xi} (z)= \sum_{i=0}^n \sqrt{ \binom{n}{i}} \xi_i z^i,$$
where we normalize $\xi$ to have mean zero and variance one.  As shown in \cite{kab}, the majority of the zeroes of this polynomial have norm $O(1)$ asymptotically almost surely. As a matter of fact, the limiting density function is $\frac{n}{\pi} (1+|z|^2)^{-2} $; see \cite{kab}.  In particular, we expect the typical separation between zeroes to be of the order of $1/\sqrt{n}$, in contrast to the flat case.  In order to renormalize the typical separation between zeroes to be comparable to one (which is the scale to which the replacement principle is adapted), we replace the polynomial $f'_{n,\xi} (z)= \sum_{i=0}^n \sqrt{ \binom{n}{i}} \xi_i z^i$ by the \emph{rescaled} version
\begin{equation} \label{newpoly1} 
f_{n,\xi}(z) :=\sum_{i=0}^n \sqrt { \binom{n}{i} n^{-i} } \xi_i z^i. 
\end{equation}

It is clear that if $z$ is a zero of the non-scaled polynomial then $\sqrt z$ is a zero of the rescaled one. In the following theorems (and their proofs), $f_n$ is the rescaled polynomial. 
By the results of \cite{kab}, the limiting density function for the rescaled elliptic functions is given by the formula
\begin{equation} \label{eqn:limit-e} 
\rho_e(z) := \frac{1}{\pi} (1+ |z|^2/n)^{-2}. 
\end{equation} 
This can be compared with the limiting density $\frac{1}{\pi} 1_{B(0,\sqrt{n})}(z)$ for flat polynomials.

We can now give the analogues of Theorems \ref{tmt}, \ref{tmt-real}.  

\begin{theorem}[Two moment theorem for elliptic polynomials; complex case]\label{tmt-elliptic} 
Let $k \geq 1$, $\eps > 0$, and $C>0$ be constants.  Let $n$ be a natural number and $f_{n,\xi}, f_{n,\tilde \xi}$ be rescaled elliptic polynomials with atom distributions $\xi, \tilde \xi$ being complex random variables of mean zero and variance one, matching moments to second order and also obeying the bounds $\E |\xi|^{2+\eps}, \E |\tilde \xi|^{2+\eps} \leq C$.
Let  $z_1,\ldots,z_k \in \C$ be quantities depending on $n$ with $n^\eps \leq |z_i| \leq C\sqrt{n}$ for all $i=1,\ldots,k$. 

Let $G: \C^k\to \C$ be a smooth function supported on the polydisc $B(0,C)^k$ that obeys the bounds
$$
|\nabla^a G(w)| \leq C
$$
for all $0 \leq a \leq 5k + 1$ and all $w \in \C^k$.  Then
\begin{align*}
&\Big| \int_{\C^k} G(w_1,\ldots,w_k) \rho^{(k)}_{f_{n,\xi}}(z_1 + w_1,\ldots,z_k + w_k)\ dw_1 \ldots dw_k \\
&\quad 
- \int_{\C^k} G(w_1,\ldots,w_k) \rho^{(k)}_{f_{n,\tilde \xi}}(z_1 + w_1,\ldots,z_k + w_k)\ dw_1 \ldots dw_k  \Big| \le \tilde C n^{-c_0}
\end{align*}
for some $\tilde C$ depending only on $k,\eps,C$, and some $c_0>0$ depending only on $\eps$.
\end{theorem}

Note that we allow our reference points $z_1,\ldots,z_k$ in the spectrum to have magnitude as large as $C \sqrt{n}$, as compared to the flat case where we can only establish universality up to magnitude $\sqrt{n}+C$.  This reflects the different nature of the spectrum in the elliptic case, which does not have an' edge at $\{|z|=\sqrt{n}\}$ in contrast to the flat case.  Note that the reflected polynomial $\tilde f_{n,\xi}(z) := z^n f_{n,\xi}(1/z)$ has the same distribution as $f_{n,\xi}$, so the law of the zeroes of $f_{n,\xi}$ is invariant with respect to the inversion map $z \mapsto 1/z$.  Because of this, one can invert Theorem \ref{tmt-elliptic} (and Theorem \ref{tmt-elliptic-real} below) to give universality results in the region $\sqrt{n}/C \leq |z| \leq n^{1-\eps}$ as well (albeit with some additional Jacobian factors that are powers of the $|z_i|$).  We omit the details.

\vskip2mm

\begin{theorem}[Two moment theorem for elliptic polynomials; real case]\label{tmt-elliptic-real} 
Let $k,l \geq 0$, $\eps > 0$, and $C>0$ be constants with $k+l > 0$.  Let $n$ be a natural number and  $f_{n,\xi}, f_{n,\tilde \xi}$ be rescaled elliptic polynomials with atom distributions $\xi, \tilde \xi$ being real random variables of mean zero and variance one obeying the bounds $\E |\xi|^{2+\eps}, \E |\tilde \xi|^{2+\eps} \leq C$.
Let  $x_1,\ldots,x_k \in \R$ and $z_1,\ldots,z_l \in \C$ be quantities depending on $n$ with $n^\eps \leq |x_i|, |z_j| \leq C \sqrt{n}$ for all $i=1,\ldots,k$, $j=1,\ldots,l$.  
Let $G: \R^k \times \C^l \to \C$ be a smooth function supported on $[-C,C]^k \times B(0,C)^l$ that obeys the bounds
$$
|\nabla^a G(w)| \leq C
$$
for all $0 \leq a \leq 5(k+l) + 1$ and all $w \in \R^k \times \C^l$. Then
\begin{align*}
& \Big| \int_{\R^k} \int_{\C^l} G(y_1,\ldots,y_k,w_1,\ldots,w_l) \\
&\quad\quad \rho^{(k,l)}_{f_{n,\xi}}(x_1+y_1,\ldots,x_k+y_k,z_1 + w_1,\ldots,z_l + w_l)\ dw_1 \ldots dw_l dy_1 \ldots dy_l \\
&\quad - 
\int_{\R^k} \int_{\C^l} G(y_1,\ldots,y_k,w_1,\ldots,w_l) \\
&\quad\quad\quad \rho^{(k,l)}_{f_{n,\tilde \xi}}(x_1+y_1,\ldots,x_k+y_k,z_1 + w_1,\ldots,z_l + w_l)\ dw_1 \ldots dw_l dy_1 \ldots dy_l \Big|  \\
&\le \tilde C n^{-c_0}
\end{align*}
where $\tilde C$ depends only on $C, k, l$, and $c_0>0$ depends only on $k,l$.
\end{theorem}

Finally, we can give an analogue of Theorem \ref{tmt-zeroes}:

\begin{theorem}[Real zeroes]\label{rz-e}  
Let $f_{n,\xi}$ be a rescalled  elliptic polynomial with  $\xi$ being a real random variable of mean zero and variance one obeying the bound $\E |\xi|^{2+\eps} \leq C$ for some $C,\eps>0$.  Then one has 
$$\E N_\R = \sqrt{n} + O(n^{1/2-c})$$
for some $c>0$ depending only on $\eps$, where the implied constant depends only on $C,\eps$.  More generally, one has
$$ \E N_I =  \int_I \frac{1}{\pi} \frac{dx}{1+x^2/n} + O(n^{1/2-c})$$
for any interval $I \subset \R$.
\end{theorem}

We establish these results in Section \ref{section:ellip}.

\subsection{Kac polynomials} 

We now turn to the situation of Kac polynomials $f(z)= f_{n,\xi}(z) = \sum_{i=0}^n \xi_i z^i$.  In the case that the atom distribution $\xi$ is a complex Gaussian $N(0,1)_\C$, then the first intensity $\rho^{(1)}_n$ can be computed explicitly from either the Kac-Rice formula from \cite{kac-0}, \cite{kac}, \cite{rice} or the formula of \cite{EK} as
$$ \rho^{(1)}_n(z) = \frac{1}{4\pi} \Delta \log \sum_{i=0}^n |z|^{2i}$$
which can be shown to be $(1+o(1)) n^2 F(a) $ if $||z|-1| = \frac{a}{n}$ for  constant $a>0$, with 
$$F(a):= \frac{1 - (a/\sinh a)^2}{4\pi a^2};$$
see  \cite{IZ}.

In particular, this shows that  the zeroes concentrate uniformly around the unit circle.

It turns out that there are some additional technical difficulties when using the methods of this paper to study Kac polynomials instead of elliptic or flat polynomials.  The singular nature of the limit first intensity at the unit circle is the most obvious such difficulty, but a less obvious difficulty is the partial breakdown of concentration of the log-magnitude $\log |f(z)|$ when $z$ is a root of unity.  For instance, consider the log-magnitude 
$$ \log |f_{n,\xi}(1)| = |\sum_{i=0}^n \xi_i|$$
at $1$.  If the atom distribution $\xi$ is Bernoulli  and $n$ is odd, then the RHS equals $0$ with probability $\frac{ \binom{n+1}{(n+1)/2}}{2^{n+1}} = \Theta (n^{-1/2})$. Therefore,  the logarithm  diverges to $-\infty$ with probability  $\Theta (n^{-1/2})$, which is not  strong enough for the purposes of Proposition \ref{critloc}.  In a similar spirit, the log-magnitude $\log |f_{n,\xi}(e^{2\pi \sqrt{-1} a/b})|$ when $a,b$ are coprime integers with $b$ bounded can be shown in the Bernoulli case to diverge to $-\infty$ with probability $\Theta(n^{-b/2})$; we omit the details.  To overcome this difficulty, we 
make use of  recent results in both inverse Littlewood-Offord theory (see \cite{nguyen}) and quantitative versions of Gromov's theorem (see \cite{shalom}). With these tools, we 
are able to  show that the roots of unity are essentially the only new obstruction to this concentration, allowing the rest of the theory to go through without much further modification.  In particular, we can establish the following local universality results for Kac polynomials;

\begin{theorem}[Two moment theorem for Kac polynomials; complex case]\label{tmt-kac} 
Let $k \geq 1$, $\eps > 0$, and $C>0$ be constants.  Let $n$ be a natural number, Let $f_{n,\xi}, f_{n,\tilde \xi}$ be Kac polynomials with atom distributions $\xi, \tilde \xi$ being complex random variables of mean zero and variance one, matching moments to second order and also obeying the bounds $\E |\xi|^{2+\eps}, \E |\tilde \xi|^{2+\eps} \leq C$.  Let $1/n \leq r \leq n^{-\eps}$ be a radius, and let  $z_1,\ldots,z_k \in \C$ be quantities depending on $n$ with 
$$r \leq \frac{1}{n} + ||z_i|-1| \leq 2r$$
for all $i=1,\ldots,k$. 

Let $G: \C^k\to \C$ be a smooth function supported on the polydisc $B(0,10^{-3})^k$ that obeys the bounds
$$
|\nabla^a G(w)| \leq C
$$
for all $0 \leq a \leq 5k + 1$ and all $w \in \C^k$.  Then
\begin{align*}
&\Big| \int_{\C^k} G(w_1,\ldots,w_k) r^{2k} \rho^{(k)}_{f_{n,\xi}}(z_1 + rw_1,\ldots,z_k + rw_k)\ dw_1 \ldots dw_k \\
&\quad 
- \int_{\C^k} G(w_1,\ldots,w_k) r^{2k} \rho^{(k)}_{f_{n,\tilde \xi}}(z_1 + rw_1,\ldots,z_k + rw_k)\ dw_1 \ldots dw_k  \Big| \le \tilde C n^{-c_0}
\end{align*}
for some $\tilde C$ depending only on $k,\eps,C$, and some $c_0>0$ depending only on $\eps$.
\end{theorem}

Note that the reference points $z_1,\ldots,z_k$ are required to remain at essentially the same distance $r$ from the unit circle.  It is possible to use the methods of this paper to obtain more general local universality results when the $z_1,\ldots,z_k$ are at widely differing distances from the unit circle, but this requires the generalization of Theorem \ref{replace} alluded to in Remark \ref{land}, and we omit this generalization in order to simplify the exposition.

As usual, we have an analogue of the above local universality result in the real case:

\begin{theorem}[Two moment theorem for Kac polynomials; real case]\label{tmt-kac-real} 
Let $k,l \geq 0$, $\eps > 0$, and $C>0$ be constants with $k+l > 0$.  Let $n$ be a natural number, Let $f_{n,\xi}, f_{n,\tilde \xi}$ be Kac polynomials with atom distributions $\xi, \tilde \xi$ being real random variables of mean zero and variance one obeying the bounds $\E |\xi|^{2+\eps}, \E |\tilde \xi|^{2+\eps} \leq C$.  Let $1/n \leq r \leq n^{-\eps}$ be a radius, and let  $x_1,\ldots,x_k \in \R$ and $z_1,\ldots,z_l \in \C$ be quantities depending on $n$ with 
$$r \leq \frac{1}{n} + ||x_i|-1|, \frac{1}{n} + ||z_j|-1| \leq 2r$$
for all $i=1,\ldots,k$, $j=1,\ldots,l$.  
Let $G: \R^k \times \C^l \to \C$ be a smooth function supported on $[-10^{-3},10^{-3}]^k \times B(0,10^{-3})^l$ that obeys the bounds
$$
|\nabla^a G(w)| \leq C
$$
for all $0 \leq a \leq 5(k+l) + 1$ and all $w \in \R^k \times \C^l$. Then
\begin{align*}
& \Big| \int_{\R^k} \int_{\C^l} G(y_1,\ldots,y_k,w_1,\ldots,w_l) \\
&\quad\quad r^{2k+l} \rho^{(k,l)}_{f_{n,\xi}}(x_1+ry_1,\ldots,x_k+ry_k,z_1 + rw_1,\ldots,z_l + rw_l)\ dw_1 \ldots dw_l dy_1 \ldots dy_l \\
&\quad - 
\int_{\R^k} \int_{\C^l} G(y_1,\ldots,y_k,w_1,\ldots,w_l) \\
&\quad\quad\quad r^{2k+l} \rho^{(k,l)}_{f_{n,\tilde \xi}}(x_1+ry_1,\ldots,x_k+ry_k,z_1 + rw_1,\ldots,z_l + rw_l)\ dw_1 \ldots dw_l dy_1 \ldots dy_l \Big|  \\
&\le \tilde C n^{-c_0}
\end{align*}
where $\tilde C$ depends only on $C, k, l$, and $c_0>0$ depends only on $k,l$.
\end{theorem}

 Notice that  a rescalling already took place in the conclusion of the theorems, so we do not need to rescale $f$ here. 
We establish these results in Theorem \ref{kac-sec}. As far as real roots are concerned, our results yield statements 
about the distribution of number of real roots in short l intervals (where the expectation of the number of real roots is $\Theta (1)$). 
To our best knowledge, such results have not been obtained anywehre else for general Kac polynomials.  On the other hand,  on the global scale, we do not obtain a  better  estimate than Ibragimov-Maslova bound. 

\subsection{General polynomials} 

Our result applies for general random polynomials of the form 
$f_{n, \xi}  = \sum_{i=0}^n c_i \xi_i z^i$ where the (deterministic) coefficients $c_i$ 
need to satisfy some mild conditions, but otherwise can be farily arbitrary. Thus, one can use our result to
derive information about the zeroes (in particular the real zeroes) of these polynomials. 

As an example, the expectation of the number of real zeroes of  $f_{n,\xi}$ in an interval 
 can be computed  using Kac formula or Edelman-Kostlan formula when $\xi_i$ are iid standard real gaussian. Our universality result (for the real case)
would should that this expectation remains asymptotically the same when  the atom variable $\xi$ is Bernoulli.
 As far as we know, prior to this paper, no general method has been available  to prove such a  result. 
The reader is invited to work out a few examples.

\section{Proof of the replacement principle, complex case}\label{tmt-sec}

In this section we  establish Theorem \ref{replace}. We will use the approach developed in  \cite[\S 6]{tv-iid}.

Fix $k,C,r_0,c_0,a_0$ as in Theorem \ref{replace}; all implied constants in the $O()$ notation will be allowed to depend on these parameters.  Let $A$ be defined by \eqref{adef}, let $n$ be a natural number, and let $z_1,\ldots,z_k$ be complex numbers and $f, \tilde f$ be random polynomials obeying the hypotheses of the theorem.  We may assume that $n$ is sufficiently large depending on the parameters $k, C, r_0, c_0, a_0$, as the claim is trivial otherwise.

By conditioning out the event that $f$ or $\tilde f$ vanish identically (which by the nondegeneracy axiom (i) only occurs with probability $O(n^{-A})$), we may assume that $f$ and $\tilde f$ are non-vanishing almost surely, as this conditioning does not significantly alter the hypotheses (i)-(iii) or conclusion of the theorem (after adjusting $C$ by a multiplicative constant).  This conditioning might destroy any independence properties enjoyed by the coefficients of the $f, \tilde f$, but this will not be an issue as such independence properties are not directly assumed in Theorem \ref{replace}.

The first observation to use Fourier analysis to reduce to proving the following variant of the conclusion of Theorem \ref{replace}: we have the bound
\begin{equation}\label{Exist}
\begin{split}
&\Big| \int_{\C^k} G(w_1,\ldots,w_k) \rho^{(k)}_{f}(z_1 + w_1,\ldots,z_k + w_k)\ dw_1 \ldots dw_k \\
&\quad 
- \int_{\C^k} G(w_1,\ldots,w_k) \rho^{(k)}_{\tilde f}(z_1 + w_1,\ldots,z_k + w_k)\ dw_1 \ldots dw_k  \Big| \ll n^{-c_0/4},
\end{split}
\end{equation}
whenever $G$ is a function of the form
\begin{equation}\label{gform}
 G(w_1,\ldots,w_k) = G_1(w_1) \dots G_k(w_k)
\end{equation}
for some smooth $G_1,\ldots,G_k: \C \to \C$ supported in $B(0,10r_0)$ and such that
\begin{equation}\label{gb}
 |\nabla^a G_j(w)| \ll 1
 \end{equation}
for all $0 \leq a\leq a_0$ and $1\leq j \leq k$.  

Indeed, suppose we had the bound \eqref{Exist}.  Now let $G$ be a function of the form required for Theorem \ref{replace}.  We view $B(0,r_0)$ as a subset of the square $[-1.1 r_0,1.1 r_0]^2$, which in turn we can identify with the torus $(\R/2.2 r_0\Z)^2$.  Thus $G$ can be viewed as a smooth function on the torus $(\R/(2.2 r_0)\Z)^{2k}$.  We can then expand $G$ as a Fourier series
$$ G(w) = \sum_{b,c \in \Z^k} g_{b,c} e^{2\pi \sqrt{-1} (b \cdot \Re(w) + c \cdot \Im(w)) / (2.2 r_0)}$$
in $[-1.1 r_0,1.1 r_0]^2$,  where the Fourier coefficients $g_{b,c}$ are given by the formula
$$ g_{b,c} := (2.2 r_0)^{-2k} \int_{B(0,r_0)^k} e^{2\pi \sqrt{-1} (b \cdot \Re(w) + c \cdot \Im(w)) / (r_0/4)} G(w)\ dw.$$
Let $\eta: \R \to \R$ be a function supported on $[-1.1 r_0,1.1 r_0]$ that equals one on $[-r_0,r_0]$.  We can then write
$$ G(w) = \sum_{b, c \in \Z^k} G_{b,c}(w)$$
for all $w \in \C^k$, where
$$ G_{b,c}(w) := g_{b,c} \prod_{i=1}^k \psi_{b,c,i}(w_i)$$
and
$$ \psi_{b,c,i}(w_i) := \prod_{i=1}^k e^{2\pi \sqrt{-1} (b_i \Re(w_i) + c_i \Im(w_i)) / (2.2 r_0)} \eta( \Re(w_i) ) \eta(\Im(w_i) ).$$
Observe that $\psi_{b,c,i}$ is supported on $B(0,10r_0)$ and that
$$ |\nabla^a G_{b,c}(w)| \ll (1 + |b| + |c|)^{a_0} |g_{b,c}|$$
for all $w \in \C^k$ and $0 \leq a \leq a_0$.  From \eqref{Exist} and the triangle inequality, we conclude that
\begin{align*}
&\Big| \int_{\C^k} G(w_1,\ldots,w_k) \rho^{(k)}_{f}(z_1 + w_1,\ldots,z_k + w_k)\ dw_1 \ldots dw_k \\
&\quad 
- \int_{\C^k} G(w_1,\ldots,w_k) \rho^{(k)}_{\tilde f}(z_1 + w_1,\ldots,z_k + w_k)\ dw_1 \ldots dw_k  \Big| \\
&\quad \ll n^{-c_0/4} \sum_{b,c \in \Z^k} |g_{b,c}| (1 + |b| + |c|)^{a_0}.
\end{align*}
On the other hand, from \eqref{gsmooth} and integration by parts we have
$$ |g_{b,c}| \ll (1+|b|+|c|)^{-(a_0+2k+1)} M$$
and Theorem \ref{replace} follows.

Now let $G$ be of the form \eqref{gform}.  For any $\alpha>0$, we call a statistic $X(f) \in \C$ of a random polynomial $f$ \emph{$\alpha$-insensitive} if one has
\begin{equation} \label{insensitive} \E  | X(f) - X(\tilde f) | =O(n^{-\alpha}).
\end{equation}  
Thus, for instance, the comparability axiom (iii) tells us that the statistic
$$ F( \log |f(z'_1)|, \ldots, \log |f(z'_{k'})| )$$
is $c_0$-insensitive for all $1 \leq k' \leq n^{c_0}$, $z'_1,\ldots,z'_{k'} \in \bigcup_{i=1}^k B(z_i,20r_0)$, and smooth $F: \C^{k'} \to \C$ obeying the derivative bounds
$$ |\nabla^a F(w)| \leq n^{c_0}$$
for all $w \in C^{k'}$ and $0 \leq a \leq a_0$.

It now suffices to  show that the statistic
\begin{equation}\label{gstat}
 \int_{\C^k} G(w_1,\ldots,w_k) \rho^{(k)}_n(z_1 + w_1,\ldots,z_k + w_k)\ dw_1 \ldots dw_k 
\end{equation}
is $c_0/4$-insensitive.

Let $\zeta_1,\ldots,\zeta_n$ denote the zeroes of $f$.  By \eqref{cord}, the quantity \eqref{gstat} is equal to
\begin{equation}\label{geo2}
\E \sum_{i_1,\ldots,i_k \hbox{ distinct}} G(\zeta_{i_1}-z_1,\ldots,\zeta_{i_k}-z_k)
\end{equation}
By the inclusion-exclusion formula, we may decompose this expression as
\begin{equation}\label{gstat-2}
\E \prod_{j=1}^k X_{z_j,G_j}
\end{equation}
plus a bounded number of lower order terms which are of the form \eqref{gstat-2} for a smaller value of $k$ (and different choices of $G_j$, and a subset of the $\{z_1,\ldots,z_k\}$), where $X_{z_j,G_j} = X_{z_j,G_j}(f)$ denotes the linear statistic
\begin{equation}\label{xdef}
 X_{z_j,G_j} := \sum_{i=1}^n G_j(\zeta_i - z_j).
\end{equation}
For instance, in the $k=2$ case, we have
\begin{align*}
\sum_{1 \le i \neq j \le n } G_1( \zeta_i - z_1) G_2 (\zeta_j -z_2)  &= \left[\sum_{i=1}^n G_1 ( \zeta_i -z_1)\right] \left[\sum_{j=1}^n G_2 (\zeta_j -z_2)\right]\\
&\quad - \sum_{i=1}^n G_1 (\zeta_i-z_1) G_2 (\zeta_i -z_2) \\
&= X_{z_1, G_1} X_{z_2, G_2} - X_{z_1, G_3},
\end{align*}
where
$$G_3 (\zeta):=  G_1 (\zeta) G_2 (\zeta -z_2 +z_1) . $$
Note that $G_3$ obeys similar bounds \eqref{gb} to $G_1,G_2$, though with a slightly different choice of implied constant.  Clearly, similar decompositions are also available for more general values of $k$.

By induction on $k$, it thus suffices to show that the expression \eqref{gstat-2} is $c_0/4$-insensitive.
By the non-clustering hypothesis,  we have $X_{z_j,G_j} = O(n^{1/A} )$ with probability at least $1 - O(n^{-A})$ for each $1 \leq j \leq k$, while from the pointwise bounds \eqref{gb} we have the crude deterministic bound $X_{z_j,G_j} = O(n)$. 
To use these bounds, we introduce a smooth approximation $P(\zeta_1, \dots, \zeta_k)$ of the product $\zeta_1 \dots \zeta_k$ such that 

\begin{itemize}

\item[(i)] $P(\zeta_1, \dots, \zeta_k) = \zeta_1 \dots \zeta_k$ on $B(0, n^{2/A}) ^k$;

\item[(ii)] $P$ is supported on $B(0, 2n^{2/A})^k $; and

\item[(iii)] $P$ obeys the derivative bounds
\begin{equation}\label{gtaylor}
 |\nabla^a P(\zeta_1,\ldots,\zeta_k)|  \ll n^{2k/A} = n^{\frac{c_0}{50a_0}}
 \end{equation}
for all $0 \leq a \leq a_0$ and $\zeta_1,\dots,\zeta_k \in \C$.
\end{itemize} 

For instance, we may define $P$ explicitly by the formula
$$ P(\zeta_1,\ldots,\zeta_k) := \prod_{i=1}^k \zeta_i \phi(|\zeta_i| / n^{2/A} )$$
where $\phi$ is a smooth function supported on $[-2,2]$ that equals $1$ on $[-1,1]$; it is easy to see that this choice of $P$ obeys all the axioms claimed above.

Using the non-clustering axiom (ii), we have 
$$ \prod_{j=1}^k X_{z_j,G_j} = P( X_{z_1,G_1},\ldots, X_{z_k,G_k} ) $$
with probability $1-O(n^{-A})$, and we have the crude deterministic bound
$$ \prod_{j=1}^k X_{z_j,G_j} = P( X_{z_1,G_1},\ldots, X_{z_k,G_k} ) + O( n^k )$$
outside of this event.  Taking expectations, we conclude that
$$
\E \prod_{j=1}^k X_{z_j,G_j} = \E P( X_{z_1,G_1},\ldots, X_{z_k,G_k} ) + O(n^{-A} n^{k} ). 
$$
By \eqref{adef}, we certainly have $n^{-A} n^{k} = O(n^{-c_0/4})$.  It thus suffices to show that the expression
\begin{equation}\label{pstat}
\E P( X_{z_1,G_1},\ldots, X_{z_k,G_k} ) 
\end{equation}
is $c_0/4$-insensitive.

From the fundamental theorem of algebra we have
\begin{equation}\label{fnz}
\log|f(z)| = a_n + \sum_{i: \zeta_i \neq \infty} \log|\zeta_i-z|
\end{equation}
for all $z \in \C$ and some almost surely finite quantity $a_n$ independent of $z$.  (Here we are using the previous reduction that $f$ almost surely does not vanish identically.)  By Green's theorem, \eqref{xdef}, and the smooth compactly supported nature of $G_j$, we conclude that
$$ X_{z_j,G_j} = \int_\C \log|f(z)| H_j(z)\ dz$$
where
$$ H_j(z) := -\frac{1}{2\pi} \Delta G_j(z-z_j)$$
and $\Delta = \frac{\partial^2}{\partial x^2} + \frac{\partial^2}{\partial y^2}$ is the Laplacian on $\C$.
Note that $H_j$ is a bounded smooth function supported on $B(z_j,C)$. 

We now recall a standard sampling lemma from \cite[Lemma 38]{tv-iid}:

\begin{lemma}[Monte Carlo sampling lemma]\label{sampling}  Let $(X,\mu)$ be a probability space, and let $F: X \to \C$ be a square-integrable function.  Let $m \geq 1$, let $x_1,\dots,x_m$ be drawn independently at random from $X$ with distribution $\mu$, and let $S$ be the empirical average
$$ S := \frac{1}{m} (F(x_1) + \dots + F(x_m)).$$
Then $S$ has mean $\int_X F\ d \mu$ and variance $\int_X (F-\int_X F\ d\mu)^2\ d\mu$.  In particular, by Chebyshev's inequality, one has
$$ \P( |S - \int_X F\ d \mu| \geq \lambda ) \leq \frac{1}{m\lambda^2} \int_X (F-\int_X F\ d\mu)^2\ d\mu$$
for any $\lambda > 0$. Equivalently,  for any $\delta > 0$ one has the bound
$$| S - \int_X F\ d \mu | \le \frac{1}{\sqrt{m\delta}} \left(\int_X (F-\int_X F\ d\mu)^2\ d\mu\right)^{1/2}$$
with probability at least $1-\delta$.
\end{lemma}

\begin{proof} The random variables $F(x_i)$ for $i=1,\dots,m$ are jointly independent with mean $\int_X F\ d \mu$ and variance $\frac{1}{m} \int_X (F-\int_X F\ d\mu)^2\ d\mu$.  Averaging these variables, we obtain the claim.
\end{proof}

Ideally, we would like to use the Markov sampling method (Lemma \ref{sampling}) to approximate  $\int_\C \log|f(z)| H_j(z)\ dz$. However, there is an obstacle: as $f$ can have many zeroes far from $z_j$, the error term 
given in Lemma \ref{sampling} can be too large.  To overcome this difficulty, we introduce a method to reduce the variance by exploiting the cancellation properties of the function $H_j$.  Indeed, note that $H_j$ is the Laplacian of a smooth compactly supported function, thus it is orthogonal to any (affine, real-) linear function by integration by parts.  To exploit this, we define a (random) affine real-linear function $L_j: \C \to \C$ by first selecting a reference complex number $w_{j,0}$ drawn uniformly at random from $B(z_j,1)$ (independently of all previous random quantities), and defining $L_j(z)$ to be the random affine real-linear function of $\Re(z), \Im(z)$ that equals $\log |f(z)|$ when $z = w_{j,0}, w_{j,0}+1, w_{j,0}+\sqrt{-1}$.  More explicitly, we have
 \begin{equation}\label{ldef}
\begin{split}
L(z) &:= \log|f(w_{j,0})| \\
&\quad + (\log|f(w_{j,0}+1)| - \log |f(w_{j,0})|) \Re( z - w_{j,0} ) \\
&\quad + (\log|f(w_{j,0}+ \sqrt{-1})| - \log|f(w_{j,0})|) \Im( z - w_{j,0} ).
\end{split} \end{equation} 
 
By the above observation, 
$$ \int_\C L_j(z) H_j(z)\ dz = 0$$
so we can write 
$$ X_{z_j,G_j} = \int_\C K_j(z) \ dz$$
where
$$ K_j(z) := (\log |f(z)|-L_j(z)) H_j(z).$$

The point now is that with high probability, $K_j$ has reasonably small $L^2$ norm:

\begin{lemma} \label{squarebound} For any constant $\delta >0$, we have
\begin{equation}\label{kujak}  \|K_j\|_{L^2} \le n^{\delta} 
\end{equation} 
for all $1\leq j \leq k$ with probability at least $1- O( n^{-\delta + 2/A}) - O(n^{-A+1})$. 
\end{lemma}

\begin{proof}  We follows the proof of \cite[Lemma 39]{tv-iid}.  Notice that by the union bound, it suffices to prove the claim for a single $j$.  We  split $K_j = \sum_{i: \zeta_i \neq \infty} K_{j,i}(z)$, where
$$ K_{j,i}(z) := (\log |z-\zeta_i| - L_{j,i}(z)) H_j(z)$$
and $L_{j,i}: \C \to \C$ is the random linear function that equals $\log |z-\zeta_i|$ when $z = w_{j,0}, w_{j,0}+1, w_{j,0}+\sqrt{-1}$.  By the triangle inequality, we thus have
$$ \|K_j\|_{L^2} \leq \sum_{i: \zeta_i \neq \infty} \|K_{j,i}\|_{L^2}.$$

By the non-clustering  axiom,  for each $z_j$ and $r \ge 1$, one has 
$$ N_{B(z_j,r)} \ll n^{1/A} r^2
$$
with probability at least $1-O(n^{-A})$.  By taking $r$ of the form $r= 2^i, 0 \le i \le \log_2 \sqrt n$ and using the union bound, we can conclude that 

\begin{equation} \label{uniformbound}  N_{B(z_j,r)} \ll n^{1/A} r^2,
\end{equation}  for all $z_j$ and any $r \ge 1$, with probability at least $1 - O(n^{-A+1})$. (Notice that if $r \ge  \sqrt n$ the bound holds trivially as there are at most $n$ zeroes overall.) 

We may now condition on the polynomial $f$ and assume it obeys \eqref{uniformbound}. The only remaining source of randomness are the $w_{j,0}$'s.  In particular, the zeroes $\zeta_i$ are now deterministic.   By Markov's inequality, it suffices to show that 
 \begin{equation} \label{markov} \E  \|K_j\|_{L^2} \ll n^{2/A}. \end{equation}  
 (The expectation is with respect to the $w_{j,0}$, of course.) 
  
Recall  that $H_j$ is supported in $B(z_j, 10r_0)$.  If $1 \leq i \leq n$ is such that $\zeta_i \in B(z_j, 20r_0)$, then a short computation (based on the square-integrability of the logarithm function) shows that the expected value of $\|K_{j,i}\|_{L^2}$ (averaged over all choices of $w_{j,0}$) is $O(1)$. By \eqref{uniformbound}, there are $O(n^{1/A})$ indices $i$ in this case. Thus, the total contribution from this case is $O(n^{1/A})$, which is acceptable.

Now, we consider the more delicate case  when  $\zeta_i \not \in B(z_j, 20r_0)$.  Let us write $z:=  w_{j,0 } + x + \sqrt{-1} y$ and Taylor expand $\log |z- \zeta_i|$
around the point $w_{j,0} -\zeta_i$. Since we only care about $z \in B(z_j, 10r_0)$, we have $|x| , |y| = O(1)$ in this neighborhood and so
\begin{equation}\label{Idef-1}  \begin{split}  
 \log |z-\zeta_i|  &= \log |x   + \sqrt{-1} y + w_{j,0} -\zeta_i | \\
 &\quad = \log |w_{j,0} -\zeta_i| + \frac{ \Re (w_{j,0} -\zeta_i ) } {  |w_{j,0} -\zeta_i |^2 } x +   \frac{ \Im (w_{j,0} -\zeta_i ) } {  |w_{j,0} -\zeta_i |^2 } y \\
 &\quad + O \left(\frac{1}{|w_{j,0} -\zeta_i|^2} \right) .\end{split} \end{equation}

Under the new notation,  we can write $L_{j,i}$ as 
 \begin{equation}\label{ldef-2}
\begin{split}
L_{j,i} (z) &:= \log|w_{j,0} -\zeta_i | \\
&\quad + (\log|w_{j,0}+1 - \zeta_i| - \log |w_{j,0} -\zeta_i|) x \\
&\quad + (\log|w_{j,0}+ \sqrt{-1} -\zeta_i| - \log|w_{j,0}- \zeta_i|) y.
\end{split} \end{equation} 

The point here that this almost cancels out the linear part in \eqref{Idef-1}. Indeed, by considering the Taylor expansion of  
$$\log|w_{j,0}+1 - \zeta_i| - \log |w_{j,0} -\zeta_i| $$ and 
$$\log|w_{j,0}+ \sqrt{-1} -\zeta_i| - \log|w_{j,0}-\zeta_i| $$ we easily see that the difference between $L_{j,i}$ and the linear part of \eqref{Idef-1} is at most 
$O( \frac{1}{|w_{j,0} -\zeta_i|^2 } ) $. Thus, we conclude that the (conditional) expectation of  $\|K_{j,i}\|_{L^2}$ (with respect to the random choice of 
$w_{j,0}$)  is   only $O( \frac{1}{|w_{j,0} -\zeta_i|^2} ) $. As $C \geq 1$, we can replace it by a more convenient bound 
$$O\left( \frac{1}{1+ |w_{j,0} -\zeta_i|^2}  \right) = O\left(\frac{1}{1+|\zeta_i - z_j|^2}\right) , $$ which also holds 
for $\zeta_i $ close to $z_j$.

  Summing over $i$,  we see that the (conditional) expected value of $\|K_j\|_{L^2}$ is at most
$$ O\left( \sum_{i: \zeta_i \neq \infty} \frac{1}{1+|\zeta_i - z_j|^2}\right) .$$

By \eqref{uniformbound}, the number of $\zeta_i$ such that $2^{l}  < |\zeta_i -z_j| \le 2^{l+1} $ is $O(n ^{1/A} 4^l )$, for all $0 \le l \le \log_2 \sqrt n$.  Furthermore, there are 
$O(n^{1/A})$ indices such that $|\zeta_i -z_j| \le 1$, and there are trivially at most $n$ indices for which $|\zeta_i-z_j| \geq \sqrt{n}$. Thus, the above sum is 

$$O\left( \sum_{l=0}^{\log_2 \sqrt n}   \frac{ n^{1/A} 4^l }{ 4 ^l} + \frac{n}{1+(\sqrt{n})^2}\right)  = O (n^{2/A}),$$
proves \eqref{markov} and hence the lemma. 
\end{proof}

  Let $\gamma_0,\gamma_1,\gamma_2$ be positive constants to be determined later (they will end up being constant multiples of $c_0$). Set $m := \lfloor n^{\gamma_0}\rfloor$, and for each $1 \leq j \leq k$ let $w_{j,1},\ldots,w_{j,m}$ be drawn uniformly at random from $B(z_j,10r_0)$, independently of $f$ and the $w_{j,0}$. 
 
 From Lemma \ref{squarebound}, we have that with probability $1- O(n^{-\gamma_1 +2/A} + n^{-A+1})$, we have $\| K_j \|_{L_2}  \le n^{\gamma_1} $. 
 If we condition on this event and use Lemma \ref{sampling} (with respect to the sample points $w_{j,1}, \dots, w_{j,m}$), then we have the estimate
$$X_{z_j,F_j} = \frac{\pi (10r_0)^2}{m} \sum_{i=1}^m K_j(w_{j,i}) +  O\left(  \frac{ 1}  {\sqrt { m n^{-\gamma_2} } }  n^{\gamma_1/2}  \right)$$
with probability $1 -O(n^{-\gamma_2} )$.

Putting all this together, we conclude that 

\begin{equation} \label{approximation} X_{z_j,F_j} = \frac{\pi (10r_0)^2}{m} \sum_{i=1}^m K_j(w_{j,i}) +  O\left( n^{- \frac{\gamma_0 -\gamma_1- \gamma_2}{2} }  \right) \end{equation}
\noindent with probability at least $1- O(n^{-\gamma_1 +2/A} + n^{-\gamma_2} + n^{-A+1}) $.  

Notice that if \eqref{approximation} holds, then by \eqref{gtaylor} we have

$$ P(X_{z_1,F_1},\dots,X_{z_k,F_k}) = P \left( \left(  \frac{\pi (10r_0)^2}{m} \sum_{i=1}^m K_j(w_{j,i}) \right)_{1 \leq j \leq k} \right) + O( n^{- \frac{\gamma_0 -(\gamma_1+ \gamma_2)}{2} } n^{\frac{c_0}{50a_0}}).     $$

\noindent Now we can estimate the expectation of $P$ as

\begin{eqnarray*}  \E P(X_{z_1,F_1},\dots,X_{z_k,F_k}) &=& \E P\left( \left(  \frac{\pi (10r_0)^2}{m} \sum_{i=1}^m K_j(w_{j,i}) \right)_{1 \leq j \leq k} \right) \\
&+&  O( n^{- \frac{\gamma_0 -(\gamma_1+ \gamma_2)}{2} } n^{\frac{c_0}{50a_0}}) \\
&+& O(n^{-\gamma_1 +2/A} +  n^{-\gamma_2} + n^{-A+1} ) n^{\frac{c_0}{50a_0}} . \end{eqnarray*} 

If we set $\gamma_0 =0.99 c_0, \gamma_1= \gamma_2 = 
0.3 c_0$ (say) and use \eqref{adef}, then it is easy to see that the two error terms on the RHS are of size $O(n^{-c_0/4} )$.   Furthermore,
with these choices of $\gamma_0,\gamma_1,\gamma_2$, one sees (using \eqref{gtaylor}) that the statistic
$$ \E P\left( \left(  \frac{\pi (10r_0)^2}{m} \sum_{i=1}^m K_j(w_{j,i}) \right)_{1 \leq j \leq k} \right) $$
obeys the hypotheses required for the comparability axiom (iii) and is thus $c_0$-insensitive, uniformly for all deterministic choices of $w_{j,0} \in B(z_j,1)$ and $w_{j,l} \in B(z_j,C)$; $l=1, \dots, m$. 
It follows that  $\E P(X_{z_1,F_1},\dots,X_{z_k,F_k}) $ is $c_0/4$-insensitive, concluding the proof of the theorem.

\section{Proof of the replacement principle, real case }\label{real-sec}

We now prove Theorem \ref{replace-real}.  
Let $k,l,C,c_0,r_0,a_0$ be as in that theorem; all implied constants in the $O()$ notation will be allowed to depend on these parameters.  Let $A$ be defined by \eqref{adef-real}.  Let $n$ be a natural number, and let $x_1,\ldots,x_k$, $z_1,\ldots,z_l$, and $f=f_n, \tilde f=\tilde f_n$ obeying the hypotheses of the theorem.  We may assume that $n$ is sufficiently large depending on $k,l,C,c_0,r_0,a_0$, as the claim is trivial otherwise.

As in the proof of Theorem \ref{replace}, we may assume that $f_n$ and $\tilde f_n$ are almost surely non-vanishing, and that $m$ and $M$ are equal to $1$.

Write $c_1 := \frac{c_0}{100(a_0+2)(k+l)}$.  By the Fourier-analytic arguments of the previous section, it will suffice to show that the quantity
\begin{equation}\label{insens}
\begin{split}
&\int_{\R^k} \int_{\C^l} G(y_1,\ldots,y_k,w_1,\ldots,w_l) \\
&\quad \rho^{(k,l)}_{f}(x_1+y_1,\ldots,x_k+y_k,z_1 + w_1,\ldots,z_l + w_l)\ dw_1 \ldots dw_l dy_1 \ldots dy_l 
\end{split}
\end{equation}
is $c_1$-insensitive, whenever $G$ takes the form
$$ G(y_1,\ldots,y_k,w_1,\ldots,w_l) = F_1(y_1) \ldots F_k(y_k) G_1(w_1) \ldots G_l(w_l)$$
where $F_i: \R \to \C$ and $G_j: \C \to \C$ are smooth functions supported on $[-10r_0,10r_0]$ and $B(0,10r_0)$ respectively, such that
$$ |\nabla^a F_i(x)|, |\nabla^a G_j(z)| \ll 1$$
for all $1 \leq i \leq k; 1 \leq j \leq l$, $0 \leq a \leq a_0$, and $x \in \R, z \in \C$.

By repeating the inclusion-exclusion arguments in the complex case, by separating the spectrum into contributions from $\R, \C_+, \C_-$ (and increasing $C$ as necessary), it  suffices to show that the quantity
\begin{equation}\label{egf}
 \E 
\left(\prod_{i=1}^{\tilde k} X_{\tilde x_i,F_i,\R}\right)
\left(\prod_{j=1}^{\tilde l} X_{\tilde z_j,G_j,\C_+}\right)
\left(\prod_{j'=1}^{\tilde l'} X_{\tilde z'_{j'},G'_{j'},\C_-}\right)
\end{equation}
is $c_1$-insensitive, where $\tilde k \leq k$ and $\tilde l + \tilde l' \leq l$, $\tilde x_1,\ldots,\tilde x_{\tilde k} \in \{x_1,\ldots,x_k\}$ and $\tilde z_1,\ldots,\tilde z_{\tilde l},\tilde z'_1,\ldots,\tilde z'_{\tilde l'} \in \{z_1,\ldots,z_l\}$, and 
$$ X_{x,F,\R} := \sum_{i: \zeta_i \in \R} F(\zeta_i- x)$$
and
$$ X_{z,G,\C_\pm} := \sum_{i: \zeta_i \in \C_\pm} G(\zeta_i- z),$$
and the $F_i: \R \to \C$, $G_j: \C \to \C$, $G'_{j'}: \C \to \C$ are smooth functions supported on $B(0,10r_0)$ obeying the bounds
$$ |\nabla^a F_i(x)|, |\nabla^a G_j(z)|, |\nabla^a G'_{j'}(z)| \ll 1$$
for all $0 \leq a \leq a_0$, $x \in \R$, $z \in \C$, and $\zeta_i$ enumerates the zeroes of $f$.  

As the zeroes of $f$ are symmetric around the real axis (and $f(\bar{z}) = \overline{f(z)}$), one has
$$ X_{z,G,\C_-} = X_{\overline{z}, \tilde G, \C_+}$$
where $\tilde G(z) := G(\overline{z})$.  Thus we may concatenate the $G_j$ with the $G'_{j'}$, and assume without loss of generality that $\tilde l'=0$, at the cost of placing $\tilde z_1,\ldots,\tilde z_{\tilde l}$ in $\{ z_1,\ldots,z_l,\overline{z_1},\ldots,\overline{z_l}\}$ rather than $\{z_1,\ldots,z_l\}$.  Thus we are now seeking to establish the $c_1$-insensitivity of
\begin{equation}\label{egf-2}
 \E 
(\prod_{i=1}^{\tilde k} X_{\tilde x_i,F_i,\R})
(\prod_{j=1}^{\tilde l} X_{\tilde z_j,G_j,\C_+}).
\end{equation}

On the other hand, by repeating the remainder of the arguments for the complex case with essentially no changes, we can show that the quantity
\begin{equation}\label{goo}
 \E \prod_{p=1}^m X_{z'_p, H_p}
\end{equation}
is $c_0/4$-insensitive for any $m \leq k+l$, any complex numbers $z'_1,\dots,z'_m$ in 
$$\bigcup_{i=1}^k B(x_i,20r_0) \cup \bigcup_{j=1}^l B(z_j,20r_0) \cup B(\overline{z_j},20r_0),$$ 
and any smooth $H_p: \C \to \C$ supported in $B(0,20r_0)$ and obeying the bounds
$$ |\nabla^a H_p(z)| \leq 1$$
for all $0 \leq a \leq a_0$ and $z \in \C$, where
$$ X_{z,H} := \sum_{i: \zeta_i \neq \infty} H(\zeta_i- z).$$
(Here we use the trivial remark that $\log |f(\overline{z})| = \log|f(z)|$, so that one can freely replace $\{z_1,\ldots,z_l\}$ by $\{z_1,\ldots,z_l.\overline{z_1},\ldots,\overline{z_l}\}$ in the comparability axiom (iii).)

We are going to  deduce the $c_1$-insensitivity of \eqref{egf-2} from the $c_0/4$-insensitivity of \eqref{goo}. The main idea is to extend a real function to a complex one without changing the value of the expectation in \eqref{goo} by too much. This will be the place where we make an essential use of the 
weak repulsion axiom (iv). 

Notice that from the non-clustering axiom and \eqref{adef-real} that 

\begin{equation}\label{holla}
\begin{split}
 \E |X_{\tilde x_i,F_i,\R}|^{\tilde k+\tilde l}, \E |X_{\tilde z_j,G_j,\C_+}|^{\tilde k+\tilde l} 
 &\ll n^{ (\tilde k+\tilde l)/A} + n^{-A +\tilde k+\tilde l }\\
 &\ll n^{ (k+l)/A}, 
\end{split}
\end{equation}
for all $1 \leq i \leq \tilde k$ and $1 \leq j \leq\tilde l$.

In the next, and critical, lemma, we use the weak repulsion hypotheses \eqref{rho-20}, \eqref{rho-01} to show that the there are very few complex 
zeroes near the real line. 

\begin{lemma}[Level repulsion]\label{wlr}  Let $\beta$ be an arbitrary small positive constant.  Let $x$ be a real number in the set
$$\bigcup_{i=1}^{\tilde k} B(\tilde x_i,50r_0) \cup \bigcup_{j=1}^{\tilde l} B(\tilde z_j,50r_0).$$
Let $\gamma := n^{-\frac{c_0}{20(a_0+2)}}$. Then we have
\begin{equation}\label{sole}
\P( N_{B(x,10\gamma)} \geq 2 ) \ll \gamma^{5/4}
\end{equation}
for both $f$ and $\tilde f$.
\end{lemma}

The exponent $5/4$ is not optimal here, but any exponent greater than $1$ suffices for our application.

\begin{proof}  Let $H$ be a non-negative bump function supported on $B(x,20\gamma)$ that equals one on $B(x,10\gamma)$.  Observe that $X_{x,H}^2 - X_{x,H^2}$ is always non-negative, and is at least $2$ when $N_{B(x,10\gamma)} \geq 2$.  Thus by Markov's inequality, it suffices to show that
$$ \E X_{x,H}^2 - X_{x,H^2} \ll \gamma^{5/4}$$
for both $f$ and $\tilde f$.
By construction we see that the first $a_0$ derivatives of $H$ and $H^2$ are less than $n^{c_0/8}$, so by Theorem \ref{replace} we have
$$ \E X_{x,H}^2(f) = \E X_{x,H}^2(\tilde f) + O(n^{-c_0/8})$$
and similarly for $X_{x,H^2}$. Since $O(n^{-c_0/8}) = O(\gamma^2)$, we conclude that it will suffice to establish the claim for $\tilde f$:
$$ \E X_{x,H}^2(\tilde f) -X_{x,H^2}(\tilde f) \ll \gamma^{5/4}.$$
Arguing as in the proof of \eqref{holla}, one can establish the crude bound
$$ \E |X_{x,H}^2(\tilde f) -X_{x,H^2}(\tilde f)|^4 \ll  n^{4/A} \ll \gamma^{-1}.$$
Thus by H\"older's inequality, it suffices to show that
$$ \P (X_{x,H}^2(\tilde f) -X_{x,H^2}(\tilde f) \neq 0) \ll \gamma^2.$$
Next, observe that the expression $X_{x,H}^2(\tilde f) -X_{x,H^2}(\tilde f)$ vanishes if $\tilde f$ has at most one zero in $B(x,20\gamma) \cap \R$ and no zeroes in $B(x,20\gamma) \cap \C_+$.  Thus it suffices to show that
\begin{equation}\label{sole-1}
\P ( N_{B(x,20\gamma) \cap \C_+}(\tilde f) \geq 1 ) \ll \gamma^2
\end{equation}
and
\begin{equation}\label{sole-2}
\P( N_{B(x,20\gamma) \cap \R}(\tilde f) \geq 2 ) \ll \gamma^2.
\end{equation}
This will follow from the bounds
$$ \int_{B(x,20\gamma) \cap \C_+} \rho^{(0,1)}_{\tilde f}(z)\ dz \ll \gamma^2$$
and
$$ \int_{B(x,20\gamma) \cap \R} \int_{B(x,20\gamma) \cap \R} \rho^{(2,0)}_{\tilde f}(y,y')\ dy dy' \ll \gamma^2$$
respectively; but these are immediate from \eqref{rho-01}, \eqref{rho-20}.
\end{proof} 

\begin{remark} If one had some additional decay on the right-hand sides of \eqref{rho-01}, \eqref{rho-20} as $|x-y|$ or $\Im z$ went to zero, then one could improve the powers of $\gamma$ in the bound \eqref{sole}.  For instance, the level repulsion bounds provided by Lemma \ref{lemma:RB1} should permit an improvement of essentially one additional factor of $\gamma$. But for the argument here, any bound on this probability which decays as $O(\gamma^c)$ for some $c>1$ will suffice.
\end{remark}

Set 
\begin{equation}\label{gammaw}
\gamma := n^{-\frac{c_0}{20(a_0+2)}}, 
\end{equation}
and for any real number $x$, let $E_{x, \gamma}$ be the event that there are two zeroes $\zeta_i,\zeta_j$ of $f$ in the strip $S_{x,\gamma} := \{ z \in B(x,20r_0): \Im(z) \leq \gamma \}$ with $i \neq j$ such that $|\zeta_i-\zeta_j| \leq 2\gamma$.  Then 
by Lemma \ref{wlr} and a covering argument, we have $\P(E_{x,\gamma}) =O(\gamma^{1/4})$ whenever $x \in \bigcup_{i=1}^{\tilde k} B(\tilde x_i,10r_0) \cup \bigcup_{j=1}^{\tilde l} B(\tilde z_j,10r_0)$. 

From the symmetry of the spectrum, we observe that if $E_{x,\gamma}$ does not hold, then there cannot be any strictly complex zero $\zeta_i$ in the strip $S_{x,\gamma}$, since in that case $\overline{\zeta_i}$ would be distinct zero in the strip at a distance at most $2\gamma$ from $\lambda_i(M_n)$.  In particular, we see that
\begin{equation}\label{sala}
 \P( N_{S_{x, \gamma} \backslash [x-10r_0,x+10r_0]} = 0 ) = 1 - O(\gamma^{1/4})
\end{equation}
whenever $x \in \bigcup_{i=1}^{\tilde k} B(\tilde x_i,C) \cup \bigcup_{j=1}^{\tilde l} B(\tilde z_j,C)$.

We can use \eqref{sala} to simplify the expression \eqref{egf-2} in two ways. First we  may ``thicken'' each factor $X_{\tilde x_i,F_i,\R}$ by replacing it with $X_{\tilde x_i,\tilde F_i}$, where $\tilde F_{i}: \C \to \C$ is a smooth extension of $F_i$ that is supported on the strip $\{ z: |\Im(z)| \leq \gamma \}$, and more specifically
$$ \tilde F_i(z) := F_i(\Re(z)) \varphi(\Im(z)/\gamma)$$
where $\varphi: \R \to \R$ is a smooth function supported on $[-1,1]$ that equals one at the origin.  From \eqref{sala} and the non-clustering axiom (iii), we see that
$$ X_{\tilde x_i,F_i,\R} = X_{\tilde x_i,\tilde F_i} + D_i , $$ where

\begin{itemize} 

\item $D_i=0$ with probability $1- O(\gamma^{1/4})$;

\item $|D_i| \ll n^{1/A}$  with probability $1- n^{-A}$; and

\item $|D_i| \ll n$ with probability $1$.  
\end{itemize} 
In particular, from \eqref{adef-real} we have
\begin{equation} \label{holla1} 
\begin{split}
\E |X_{\tilde x_i,F_i,\R} - X_{\tilde x_i, \tilde F_i}|^{\tilde k+\tilde l} &\ll \gamma n^{(\tilde k+\tilde l)/A} + n^{\tilde k+\tilde l-A} \\
&\ll \gamma^{1/4} n^{(k+l)/A}.
\end{split}
 \end{equation}  
Furthermore, by performing a smooth truncation, we have the derivative bounds $\nabla^a \tilde F_i = O(\gamma^{-a_0})$ for $0 \leq a \leq a_0$.

In a similar vein, we replace each of the $G_j$ in \eqref{egf} with a function $\tilde G_j$ that vanishes on the half-plane $\{ z-z_j: \Im(z) \leq \gamma/2 \}$; more explicitly we set
$$ \tilde G_j(z) := G_j(z) \eta(\Im(z+z_j)/\gamma)$$
where $\eta:\R \to \R$ is a smooth function supported on $[1/2,\infty)$ that equals one on $[1,\infty)$.  Then we have
$$ X_{\tilde z_j, G_j,\C_+} = X_{\tilde z_j,\tilde G_j} + H_j, $$ 
where $H_j$ has properties similar to $D_i$.  In particular we have
\begin{equation} \label{holla3} \E |X_{\tilde z_j,G_j,\C_+} - X_{\tilde z_j, \tilde G_j}|^{\tilde k+\tilde l} \ll \gamma^{1/4} n^{(k+l)/A}.
 \end{equation}  
By telescoping the difference 
$$
(\prod_{i=1}^{\tilde k} X_{\tilde x_i,F_i,\R})
(\prod_{j=1}^{\tilde l} X_{\tilde z_j,G_j,\C_+}) -
(\prod_{i=1}^{\tilde k} X_{\tilde x_i,\tilde F_i})
(\prod_{j=1}^{\tilde l} X_{\tilde z_j,\tilde G_j})$$
and applying H\"older's inequality followed by \eqref{holla}, \eqref{holla1}, \eqref{holla3}, we see that
$$
\E \left|(\prod_{i=1}^{\tilde k} X_{\tilde x_i,F_i,\R})
(\prod_{j=1}^{\tilde l} X_{\tilde z_j,G_j,\C_+}) -
(\prod_{i=1}^{\tilde k} X_{\tilde x_i,\tilde F_i})
(\prod_{j=1}^{\tilde l} X_{\tilde z_j,\tilde G_j})\right| \ll \gamma^{1/4} n^{(k+l)^2/A}.$$
From \eqref{adef-real} and \eqref{gammaw} we see that the right-hand side is $O(n^{-c_1})$.  Thus, to show the $c_1$-insensitivity of \eqref{egf-2}, it suffices to show that the quantity
$$ \E (\prod_{i=1}^{\tilde k} X_{\tilde x_i,\tilde F_i})
(\prod_{j=1}^{\tilde l} X_{\tilde z_j,\tilde G_j})$$
is $c_1$-insensitive.  However, from the $c_0/4$-insensitivity of \eqref{goo} and the derivative bounds on $\tilde F_i, \tilde G_j$ (and homogeneity) we see that this quantity changes by at most
$$ O\left( n^{-c_0/4} (\gamma^{-a_0})^{\tilde k+\tilde l} \right)$$
when one replaces $f$ with $\tilde f$.  From \eqref{gammaw} this quantity is $O(n^{-c_1})$, and the claim follows.


\section{Non-clustering via sharp concentration}\label{sample-sec}

In this section we prove Proposition \ref{critloc}.  
Let $n, f_n, z_0, c, r, G$ be as in that proposition.  Since the condition $\log |f(z)| = G(z)+O(n^{o(1)})$ can only hold when $f$ is non-vanishing, we see from the concentration axiom that $f$ is non-vanishing with overwhelming probability.  We now condition to the event that $f$ is non-vanishing, noting that this does not significantly impact the hypothesis or conclusion of the proposition, and so we assume henceforth that $f$ is almost surely non-vanishing.

We first prove the upper bound
\begin{equation}\label{upp}
 N_{B(z_0,r)}(f) \leq \frac{1}{2\pi} \int_{B(z_0,r)} \Delta G(z)\ dz + O( n^{o(1)} c^{-1} r ) + O\left( \int_{B(z_0,r+c) \backslash B(z_0,r-c)} |\Delta G(z)|\ dz \right)
\end{equation}
with overwhelming probability, and then explain how to modify the argument to obtain the matching lower bound at the end of this section.

Let $\varphi_+$ be a smooth function supported on $B(z_0,r+c)$ which equals $1$ on $B(z_0,r)$, is bounded between $0$ and $1$ on the annulus $B(z_0,r+c) \backslash B(z_0,r)$, and has the second derivative bound $|\nabla^2 \varphi_+| = O(c^{-2})$ on this annulus; such a function is easily constructed since $0 < c \leq r$.  Then
$$ N_{B(z_0,r)}(f) \leq \sum_{i=1}^n \varphi_+( \zeta_i )$$
where $\zeta_1,\ldots,\zeta_n$ are the zeroes of $f$.  Applying Green's theorem as in the proof of Theorem \ref{replace}, we have the identity
$$ \sum_{i=1}^n \varphi_+( \zeta_i ) = \frac{1}{2\pi} \int_\C (\Delta \varphi_+(z)) \log |f_n(z)|\ dz.$$
Meanwhile, from another application of Green's theorem we have
\begin{align*}
 \int_\C (\Delta \varphi_+(z)) G(z)\ dz &=  \int_\C \varphi_+(z) \Delta G(z)\ dz \\
&=  \int_{B(z_0,r)} \Delta G(z)\ dz + O( \int_{B(z_0,r+c) \backslash B(z_0,r)} |\Delta G(z)|\ dz ).
\end{align*}

Set
$ H(z) := |\log |f_n(z)| - G(z)|$;  by the triangle inequality, we thus have
$$ N_{B(z_0,r)}(f) \leq \frac{1}{2\pi} \int_{B(z_0,r)} \Delta G(z)\ dz + O\left( \int_{B(z_0,r+c) \backslash B(z_0,r)} |\Delta G(z)|\ dz \right)
+ O\left(  \int_\C |\Delta \varphi_+(z)| H(z)\ dz \right).$$

Since $\Delta \varphi_+$ is supported on $B(z_0,r+c) \backslash B(z_0,r)$ and has magnitude $O( c^{-2})$, it thus suffices by the triangle inequality to establish the upper bound
$$  \int_{B(z_0,r+c) \backslash B(z_0,r)} H(z)\ dz \ll n^{o(1)} c r.$$

We first observe a crude polynomial bound
\begin{equation}\label{lao}
 \int_{B(z_0,r+c) \backslash B(z_0,r)} |H(z)|^2\ dz  \ll n^{O(1)}
\end{equation}
with overwhelming probability.  To see this, first observe from \eqref{go} and the polynomial size bound on $r$ (and hence on $c$) that
$$  \int_{B(z_0,r+c) \backslash B(z_0,r)} |G(z)|^2\ dz  \ll n^{O(1)}$$
and so it suffices to show that
$$  \int_{B(z_0,r+c) \backslash B(z_0,r)} |\log |f_n(z)||^2\ dz  \ll n^{O(1)}.$$
Let $z_1$ be any element of $B(z_0,r+c) \backslash B(z_0,r)$.  By the hypotheses, we have $\log f_n(z_1) = O(n^{O(1)})$ with overwhelming probability, so it suffices by the triangle inequality again to show that
$$  \int_{B(z_0,r+c) \backslash B(z_0,r)} |\log |f_n(z)| - \log |f_n(z_1)||^2\ dz  \ll n^{O(1)}.$$
But as 
$$\log |f_n(z)| - \log |f_n(z_1)| = \sum_{1 \leq i \leq n: \zeta_i \neq \infty} \log |z-\zeta_i| - \log |z_1 - \zeta_i|,$$
the claim follows from yet another application of the triangle inequality, together with a direct calculation using the square-integrablity the log function $\log |z|$.

Now we apply Lemma \ref{sampling}. 
To use this lemma, let $m := n^A$ for some large fixed $A$ to be chosen later, and let $z_1,\ldots,z_m$ be drawn uniformly at random from the annulus
$B(z_0,r+c) \backslash B(z_0,r)$, independently of each other and of $f$ (and hence of $H$).  After temporarily conditioning $H$ to be fixed, applying Lemma \ref{sampling} to the normalised measure on the annulus $B(z_0,r+c) \backslash B(z_0,r)$, and then undoing the conditioning, we see from \eqref{lao} that one has
$$  \int_{B(z_0,r+c) \backslash B(z_0,r)} H(z)\ dz = |B(z_0,r+c) \backslash B(z_0,r)| \left(\frac{1}{m} \sum_{i=1}^m H(z_i) + O( n^{O(1)-A/4} ) \right)$$
with probability $1-O(n^{-A/2})$.  On the other hand, we have $|B(z_0,r+c) \backslash B(z_0,r)| \ll cr$, and from the hypothesis of concentration of the log-magnitude and the union bound (and after temporarily conditioning the $z_1,\ldots,z_m$ to be fixed) we see that with overwhelming probability, one has $H(z_i) =O(n^{o(1)})$ for all $i=1,\ldots,m$.  We conclude that
$$  \int_{B(z_0,r+c) \backslash B(z_0,r)} H(z)\ dz \ll n^{o(1)} cr + O( n^{O(1)-A/4} )$$
with probability $1-O(n^{-A/2})$, and the claim then follows by diagonalising in $A$ (and using the polynomial size of $c,r$).

This concludes the proof of the upper bound \eqref{upp} with overwhelming probability.  To prove the matching lower bound
$$ N_{B(z_0,r)}(f) \geq \frac{1}{2\pi}  \int_{B(z_0,r)} \Delta G(z)\ dz - O( n^{o(1)} c^{-1} r ) - O\left( \int_{B(z_0,r+c) \backslash B(z_0,r-c)} |\Delta G(z)|\ dz \right),$$
 one performs a similar argument but with $\varphi_+$ replaced by a test function $\varphi_-$ that equals $1$ on $B(z_0,r-c)$ and $0$ outside of $B(z_0,r)$; we leave the details to the interested reader.

\begin{remark}\label{conc-strong}  The above argument also establishes the following variant of Proposition \ref{critloc}; if one is willing to weaken the conclusion of Proposition \ref{critloc} from holding with overwhelming probability to that of holding with probability $1-O(n^{-A})$ for some fixed $A$, then one may also weaken the hypothesis in (i) from holding with overwhelming probability to that of holding with probability $1-O(n^{-B})$ for some $B$ depending on $A$.
\end{remark}

\section{Assumption verification: Proof of Lemma \ref{lemma:concentration-general}} \label{conc-sec}

We now prove Lemma \ref{lemma:concentration-general}.  We first need an elementary lemma of Paley-Zygmund type.

\begin{lemma}[Paley-Zygmund type lemma]\label{lam}  Let $\xi$ be a random variable of mean zero and variance one, and obeying the bound $\E |\xi|^{2+\eps} \leq C$ for some $\eps,C>0$.  Then one can find $A>1$ depending only on $\eps, C$ such that
$$ \P( A^{-1} \leq |\xi-\xi'| \leq A ) \geq A^{-1}$$
where $\xi'$ is an independent copy of $\xi$.
\end{lemma}

\begin{proof}  Let $A$ be sufficiently large depending on $\eps,C,\delta$.  From Chebyshev's inequality, we see that 
$$ \P( |\xi| \leq A/2 ), \P( |\xi'| \leq A/2 ) \geq 1-4/A^2$$
and hence by the triangle inequality
$$ \P( |\xi-\xi'| \leq A ) \geq 1-8/A^2.$$
It thus suffices (for $A$ large enough) to show that
$$ \P( |\xi-\xi'| \geq A^{-1} ) \geq A^{-1}.$$
Suppose this were not the case, then
$$ \P( |\xi-\xi'| \leq A^{-1} ) \geq 1-A^{-1}.$$
By  conditioning on $\xi'$, there thus exists a complex number $z_0$ such that
$$ \P( |\xi-z_0| \leq A^{-1} ) \geq 1-A^{-1}.$$
From Cauchy-Schwarz one has
$$ \E \xi \le  \P( |\xi-z_0| \leq A^{-1} ) (z_0 + O(A^{-1})) + \P( |\xi-z_0| > A^{-1} )^{1/2} (\E |\xi|^2)^{1/2};$$
since $\xi$ has mean zero and variance one, we conclude that
$$ z_0 = O(A^{-1/2})$$
and thus
$$ \P( |\xi| \leq C_0 A^{-1/2} ) \geq 1-A^{-1}$$
for some absolute constant $C_0 > 0$. From H\"older's inequality, we thus have
$$ \E |\xi|^2 \ll A^{-1} + \P( |\xi| \geq C_0 A^{-1/2} )^{\eps/(2+\eps)} (\E |\xi|^{2+\eps})^{2/(2+\eps)};$$
since $\xi$ has variance one and second moment bounded by $C$, we conclude that
$$ 1 \ll A^{-1} + A^{-\eps/(2+\eps)}$$
which leads to a contradiction if $A$ is large enough.
\end{proof}

Using this lemma, we can obtain the following result of ``Littlewood-Offord'' type.

\begin{lemma}[Small ball probability for lacunary steps]\label{laclac}  Let $v_1,\ldots,v_n$ be complex numbers, and suppose there is 
a subsequence $v_{i_1}, v_{i_2},\ldots,v_{i_m}$ with the property that
$$ |v_{i_j}| \geq 2 |v_{i_{j+1}}|$$
for all $j=1,\ldots,m-1$.  
Let $\xi_1,\ldots,\xi_n$ be iid complex random variables whose common distribution $\xi$ has mean zero and variance one, and obeys the bound $\E |\xi|^{2+\eps} \leq C$ for some $\eps,C>0$.
Then one has the non-concentration inequality
$$ \sup_{z \in \C} \P( |\xi_1 v_1 + \dots + \xi_n v_n - z| \leq |v_{i_m}| ) \leq C' \exp( - cm)$$
for some $C', c>0$ depending only on $\eps,C$.
\end{lemma}

\begin{proof}   In order to set up a conditioning argument later, we will introduce some additional sources of randomness.  Let $\xi'_1,\ldots,\xi'_n$ be independent copies of $\xi_1,\ldots,\xi_n$, let $\epsilon_1,\ldots,\epsilon_n \in \{-1,1\}$ be independent Bernoulli variables (independent of both $\xi_i$ and $\xi'_i$, and let $\tilde \xi_i$ be the random variable that equals $\xi_i$ when $\epsilon_i = +1$ and $\xi'_i$ when $\epsilon_i = -1$.  Then $\tilde \xi_1,\ldots,\tilde \xi_n$ has the same joint distribution as $\xi_1,\ldots,\xi_n$, so it suffices to obtain the bound
$$ \sup_{z \in \C} \P( |\tilde \xi_1 v_1 + \dots + \tilde \xi_n v_n - z| \leq |v_{i_m}| ) \leq C' \exp( - cm)$$

Next, let $\xi'$ be an independent copy of $\xi$.  By Lemma \ref{lam} we may find $A > 1$ depending only on $\eps,C$ such that
$$ \P( A^{-1} < |\xi-\xi'| < A ) > A^{-1}.$$
In particular
\begin{equation}\label{cheo}
\P( A^{-1} < |\xi_i-\xi'_i| < A ) > A^{-1}
\end{equation}
for all $1 \leq i \leq n$.

Next, we may refine the sequence $i_1,\ldots,i_m$ to a subsequence $\tilde i_1,\ldots,\tilde i_{\tilde m}$ with
$$ \tilde m \gg m - O(1)$$
and
\begin{equation}\label{aaa0}
 |v_{\tilde i_j}| \geq 4A^2 |v_{\tilde i_{j+1}}|
\end{equation}
and
\begin{equation}\label{bbb}
 |v_{\tilde i_{\tilde m}}| \geq 4A |v_{i_m}|.
\end{equation}

Let $J \subset \{1,\ldots,\tilde m\}$ be the set of indices $j$ for which
\begin{equation}\label{sss}
 A   \ge |\xi_{\tilde i_j} -\xi'_{\tilde i_j} | \ge A^{-1}.
\end{equation}
From \eqref{cheo} and the Chernoff (or Hoeffding) inequality, one has
\begin{equation}\label{j-big}
 |J| \geq c \tilde m
\end{equation}
with probability at least $1-O(\exp(-c'\tilde m)) = 1-O(\exp(-c''m))$ for some quantities $c,c',c''>0$ depending only on $A$, where implied constants in the $O()$ notation may depend on $A$.

We now  on the event that \eqref{j-big} occurs (we even fix all values of $\xi, \xi'$), and then further fix the signs $\epsilon_i$ for $i \not \in J$. After this conditioning,  the only remaining source of randomness comes from the signs $\epsilon_{\tilde i_j}$ for $j \in J$.  We also fix the complex number $z$.  Observe from \eqref{sss} that each reversal of a sign $\epsilon_{\tilde i_j}$ alters the sum $\tilde \xi_1 v_1 + \dots + \tilde \xi_n v_n - z$ by a quantity of magnitude between $\eps |v_{\tilde i_j}|$ and $A |v_{\tilde i_j}|$.  Using \eqref{aaa0}, \eqref{bbb} and the triangle inequality, we conclude that if we modify a non-zero number of signs $\epsilon_{\tilde i_j}$ for $j \in J$, then the above sum is altered by more than $2|v_{i_m}|$.  In particular, of the $2^{|J|}$ possible choices of these signs, at most one of them can lead to the sum having magnitude bounded by $|v_{i_m}|$.  This gives an upper bound of $2^{-|J|} = O( \exp(-c''' m))$ for this event for some $c'''>0$ depending only on $A$, and the claim follows.
\end{proof}

We are now ready to prove Lemma \ref{lemma:concentration-general}.

\begin{proof} [Proof of Lemma \ref{lemma:concentration-general}]
By Markov's inequality (or Chebyshev's inequality), we have with overwhelming probability 
 $$|f(z)| \le V(z)^{1/2} \exp( \log^2 n) = V(z)^{1/2} \exp( n^{o(1)}). $$
Thus with overwhelming probability we have the upper bound

\begin{equation} \label{eqn:upper} \log |f(z)| \le \frac{1}{2} \log V(z) + n^{o(1)}. \end{equation}
Meanwhile, the lower bound
\begin{equation} \label{eqn: lower} \log |f(z)| \ge \frac{1}{2} \log V(z) - n^{o(1)} , \end{equation}
with overwhelming probability is immediate from Lemma \ref{laclac}.
\end{proof}

\section{Assumption verification: Proof of Theorem \ref{tmp-general} }\label{log-sec}.

In this section we establish Theorem \ref{tmp-general}.  We begin by proving a variant of Theorem \ref{tmp-general} in which the logarithms in \eqref{geo} are removed:

\begin{proposition}\label{nolog}  Let $\xi, \tilde \xi, \eps, M, n, f_{n,\xi}, f_{n,\tilde \xi}, c_0,\ldots,c_n,k, z_1,\ldots,z_k, \alpha_0, \alpha_1, V$ be as in Theorem \ref{tmp-general}.  Assume that $\alpha_0$ is sufficiently small depending on $\alpha_1, \eps$, and that $V(z_1),\ldots,V(z_k) > 0$.
Then for smooth function $H: \C^k \to \C$ obeying the derivative bounds
\begin{equation}\label{nabah}
 |\nabla^a H(\zeta_1,\ldots,\zeta_k)| \ll n^{\alpha_0}, 0 \le a \le 3,
\end{equation}
\noindent we have 
\begin{equation}\label{Variant}
\begin{split}
& |\E H(V(z_1) ^{-1/2} f_n(z_1), \ldots, V(z_k) ^{-1/2} f_n(z_k)) \\
&\quad - H(V(z_1) ^{-1/2} \tilde f_n(z_1), \ldots, V(z_k) ^{-1/2} \tilde f_n(z_k))| \ll n^{-\alpha_0 },
\end{split}
\end{equation} 
where the implied constants depend on $\eps, M, \alpha_0, \alpha_1$.
\end{proposition}

\begin{proof}  We use the Lindeberg swapping argument.  Let $\xi_0,\ldots,\xi_n$ be iid copies of $\xi$, and $\tilde \xi_0,\ldots,\tilde \xi_n$ be iid copies of $\tilde \xi$ that are independent of $\xi_0,\ldots,\xi_n$.  We introduce the intermediate polynomials
$$ f_{n,i_0}(z) :=  \sum_{0 \leq i < i_0} c_i \tilde \xi_i z^i + \sum_{i_0 \leq i \leq n} c_i \xi_i  z^i$$
for $0 \leq i_0 \leq n+1 $, and the random variables
$$ Y_{j,i_0} := V(z_j)^{-1/2} f_{n,i_0}(z_j)$$
for $0 \leq i_0 \leq n+1$ and $1 \leq j \leq k$.
We can then write \eqref{Variant} as
$$ |\E H(Y_{1,0}, \ldots, Y_{k,0} ) - H( Y_{1,n+1}, \ldots, Y_{k,n+1})| \ll n^{-\alpha_0 },
$$
and so by telescoping series it will suffice to show that
\begin{equation}\label{dance}
 \sum_{i_0=1}^n |\E H(Y_{1,i_0}, \ldots, Y_{k,i_0} ) - H( Y_{1,i_0+1}, \ldots, Y_{k,i_0+1})| \ll n^{-\alpha_0 }.
\end{equation}

Fix $i_0$.  We can then write
$$ f_{n,i_0}(z) = \hat f_{n,i_0}(z) + c_{i_0} \xi_{i_0} z^{i_0}$$
and
$$ f_{n,i_0+1}(z) = \hat f_{n,i_0}(z) + c_{i_0} \tilde \xi_{i_0} z^{i_0}$$
for any $z$, where 
$$ \hat f_{n, i_0} (z) :=  \sum_{0 \leq i < i_0} c_i \tilde \xi_i z^i + \sum_{i_0 < i \leq n} c_i \xi_i  z^i.$$
In particular we have
$$ Y_{j,i_0} = \tilde Y_{j,i_0} + a_{j,i_0} \xi_{i_0}$$
and
$$ Y_{j,i_0+1} = \tilde Y_{j,i_0} + a_{j,i_0} \tilde \xi_{i_0}$$
where
$$ \tilde Y_{j,i_0} := V(z_j)^{-1/2} \hat f_{n,i_0}(z_j)$$
and
\begin{equation}\label{adefi}
 a_{j,i_0} := \frac{c_{i_0} z_j^{i_0}}{V(z_j)^{1/2}}.
 \end{equation}

Now let us condition all the $\xi_i,\tilde \xi_i$ for $i \neq i_0$ to be fixed, leaving only $\xi_{i_0}$ and $\tilde \xi_{i_0}$ as sources of randomness; in particular, the $\tilde Y_{j,i_0}$ are now deterministic.  We consider the conditional expectation
$$ |\E_{\xi_{i_0}, \tilde \xi_{i_0}} H(Y_{1,i_0}, \ldots, Y_{k,i_0} ) - H( Y_{1,i_0+1}, \ldots, Y_{k,i_0+1})|.$$
We can write
$$
H(Y_{1,i_0}, \ldots, Y_{k,i_0}) = H(\tilde Y_{1,i_0} + a_{1,i_0} \xi_{i_0}, \ldots, \tilde Y_{k,i_0} + a_{k,i_0} \xi_{i_0}).$$
From \eqref{nabah}, the bound $k \leq n^{\alpha_0}$, and Taylor expansion with remainder, we have
$$
H(Y_{1,i_0}, \ldots, Y_{k,i_0}) = H_{0,0} + H_{1,0} \Re \xi_{i_0} + H_{0,1} \Im \xi_{i_0} + O( a_{i_0}^2 n^{4\alpha_0} |\xi_{i_0}|^2 )$$
and
\begin{equation}\label{ttt}
\begin{split}
H(Y_{1,i_0}, \ldots, Y_{k,i_0}) &= H_{0,0} + H_{1,0} \Re \xi_{i_0} + H_{0,1} \Im \xi_{i_0} \\
&\quad + 
H_{2,0} (\Re \xi_{i_0})^2 + H_{1,1} \Re \xi_{i_0} \Im \xi_{i_0} + H_{0,2} (\Im \xi_{i_0})^2 
+ O( a_{i_0}^3 n^{4\alpha_0} |\xi_{i_0}|^3 )
\end{split}
\end{equation}
(say), where
$$ H_{r,s} := \frac{1}{r! s!}\frac{\partial^{r+s}}{(\partial x)^r (\partial y)^s}
H(\tilde Y_{1,i_0} + a_{1,i_0} \xi_{i_0}, \ldots, \tilde Y_{k,i_0} + a_{k,i_0} (x+\sqrt{-1}y))
|_{x=y= 0}$$
and 
\begin{equation}\label{adefa}
a_{i_0} := \left(\sum_{j=1}^k |a_{j,i_0}|^2\right)^{1/2}.
\end{equation}
One can verify that
$$ H_{2,0}, H_{1,1}, H_{0,2} = O( a_{i_0}^2 n^{4\alpha_0} |\xi_{i_0}|^2 )$$
and so the error term in \eqref{ttt} is both $O( a_{i_0} ^2 n^{O(\alpha_0)} |\xi_{i_0}|^2 )$ and $O( a_{i_0} ^3 n^{O(\alpha_0)} |\xi_{i_0}|^3 )$.  Interpolating, we see that
\begin{align*}
H(Y_{1,i_0}, \ldots, Y_{k,i_0}) &= H_{0,0} + H_{1,0} \Re \xi_{i_0} + H_{0,1} \Im \xi_{i_0} + 
H_{2,0} (\Re \xi_{i_0})^2 + H_{1,1} \Re \xi_{i_0} \Im \xi_{i_0} \\
&\quad + H_{0,2} (\Im \xi_{i_0})^2 
+ O( a_{i_0}^{2+\eps} n^{4\alpha_0} |\xi_{i_0}|^{2+\eps} ).
\end{align*}
Similarly for $H(Y_{1,i_0+1}, \ldots, Y_{k,i_0+1})$ and $\tilde \xi_{i_0}$.
Taking expectations in $\xi_{i_0}, \tilde \xi_{i_0}$ and using the bounded moment assumption, and the fact that $\xi, \tilde \xi$ match moments to second order, we conclude that
$$ |\E_{\xi_{i_0}, \tilde \xi_{i_0}} H(Y_{1,i_0}, \ldots, Y_{k,i_0} ) - H( Y_{1,i_0+1}, \ldots, Y_{k,i_0+1})| \ll a_{i_0}^{2+\eps} n^{4\alpha_0}.$$
Integrating out all the other variables, we see that we may bound the left-hand side of \eqref{dance} by
$$ n^{4\alpha_0} \sum_{i_0=1}^n a_{i_0}^{2+\eps}.$$
From \eqref{V-def}, \eqref{adefi}, \eqref{adefa} we have
$$ \sum_{i_0=1}^n a_{i_0}^2 = k \leq n^{\alpha_0}$$
and from \eqref{maxbound} we have
$$ \sup_{ 1\leq i_0 \leq n} a_{i_0} \leq k n^{-\alpha_1}$$
and the claim \eqref{dance} now follows if $\alpha_0$ is sufficiently small depending on $\eps,\alpha_1$.
\end{proof}

Now we can reinstate the logarithms and complete the proof of Theorem \ref{tmp-general}.  Let the notation and hypotheses be as in that theorem.  If one of the $V(z_j)$ vanishes then $f(z_j)$ and $\tilde f(z_j)$ are almost surely zero and the claim is vacuously true thanks to our conventions, so we may assume that $V(z_j)>0$ for all $j$.   

As the conclusions of the theorem are transitive in $\xi,\tilde \xi$, we may assume without loss of generality that one of these distributions, say $\tilde \xi$, has a gaussian distribution (whose real and imaginary part have the same covariance matrix as that of $\xi$, in particular having mean zero and variance one)

Using a smooth partition of unity, we can split $G = G_1 + G_2$, where $G_1$ is supported on those $\zeta_1,\ldots,\zeta_k$ with $\inf_{1 \leq i \leq k} \zeta_i \leq - 50 \alpha_0 \log n$, and $G_2$ is supported on those $\zeta_1,\ldots,\zeta_k$ with $\inf_{1 \leq i \leq k} \zeta_i  \geq -50 \alpha_0 \log n - 1$, and with the bounds
$$ |\nabla^a G_i(x_1,\dots,x_k)| \ll n^{5 \alpha_0}$$
(say) for $0 \leq a \leq 3$, $i=1,2$, and all $x_1,\dots,x_k \in \R$.  (The constants $5,10, 50, 100$ are rather arbitrary and generous.) 

We first show that the contribution coming from $G_1$ is negligible. Indeed, 

$$
 |\E G_1(\log |Y_1|, \ldots, \log |Y_k|)| \leq \E H_1(Y_1,\ldots,Y_k)$$
 for some smooth function $H_1: \C^k \to\R^+$ supported on the region $\{ (\zeta_1,\dots,\zeta_k) \in \C^k: \inf_{1 \leq i \leq k} |\zeta_i| \ll n^{-50 \alpha_0} \}$ obeying the derivative bounds \eqref{nabah} (but with $\alpha_0$ replaced by a constant multiple of itself).  By Proposition \ref{nolog} (and reducing $\alpha_0$ as necessary), we have
$$ \E H_1(Y_1,\ldots,Y_k) \leq \E H_1(\tilde Y_1,\ldots,\tilde Y_k) + O( n^{- \alpha_0} ).$$

But as the $\tilde Y_1,\ldots,\tilde Y_k$ are independent gaussian with mean zero and variance one, the support of $H_1$ has measure 
$O(kn^{-50 \alpha_0}) =O(n^{-49 \alpha_0})$ with respect to the product gaussian measure (regardless of the structure of the covariance matrix).  Furthermore, by assumption  $|H_1| \le n^{10 \alpha_0}$. This implies

$$  |\E G_1(\log |Y_1|, \ldots, \log |Y_k|)| = O(n^{-39 \alpha_0}) = o( n^{-\alpha_0}).$$

With $\tilde Y_i$, we can argue similarly, and without using Proposition \ref{nolog}. 
  
To conclude the proof, it suffices to show that
 $$
 \E G_2(\log |Y_1|, \ldots, \log |Y_k|) - G_2(\log |\tilde Y_1|, \ldots, \log |\tilde Y_k|) = O(n^{-\alpha_0}).
$$
We can rewrite this as
$$
 \E H_2(Y_1, \ldots, Y_k) - H_2(\tilde Y_1, \ldots, \tilde Y_k) = O(n^{-\alpha_0}),
$$
where
$$ H_2(\zeta_1,\ldots,\zeta_k) := G_2(\log |\zeta_1|, \ldots, \log |\zeta_k| ).$$
From the derivative and support hypotheses on $G_2$, we see from the chain rule that $H_2$ obeys the derivative bounds \eqref{nabah} (but with $\alpha_0$ replaced by a constant multiple of itself), and the claim now follows from Proposition \ref{nolog} (again reducing $\alpha_0$ as necessary).

\section{Assumption verification: repulsion bounds} \label{repulse-sec}

In this section we prove Lemma \ref{lemma:RB1}.  Let $n, f, c_0,\ldots,c_n,R,v,x,\delta$ be as in that lemma.  Our primary tool will be the following (vector-valued) version of the well-known Kac-Rice formula from \cite{kac-0}, \cite{kac}, \cite{rice}:

\begin{lemma}[Kac-Rice formula]\label{krf}  Let $f$ be as above.  Let $k,l,n \geq 0$ be integers with $k+2l \leq n$.  Let $x_1,\ldots,x_k \in \R$ be distinct real numbers, and $z_1,\ldots,z_l \in \C_+$ be distinct complex numbers in the upper half-plane.    Then we have
\begin{align*}
& \rho^{(k,l)}(x_1,\ldots,x_k,z_1,\ldots,z_l)  = \\ & p_{\R^k \times \C^l}\left( (f(x_1),\dots,f(x_k),f(z_1),\dots,f(z_l)) = (0,\dots,0)\right) \\
&\quad \times \E\left( |f'(x_1)| \dots |f'(x_k)| |f'(z_1)|^2 \dots |f'(z_l)|^2 | (f(x_1),\dots,f(x_k),f(z_1),\dots,f(z_l)) = (0,\dots,0)\right ).
\end{align*}
where $p_{\R^k \times \C^l}( (f(x_1),\dots,f(x_k),f(z_1),\dots,f(z_l)) = (0,\dots,0))$ denotes the probability density function of the random variable 
$(f(x_1),\dots,f(x_k),f(z_1),\dots,f(z_l))$ (viewed as taking values in $\R^k \times \C^l$) at the origin $(0,\dots,0$).
\end{lemma}

Specialising the above lemma to the cases $(k,l) = (2,0), (0,1)$ and $n \geq 2$, we see that
\begin{equation}\label{rho-2}
\begin{split}
 \rho^{(2,0)}(x,x+\delta) &= p_{\R^2}\left( (f(x), f(x+\delta)) = (0,0) \right)\\
 &\quad \times \E\left( |f'(x)| |f'(x+\delta)| | (f(x),f(x+\delta)) = (0,0) \right) \end{split} 
\end{equation}
and
\begin{equation}\label{rho-1}
\begin{split}
 \rho^{(0,1)}(x+\sqrt{-1}\delta) &= p_{\C}( f(x+\sqrt{-1}\delta) = 0 ) \\
 &\quad \times \E\left( |f'(x+\sqrt{-1}\delta)|^2 | f(x+\sqrt{-1}\delta) = 0 \right). \end{split} 
\end{equation}

Observe that   random variables such as   $$f(x), f(y), \Re f(z), \Im f(z), f'(x), f'(y), \Re f'(z), \Im f'(z)$$ can be written in the form $X \cdot v$, where $X \in \R^{n+1}$ is the random real gaussian vector $X := (\xi_0,\ldots,\xi_n)$, and $v \in \R^{n+1}$ is a deterministic vector depending on $x$, $y$, or $z$.  For computing the quantities in \eqref{rho-2}, \eqref{rho-1}, we observe the following identities:

\begin{lemma}[Gaussian identities]\label{gauss}  Let $1 \leq m \leq n$, and let $v_1,\ldots,v_m$ be linearly independent (deterministic) vectors in $\R^{n+1}$, and let $X \in \R^{n+1}$ be a random real gaussian vector.  Then
$$ p_{\R^m}( (X \cdot v_1, \ldots, X \cdot v_m)  = (0,\ldots,0) ) = (2\pi)^{-m/2} |v_1 \wedge \dots \wedge v_m|^{-1}.$$
Furthermore, if $v$ is another vector in $\R^{n+1}$, then
\begin{equation}\label{xo}
 \E( |X \cdot v|^2 | (X \cdot v_1, \ldots, X \cdot v_m) = (0,\ldots,0) ) = \operatorname{dist}(v, \operatorname{span}(v_1,\ldots,v_m))^2.
\end{equation}
and similarly
\begin{equation}\label{xoe}
 \E( |X \cdot v| | (X \cdot v_1, \ldots, X \cdot v_m) = (0,\ldots,0) ) = \sqrt{\frac{2}{\pi}} \operatorname{dist}(v, \operatorname{span}(v_1,\ldots,v_m)).
\end{equation}
\end{lemma}

\begin{proof}  We can assume that $v $ does not  belong to the span of $v_1, \dots, v_m$, as otherwise both sides of \eqref{xo} and \eqref{xoe} are zero. 
By applying an invertible linear transformation to the $v_1,\ldots,v_m$, we may reduce to the case when the $v_1,\ldots,v_m$ are an orthonormal system.  As the distribution of the gaussian random vector $X$ is rotation invariant, we may then assume without loss of generality that $v_1,\ldots,v_m$ are the first $m$ vectors $e_1,\ldots,e_m$ of the standard basis $e_1,\ldots,e_{n+1}$. 
 Since we may subtract any linear combination of $v_1,\ldots,v_m$ from $v$ without affecting either side of \eqref{xo}, we may assume without loss of generality that $v$ is orthogonal to $e_1,\ldots,e_m$; by rotating and rescaling we may then normalize $v=e_{m+1}$.  The claims then follow by direct computation.
\end{proof}

\subsection{Estimating $\rho^{(2,0)}(x,y)$} 

We apply this lemma to obtain the bound \eqref{rhofxy}.  By Cauchy-Schwarz, we have
\begin{align*}
\E( |f'(x)| |f'(x+\delta)| | (f(x),f(x+\delta)) = (0,0) ) &\leq \E\left( |f'(x)|^2 | (f(x),f(x+\delta)) = (0,0) \right)^{1/2} \\
&\quad \times \E\left( |f'(x+\delta)|^2 | (f(x),f(x+\delta)) = (0,0) \right)^{1/2}
\end{align*}
and hence by \eqref{rho-2} and Lemma \ref{gauss} we have
\begin{equation}\label{vowed}
\rho^{(2,0)}(x,x+\delta) \ll |\bv_x \wedge \bv_{x+\delta}|^{-1} 
\operatorname{dist}(\bw_x, \operatorname{span}(\bv_x,\bv_{x+\delta})) \operatorname{dist}(\bw_{x+\delta}, \operatorname{span}(\bv_x,\bv_{x+\delta}))
\end{equation}
where $\bv_x,\bv_{x+\delta},\bw_x,\bw_{x+\delta}$ are the  vectors
\begin{align}
\bv_x &:= (c_i x^i)_{i=0}^n \label{vx1}\\
\bv_{x+\delta} &:= (c_i (x+\delta)^i)_{i=0}^n \label{vx2}\\
\bw_x &:=  (i c_i x^{i-1} )_{i=0}^n \label{vx3}\\
\bw_{x+\delta} &:= (i c_i (x+\delta)^{i-1})_{i=0}^n.\label{vx4}
\end{align}

Note from the quotient rule and the hypotheses on $v, R$ in Lemma \ref{lemma:RB1} that
\begin{align*}
\bv_x &= \frac{v(x)}{R(x)} \\
\bv_{x+\delta} &= \frac{v(x+\delta)}{R(x+\delta)} \\
\bw_x &= \frac{1}{R(x)} v_x(x) - \frac{R'(x)}{R(x)} \bv_x \\
\bw_{x+\delta} &= \frac{1}{R(x+\delta)} v'(x+\delta) - \frac{R'(x+\delta)}{R(x+\delta)} \bv_{x+\delta}
\end{align*}
where $v'$ is the complex derivative of the holomorphic function $v$, and similarly for $R$.  One can thus write the right-hand side of \eqref{vowed} as
$$ 
|v(x) \wedge v(x+\delta)|^{-1} 
\operatorname{dist}(v'(x), \operatorname{span}(v(x),v(x+\delta))) \operatorname{dist}(v_x(x+\delta), \operatorname{span}(v(x),v(x+\delta))).$$

To obtain the desired bound \eqref{rhofxy}, it will thus suffice to establish the bounds
\begin{align}
|v(x) \wedge v(x+\delta)| &\gg \delta \label{lip-1} \\
\operatorname{dist}(v'(x), \operatorname{span}(v(x),v(x+\delta))) &\ll \delta \label{lip-2} \\
\operatorname{dist}(v'(x+\delta), \operatorname{span}(v(x),v(x+\delta))) &\ll \delta \label{lip-3}.
\end{align}

From \eqref{vo} and the Cauchy integral formula we have the bounds
\begin{equation}\label{vosk}
|\frac{d^k}{dz^k} v(z)| \ll 1
\end{equation}
for all $z \in B(x,\delta)$ and $k=0,1,2$ if $\delta$ is sufficiently small depending on $r_0$ (recall that implied constants are allowed to depend on $r_0,C$).  Using this and Taylor's theorem with remainder, we see that
$$ |v(x+\delta) - v(x) - \delta v'(x)| \ll \delta^2$$
and thus
$$ |v(x) \wedge v(x+\delta) - \delta v(x) \wedge v'(x)| \ll \delta^2$$
which together with \eqref{vos} gives \eqref{lip-1} for $\delta$ sufficiently small.  Also, from \eqref{vosk} and Taylor's theorem with remainder we have
$$ v(x+\delta) = v(x) + \delta v'(x) + O( \delta^2 )$$
and so
$$ v'(x) = \frac{1}{\delta} v(x+\delta) - \frac{1}{\delta} v(x) + O(\delta)$$
and \eqref{lip-2} follows.  A similar argument gives \eqref{lip-3}, and \eqref{rhofxy} follows.

\subsection{Estimating $\rho^{(0,1)}(z)$ } 

We now establish the bound \eqref{rhofz}.

Applying \eqref{rho-1}, splitting into real and imaginary parts, and then using Lemma \ref{gauss}, we see that
\begin{equation}\label{vowed-2}
\begin{split}
 \rho^{(0,1)}(z) &\ll |\Re \bv_{x+\sqrt{-1}\delta} \wedge \Im \bv_{x+\sqrt{-1}\delta}|^{-1} \\
 &\quad  ( \operatorname{dist}( \Re \bw_{x+\sqrt{-1}\delta}, \operatorname{span}(\Re \bv_{x+\sqrt{-1}\delta}, \Im \bv_{x+\sqrt{-1}\delta}) )^2 \\
 &\quad\quad + 
\operatorname{dist}( \Im \bw_{x+\sqrt{-1}\delta}, \operatorname{span}(\Re \bv_{x+\sqrt{-1}\delta}, \Im \bv_{x+\sqrt{-1}\delta}) )^2 )
\end{split}
\end{equation}
where 
\begin{align}
\bv_{x+\sqrt{-1}\delta} &:= (c_i (x+\sqrt{-1}\delta)^i)_{i=0}^n \label{vz1}\\
\bw_{x+\sqrt{-1}\delta} &:=  (i c_i (x+\sqrt{-1}\delta)^{i-1} )_{i=0}^n, \label{vz3}
\end{align}
distances and span are computed over the reals rather than over the complex numbers, and we adopt the convention that the real or imaginary part of a complex vector is computed by taking the real or imaginary part of each of its coefficients separately.  Again, the quotient rule gives

\begin{align*}
\bv_{x+\sqrt{-1}\delta} &= v(x+\sqrt{-1}\delta) / R(x+\sqrt{-1}\delta) \\
\bw_{x+\sqrt{-1}\delta} &= \frac{1}{R(x+\sqrt{-1}\delta)} v'(x+\sqrt{-1}\delta) - \frac{R'(x+\sqrt{-1}\delta)}{R(x+\sqrt{-1}\delta)} \bv_{x+\sqrt{-1}\delta}.
\end{align*}
Thus we may rewrite the right-hand side of \eqref{vowed} as
\begin{align*}
&|\Re v(x+\sqrt{-1}\delta) \wedge \Im v(x+\sqrt{-1}\delta)|^{-1}  \\
& \quad \times 
( \operatorname{dist}( \Re v'(x+\sqrt{-1}\delta), \operatorname{span}(\Re v(x+\sqrt{-1}\delta), \Im v(x+\sqrt{-1}\delta) )^2 \\
&\quad \quad + 
\operatorname{dist}( \Im v'(x+\sqrt{-1}\delta), \operatorname{span}(\Re v(x+\sqrt{-1}\delta), \Im v(x+\sqrt{-1}\delta) )^2 ).
\end{align*}
Since
$$ \Re v(x+\sqrt{-1}\delta) = \frac{v(x+\sqrt{-1}\delta) + v(x-\sqrt{-1}\delta)}{2}$$
and
$$ \Im v(x+\sqrt{-1}\delta) = \frac{v(x+\sqrt{-1}\delta) - v(x-\sqrt{-1}\delta)}{2\sqrt{-1}}$$

It thus suffices to establish the bounds
\begin{align}
|v(x+\sqrt{-1}\delta) \wedge \Im v(x-\sqrt{-1}\delta)| &\gg \delta \label{lipz-1} \\
\operatorname{dist}(v_y(x+\sqrt{-1}\delta), \operatorname{span}(v(x+\sqrt{-1}\delta),v(x-\sqrt{-1}\delta))) &\ll \delta \label{lipz-2} \\
\operatorname{dist}(v_y(x-\sqrt{-1}\delta), \operatorname{span}(v(x+\sqrt{-1}\delta),v(x-\sqrt{-1}\delta))) &\ll \delta \label{lipz-3},
\end{align}
where the notions of distance and span are now over the complex numbers rather than the reals.  But these bounds can be achieved by adapting the proofs of \eqref{lip-1}, \eqref{lip-2}, \eqref{lip-3} (inserting factors of $\sqrt{-1}$ at various stages of the argument; we leave the details to the interested reader.

\section{Universality for flat  polynomials} \label{section:flat}

In this section we establish our main results for flat polynomials, namely Theorems \ref{tmt}, \ref{tmt-real}, \ref{tmt-zeroes}.  This will largely be accomplished by invoking the results obtained in previous sections.

The first basic result we will need is a concentration result for the log-magnitude $\log |f(z)|$ of a flat polynomial:

\begin{lemma}[Concentration for log-magnitude]  Let $C, \eps > 0$ be constants, let $n$ be a natural number, and let $z$ be a complex number with
$$n^{\eps} \le |z|  \le C n^{1/2}.$$
Let $f = f_{n,\xi}$ be a flat polynomial whose atom distribution $\xi$ has mean zero and variance one with $\E |\xi|^{2+\eps} \leq C$.  Then with overwhelming probability, one has
$$ \log |f(z)| = \frac{1}{2} |z|^2 + O(n^{o(1)})$$
when $n^\eps \leq |z| \leq n^{1/2}$, and
$$ \log |f(z)| = n \log |z| - \frac{1}{2} n \log n + \frac{1}{2} n + O(n^{o(1)})$$
when $n^{1/2} \leq |z| \leq C n^{1/2}$.
The implied constants in the asymptotic notation can depend on $C,\eps$.
\end{lemma}

Note that some lower bound on $|z|$ is necessary here, because $\log |f(0)|$ has the distribution of $\log |\xi|$ and this does not need to concentrate at the origin if $\xi$ is discrete (and in particular, $\xi$ could equal zero with non-zero probability).

\begin{proof}  We first compute the quantity $V(z)$ from \eqref{V-def}.  In the flat case we have
$$ V(z) = \sum_{i=0}^n \frac{|z|^{2i}}{i!}.$$
A standard application of Taylor expansion or Stirling approximation (see e.g. \cite[Lemma 64]{tv-iid}) shows that
$$ \log V(z) = |z|^2 + O(n^{o(1)})$$
for $|z| \leq n^{1/2}$ and
$$ \log V(z) = 2 n \log |z| - n \log n + n + O(n^{o(1)})$$
for $|z| \geq n^{1/2}$.  

We now apply Lemma \ref{lemma:concentration-general}.  Comparing that lemma with the current situation, we see that it will suffice to find indices
 indices $i_1, \dots, i_m \in \{0,\ldots,n\}$ for some $m = \omega(\log n)$ such that we have the lacunarity property
$$ |z^{i_j}/\sqrt{i_j!}|  \ge  2 |z^{i_{j+1} }/\sqrt{i_{j+1}!}| $$
for all $1 \leq j < m$, and the lower bound
$$
 |z^{i_m}/\sqrt{i_m}!|  \ge   V(z)^{1/2} \exp (- n^{o(1)} ).
$$

Observe that the sequence $i \mapsto |z^i/\sqrt{i!}|$ is increasing for $i<|z|^2$ and decreasing for $i>|z|^2$, with its largest value being at least $(V(z)/(n+1))^{1/2}$.  Also, the ratio between adjacent elements of this sequence is $O(1)$ when $i$ is comparable to $|z|^2$.  If $n^\eps \leq |z| \leq \sqrt{n}$, then the desired indices $i_0,\ldots,i_m$ can then be obtained by applying Lemma \ref{lemma:length} to the (reversal of the) subsequence of the $|z^i/\sqrt{i!}|$ for which $|z|^2/2 \leq i \leq |z|^2$ (note that the ratio between the largest and smallest elements of this sequence is at least $\exp(c |z|^2) \geq \exp(c n^{2\eps})$ for some $c>0$).  Similarly, if $\sqrt{n} \leq |z| \leq C \sqrt{n}$, the claim follows by applying Lemma \ref{lemma:length} to the (reversal of the) subsequence of the $|z^i/\sqrt{i!}|$ for which $n/2 \leq i \leq n$.
\end{proof}

Note that if we let $G: \C \to \R$ be the function defined by 
$$ G(z) := \frac{1}{2} |z|^2$$
for $|z| \leq \sqrt{n}$ and $$ G(z) := n \log |z| - \frac{1}{2} n \log n + \frac{1}{2} n$$
then a short computation shows that
$$ \Delta G(z) = 2 1_{B(0,\sqrt{n})}(z)$$
in the sense of distributions.  Applying Proposition \ref{critloc} (after performing an infinitesimal regularization of $G$ at the boundary of $B(0,\sqrt{n})$ to make it smooth), we conclude that for any $n^{-C} \leq c \leq r \leq C \sqrt{n}/3$ and $z_0 \in B(0,C \sqrt{n}/3)$ with the property that $B(z_0,r+c) \backslash B(z_0,r-c)$ is disjoint from $B(0,n^{\eps})$, with overwhelming probability $f$ is non-vanishing and obeys the \emph{local circular law}
\begin{equation}\label{aaa}
 N_{B(z_0,r)}(f) = \int_{B(z_0,r)} \frac{1}{\pi} 1_{B(0,\sqrt{n})}(z)\ dz + O( n^{o(1)} c^{-1} r ) + 
O\left( \int_{B(z_0,r+c) \backslash B(z_0,r-c)} 1_{B(0,\sqrt{n})}(z)\ dz \right).
\end{equation}
 This already gives axiom (i) for Theorems \ref{replace}, \ref{replace-real}.  The formula \eqref{aaa} leads to two further consequences of importance to us.  First, for any $z_0 \in B(0, C \sqrt{n}/3)$ and $r \geq 1$, one has with overwhelming probability that
\begin{equation}\label{lc2}
N_{B(z_0,r)} \ll n^{o(1)} r^2.
\end{equation}
Indeed, this claim is trivial for $r \geq \sqrt{n}/3$ (say) from the trivial bound $N_{B(z_0,r)} \leq n$, and for $1 \leq r \leq \sqrt{n}/3$ the claim follows from \eqref{aaa} with $c=1$ and bounding $O( \int_{B(z_0,r+1) \backslash B(z_0,r-1)} 1_{B(0,\sqrt{n})}(z)\ dz )$ by $O(r)$, after enlarging $r$ as necessary in order to avoid the ball $B(0,n^\eps)$.  (This incurs a loss of $n^{\eps+o(1)}$ rather than $n^{o(1)}$, but the gain of $n^{o(1)}$ can then be recovered by diagonalizing in $\eps$.)  In particular, this gives axiom (ii) for Theorems \ref{replace}, \ref{replace-real}.

We will also apply \eqref{aaa} in the case when $z_0=0$, $r=\sqrt{n} + n^{1/2-\eps}$, and $c := n^{1/2-\eps}$, leading to the bound
\begin{equation}\label{lc1}
N_{B(0,\sqrt{n}+n^{1/2-\eps})} = n + O(n^{\eps+o(1)})
\end{equation}
with overwhelming probability.  In other words, with overwhelming probability, all but $O(n^{\eps+o(1)})$ of the zeroes of $f$ lie inside the disk $B(0,\sqrt{n} + n^{1/2-\eps})$.

Next, we establish the comparability of log magnitudes required for axioms (iii) of Theorems \ref{replace}, \ref{replace-real}.

\begin{proposition}[Comparability of log-magnitudes]\label{clm}
Let $C, \eps > 0$ be constants, and let $c_0 > 0$ be sufficiently small depending on $\eps$.  Let $n$ be a natural number, let $1 \leq k \leq n^{c_0}$ be another natural number, and let $z_1,\ldots,z_k$ be complex numbers such that
$$n^{\eps} \le |z|  \le n^{1/2}+C.$$
Let $f_{n,\xi}, f_{n,\tilde \xi}$ be flat polynomials whose atom distributions $\xi, \tilde \xi$ have mean zero and variance one matching moments to second order with $\E |\xi|^{2+\eps}, |\tilde \xi|^{2+\eps} \leq C$.  Let $F: \C^k \to \C$ be a smooth function obeying the bounds
$$ |\nabla^a F(z)| \leq C$$
for all $0 \leq a \leq 3$.  Then, if $c_0$ is sufficiently small, one has
$$ \E \Big( F(\log |f(z'_1)|, \ldots, \log |f(z'_{k'})|) - F(\log |\tilde f(z'_1)|, \ldots, \log |\tilde f(z'_{k'})|)  \Big) = O(n^{-c_0}),$$
where the implied constant in the $O()$ notation depends on $C,\eps,c_0$.
\end{proposition}

\begin{proof}  We may assume that $n$ is sufficiently large depending on $C,\eps,c_0$, as the claim is trivial otherwise.  We may also take $\eps$ to be small (e.g. $\eps < 1/4$).

We use Theorem \ref{tmp-general}.  Inspecting the hypotheses and conclusion of that theorem, we see that it will suffice to verify the delocalization bound
\begin{equation}\label{zji}
 |z_j^i / \sqrt{i!}| \ll n^{-\alpha_1} V(z_j)^{1/2}
 \end{equation}
for all $1 \leq j \leq n$ and some $\alpha_1>0$ that can depend on $\eps$ but is independent of $c_0$.  

Fix $j$.  As observed previously, the sequence $|z_j^i / \sqrt{i!}|$ is increasing for $i<|z_j|^2$ and decreasing for $i>|z_j|^2$.  A routine application of Stirling's formula reveals that the magnitudes $|z_j^i / \sqrt{i!}|$ are comparable to each other for $i = |z_j|^2 + O( |z_j| )$, which in the regime $n^\eps \leq |z_j| \leq \sqrt{n}+C$ occupies at least $\gg n^{\eps}$ of the indices $i$ in $\{0,\ldots,n\}$, including the index $i$ that maximizes $|z_j^i / \sqrt{i!}|$.   The claim \eqref{zji} then follows with $\alpha_1 := \eps/2$.
\end{proof}

As we have now verified all three axioms (i)-(iii) of Theorem \ref{replace}, we obtain Theorem \ref{tmt} as an immediate consequence.  To establish Theorem \ref{tmt-real}, we see from Theorem \ref{replace-real} (and comparing both real atom distributions $\xi, \tilde \xi$ to the real gaussian distribution $N(0,1)_\R$), it suffices to establish axiom (iv) of Theorem \ref{replace-real} in the case that $\tilde \xi$ has the distribution of $N(0,1)_\R$.  More precisely, it suffices to establish the following estimate (which is actually a little stronger than we need):

\begin{proposition}[Level repulsion]\label{level-repuls}
Let $\eps>0$, and let $C > 1$ be a sufficiently large constant. Let $n$ be a natural number, and let $x,y \in \R$ and $z \in \C$ be such that
\begin{equation}\label{xyz}
n^\eps \leq |x|, |y|, |z| \leq \sqrt{n}+C
\end{equation}
and
$$ |x-y|, |\Im z| \leq 1/C.$$
Let $f = f_{n,\xi}$ be a flat polynomial whose atom distribution $\xi$ is drawn from the real gaussian ensemble $N(0,1)_\R$.  Then we have the pointwise bounds
\begin{equation}\label{rho-20a}
\rho^{(2,0)}_{\tilde f}(x,y)  \ll |x-y|
\end{equation}
and 
\begin{equation}\label{rho-01a}
\rho^{(0,1)}_{\tilde f}(z)  \ll |\Im z|,
\end{equation}
where the implied constants depend on $C$.
\end{proposition}

A modification of the calculations below in fact show that the bounds \eqref{rho-20a}, \eqref{rho-01a} continue to hold without the hypothesis \eqref{xyz}, but we will only need the bounds under the hypothesis \eqref{xyz}.

\begin{proof}  We will apply Lemma \ref{lemma:RB1} with $R(z) := e^{-z^2/2}$.  Thus, it suffices to establish the bounds
\begin{align}
|v(z)| &\ll 1 \label{vo-a}\\
\left|v(x) \wedge v'(x)\right| &\gg 1 \label{vos-a} \\
\end{align}
for all $z \in B(x_0,1)$, where
$$ v(z) := e^{-z^2/2} ( z^i / \sqrt{i!} )_{i=0}^n.$$

We begin with the proof of \eqref{vo-a}.  We have
$$
|v(z)|^2 = |e^{-z^2}| \sum_{i=0}^n \frac{|z|^{2i}}{i!}.$$
Note that when $z \in B(x_0,1)$, one has
$$ |e^{-z^2}| \ll e^{-|z|^2}$$
while from Taylor series one has
$$ \sum_{i=0}^n \frac{|z|^{2i}}{i!} \leq \sum_{i=0}^\infty \frac{|z|^{2i}}{i!} = e^{|z|^2}$$
and the claim \eqref{vo-a} follows.

Now we prove \eqref{vos-a}.  Observe that
$$ 
v'(x) = e^{-x^2/2} \left(\frac{i-x^2}{x} \frac{x^i}{\sqrt{i!}}\right)_{i=0}^n 
$$
and so
$$ |v(x) \wedge v'(x)|^2 = e^{-2x^2} \sum_{0 \leq i < j \leq n} \frac{|i-j|^2}{x^2} \frac{x^{2i}}{i!} \frac{x^{2j}}{j!}.$$
From Stirling's approximation we see that $\frac{x^{2i}}{i!}$ is comparable to $x^{-1} e^{x^2}$ when $i = x^2 + O(x)$, and the claim \eqref{vos-a} easily follows (noting that $x^2 \leq n + O(x)$ when $x \leq \sqrt{n}+C$).
\end{proof}

As all of the hypotheses (i)-(iv) of Theorem \ref{replace-real} are obeyed, Theorem \ref{tmt-real} is now established.

Finally, we are able to establish Theorem \ref{tmt-zeroes}.  Let $\eps, n, f = f_{n,\xi}$ be as in that theorem.  From \eqref{lc1} we see that with overwhelming probability, there are at most $O(n^{1/4+o(1)})$ real zeroes outside the interval $[-\sqrt{n}-n^{1/4}, \sqrt{n}+n^{1/4}]$.  Meanwhile, by covering the intervals $[-\sqrt{n}-n^{1/4},-\sqrt{n}+n^{1/4}]$, $[-n^{1/4},n^{1/4}]$, and $[\sqrt{n}-n^{1/4},\sqrt{n}+n^{1/4}]$ by $O(n^{1/4})$ disks of radius $1$ and applying \eqref{lc2} and the union bound, we see that with overwhelming probability, there are also $O(n^{1/4+o(1)})$ zeroes in these intervals.  In view of these facts, it suffices to show that for any interval $I \subset [-\sqrt{n}+n^{1/4}, -n^{1/4}] \cup [n^{1/4}, \sqrt{n}-n^{1/4}]$, one has
$$ \E N_I = \frac{1}{\pi} |I| + O(n^{1/2-c})$$
with probability $1-O(n^{1/2-c})$.

By approximating the indicator function $1_I$ above and below by smooth functions, it will suffice to show that
$$ \E \sum_{1 \leq i \leq n: \zeta_i \in \R} F(\zeta_i) = \frac{1}{\pi} \int_\R F(x)\ dx + O(n^{1/2-c})$$
for any smooth function $F: \R \to \R$ supported in $\{ x \in \R: n^{1/4}/2 \leq |x| \leq \sqrt{n} - n^{1/4}/2\}$ which obeys the derivative bounds
$$ |F^{(a)}(x)| \ll 1$$
for all $0 \leq a \leq 100$ (say).  

Fix $F$.  By \eqref{cord1}, we may rewrite the above claim as the bound
\begin{equation}\label{lemo}
 \int_\R F(x) \rho^{(1,0)}_{f_{n,\xi}}(x)\ dx = \frac{1}{\pi} \int_\R F(x)\ dx + O(n^{1/2-c}) 
 \end{equation}
Now let $f_{n,\tilde \xi}$ be a flat polynomial whose atom distribution $\tilde\xi$ is given by the real gaussian $N(0,1)_\R$.  By smoothly decomposing $F$ into $O(n^{1/2})$ components each supported on an interval $[x-1,x+1]$ and applying Theorem \ref{tmt-real} repeatedly, we see that
$$ \int_\R F(x) \rho^{(1,0)}_{f_{n,\xi}}(x)\ dx = \int_\R F(x) \rho^{(1,0)}_{f_{n,\tilde \xi}}(x)\ dx + O(n^{1/2-c}) $$
Thus, it suffices to establish the analogue of \eqref{lemo} for the real gaussian flat polynomial $f_{n,\tilde \xi}$.  Such a bound can be implicitly extracted from the work of \cite{EK}, but for the sake of completeness we give a proof of this bound here.

Using the Kac-Rice formula (Lemma \ref{krf}) we have
$$  \rho^{(1,0)}_{f_{n,\tilde \xi}}(x) = p_{\R}( f_{n,\tilde \xi}(x) = 0 ) \E\left( |f'_{n,\tilde \xi}(x)| | f_{n,\tilde \xi}(x)=0 \right)$$
for any real $x$.  By Lemma \ref{gauss}, we can write the right-hand side as
$$ \frac{1}{\pi} \frac{|\operatorname{dist}(\bw_x,\bv_x)|}{|\bv_x|}$$
where
$$ \bv_x := \left(\frac{x^i}{\sqrt{i!}}\right)_{i=0}^n$$
and
$$ \bw_x := \left(\frac{i}{x} \frac{x^i}{\sqrt{i!}}\right)_{i=0}^n.$$
for any non-zero $x$.  We can rearrange this as
$$ \frac{1}{\pi} \frac{|\operatorname{dist}(v'(x), v(x))|}{|v(x)|}$$
where
$$ v(x) := e^{-x^2/2} \left(\frac{x^i}{\sqrt{i!}}\right)_{i=0}^n$$
and
$$ v'(x) := e^{-x^2/2} \left(\frac{i-x^2}{x} \frac{x^i}{\sqrt{i!}}\right)_{i=0}^n.$$
We can expand
\begin{align*}
|v(x)|^2 &= e^{-x^2} \sum_{i=0}^n \frac{x^{2i}}{i!} \\
v(x) \cdot v'(x) &= e^{-x^2} \sum_{i=0}^n \frac{i-x^2}{x} \frac{x^{2i}}{i!} \\
|v'(x)|^2 &= e^{-x^2} \sum_{i=0}^n \left(\frac{i-x^2}{x}\right)^2 \frac{x^{2i}}{i!}. 
\end{align*}
By differentiating the identity
$$
e^{x^2} = \sum_{i=0}^\infty \frac{x^{2i}}{i!}$$
twice, we obtain the identities\footnote{These are also the identities for the mean and variance of a Poisson random variable.}
$$
\sum_{i=0}^\infty \frac{i-x^2}{x} \frac{x^{2i}}{i!} = 0$$
and
$$
\sum_{i=0}^\infty \left(\frac{i-x^2}{x}\right)^2 \frac{x^{2i}}{i!} = 1.$$
For $x$ in the support of $F$, we have $n^{\eps}/2 \leq |x| \leq n^{1/2} - n^{1/4}/2$, and if we truncate the above infinite sums to $n$ using Stirling's approximation we conclude that
$$ |v(x)|^2, |v'(x)|^2 = 1 + O(n^{-\eps+o(1)})$$
and
$$ v(x) \cdot v'(x) = O(n^{-\eps+o(1)})$$
so that
$$ \rho^{(1,0)}_{f_{n,\tilde \xi}}(x) = \frac{1}{\pi} + O(n^{-\eps+o(1)}),$$
which gives \eqref{lemo} for $\tilde \xi$ and hence for $\xi$.
This concludes the proof of Theorem \ref{tmt-zeroes}.

\subsection{The variance bound}

As asserted in the introduction, one can extend these calculations to obtain a variance bound $\Var N_I = O(n^{1-c})$.  We sketch the argument as follows.  As before, we may assume that $I$ is contained in the region $\{ x: n^{1/4} \leq |x| \leq \sqrt{n} - n^{1/4} \}$.  In addition to the bound \eqref{lemo} just established, one needs to establish the additional bound
\begin{equation}\label{limo}
 \int_\R \int_\R F(x) F(y) \rho^{(2,0)}_{f_{n,\xi}}(x,y)\ dx = (\frac{1}{\pi} \int_\R F(x)\ dx)^2 + O(n^{1-c}).
\end{equation}
Using Theorem \ref{tmt-real} as before, we may replace $\xi$ by $\tilde \xi$.  We can then apply the Kac-Rice formula and Lemma \ref{gauss} to conclude that
$$ \rho^{(2,0)}_{f_{n,\tilde \xi}(x,y)} = \frac{1}{2\pi} |v(x) \wedge v(y)|^{-1} \E( |W_x| |W_y| | (V_x,V_y) = (0,0) )$$ 
where $V_x, V_y, W_x, W_y$ are real gaussian random variables with mean zero and covariance matrix
$$
\E \begin{pmatrix} 
V_x^2 & V_x V_y & V_x W_x & V_x W_y \\
V_y V_x & V_y^2 & V_y W_x & V_y W_y \\
W_x V_x & W_x V_y & W_x^2 & W_x W_y \\
W_y V_x & W_y V_y & W_y W_x & W_y^2
\end{pmatrix}
=
\begin{pmatrix} |v(x)|^2 & v(x)\cdot v(y) & v(x) \cdot v_x(x) & v(x) \cdot v_x(y) \\
v(y) \cdot v(x) & |v(y)|^2 & v(y) \cdot v_x(x) & v(y) \cdot v_x(y) \\
v_x(x) \cdot v(x) & v_x(x) \cdot v(y) & |v_x(x)|^2 & v_x(x) \cdot v_x(y) \\
v_x(y) \cdot v(x) & v_x(y) \cdot v(y) & v_x(y) \cdot v_x(x) & |v_x(y)|^2
\end{pmatrix}.
$$
A rather tedious calculation along the lines of those used in the proof of Theorem \ref{tmt-zeroes} reveals that this covariance matrix differs from the identity matrix by $O( \exp(-|x-y|^2/10) ) + O( n^{-\eps + o(1)} )$ (say) in the region $n^{1/4}/2 \leq |x|, |y| \leq \sqrt{n}- n^{1/4}/2$, which implies that
$$ \rho^{(2,0)}_{f_{n,\tilde \xi}(x,y)} = \frac{1}{\pi^2} + O( \exp(-|x-y|^2/10) ) + O( n^{-\eps + o(1)} )$$
which gives \eqref{limo} for $\tilde \xi$.  We omit the details.

\section{Universality  for elliptic polynomials} \label{section:ellip} 

In this section we establish our main results for elliptic polynomials, namely Theorems \ref{tmt-elliptic}, \ref{tmt-elliptic-real}, \ref{rz-e}.  Our arguments here will be closely analogous to those for flat polynomials in the previous section.

This will largely be accomplished by invoking the results obtained in previous sections.  Again, our starting point is the concentration of log-magnitudes.

\begin{lemma}[Concentration for log-magnitude]  Let $C, \eps > 0$ be constants, let $n$ be a natural number, and let $z$ be a complex number with
$$n^{\eps} \le |z|  \le n^{1-\eps}.$$
Let $f = f_{n,\xi}$ be a rescaled elliptic polynomial whose atom distribution $\xi$ has mean zero and variance one with $\E |\xi|^{2+\eps} \leq C$.  Then with overwhelming probability, one has
$$ \log |f(z)| = \frac{1}{2} n \log(1 + \frac{|z|^2}{n}) + O(n^{o(1)}).$$
The implied constants in the asymptotic notation can depend on $C,\eps$.
\end{lemma}

\begin{proof}  As before, we first compute the quantity $V(z)$ from \eqref{V-def}.  This quantity is given by
$$ V(z) = \sum_{i=0}^n \frac{|z|^{2i}}{n^i} \binom{n}{i} = (1 + |z|^2/n)^n.$$
In particular, we have 
$$ \log V(z) = n \log(1 + \frac{|z|^2}{n}).$$
Applying Lemma \ref{lemma:concentration-general} as before,
 indices $i_1, \dots, i_m \in \{0,\ldots,n\}$ for some $m = \omega(\log n)$ such that we have the lacunarity property
$$ \left|\sqrt { {\binom{n}{i_j}} n^{-i_j} } z^{i_j}/\sqrt{i_j!}\right|  \ge  2 \left|\sqrt { {\binom{n}{i_{j+1}}} n^{-i_{j+1}} } z^{i_{j+1} }\right| $$
for all $1 \leq j < m$, and the lower bound
$$
 |\sqrt { {\binom{n}{i_m}} n^{-i_m} } z^{i_m}|  \ge   V(z)^{1/2} \exp (- n^{o(1)} ).
$$

The sequence $i \mapsto |\sqrt{\binom{n}{i} n^{-i}} z^i|$ is increasing for $i < \frac{|z|^2}{1+|z|^2/n}$ and decreasing for $i>\frac{|z|^2}{1+|z|^2/n}$, with its largest value being at least $(V(z)/(n+1))^{1/2}$.  If $n^\eps \leq |z| \leq \sqrt{n}$, the ratio between adjacent elements of this sequence is $O(1)$ in the range
$$ \frac{1}{2} \frac{|z|^2}{1+|z|^2/n} \leq i \leq \frac{|z|^2}{1+|z|^2/n},$$
and the claim then follows by applying Lemma \ref{lemma:length} to the (reversal of) this subsequence.  Conversely, if $\sqrt{n} \leq |z| \leq n^{1-\eps}$, then the ratio between adjacent elements of this sequence is $O(1)$ in the range
$$ \frac{1}{2} \frac{n}{1+|z|^2/n} \leq n-i \leq \frac{n}{1+|z|^2/n},$$
and the claim follows by Lemma \ref{lemma:length} to this subsequence.
\end{proof}

If we set
$$ G(z) := \frac{1}{2} n \log(1 + \frac{|z|^2}{n})$$
then a short computation shows that
$$ \Delta G(z) = \frac{2}{(1+|z|^2/n)^2}$$
Applying Proposition \ref{critloc}, we conclude that for any $n^{-C} \leq c \leq r \leq n^{1-\eps}/3$ and $z_0 \in B(0,n^{1-\eps}/3)$ with the property that $B(z_0,r+c) \backslash B(z_0,r-c)$ is disjoint from $B(0,n^{\eps})$, with overwhelming probability $f$ is non-vanishing and obeys
\begin{equation}\label{aaa2}
 N_{B(z_0,r)}(f) = \int_{B(z_0,r)} \frac{1}{\pi} (1+|z|^2/n)^{-2}\ dz + O( n^{o(1)} c^{-1} r ) + 
O\left( \int_{B(z_0,r+c) \backslash B(z_0,r-c)} (1+|z|^2/n)^{-2}\ dz \right).
\end{equation}
 This already gives axiom (i) for Theorems \ref{replace}, \ref{replace-real}.  Setting $c=2/r$, we conclude the non-concentration estimate
\begin{equation}\label{lcl-ellip}
 N_{B(z_0,r)} \ll n^{o(1)} r^2
\end{equation}
with overwhelming probability for any $z_0 \in B(0,C\sqrt{n})$ and any $r \geq 1$ such that
 $B(z_0,r+c) \backslash B(z_0,r-c)$ is disjoint from $B(0,n^{\eps})$ (note that the claim is trivial for say $r \geq \sqrt{n}/3$); this gives axiom (ii) for Theorems \ref{replace}, \ref{replace-real}.  

As before, the next stage is to establish axiom (iii) for Theorems \ref{replace}, \ref{replace-real}.

\begin{proposition}[Comparability of log-magnitudes]
Let $C, \eps > 0$ be constants, and let $c_0 > 0$ be sufficiently small depending on $\eps$.  Let $n$ be a natural number, let $1 \leq k \leq n^{c_0}$ be another natural number, and let $z_1,\ldots,z_k$ be complex numbers such that
$$n^{\eps} \le |z|  \le C n^{1/2}.$$
Let $f_{n,\xi}, f_{n,\tilde \xi}$ be normalized elliptic polynomials whose atom distributions $\xi, \tilde \xi$ have mean zero and variance one matching moments to second order with $\E |\xi|^{2+\eps}, |\tilde \xi|^{2+\eps} \leq C$.  Let $F: \C^k \to \C$ be a smooth function obeying the bounds
$$ |\nabla^a F(z)| \leq C$$
for all $0 \leq a \leq 3$.  Then, if $c_0$ is sufficiently small, one has
$$ \E \Big( F(\log |f(z'_1)|, \ldots, \log |f(z'_{k'})|) - F(\log |\tilde f(z'_1)|, \ldots, \log |\tilde f(z'_{k'})|)  \Big) = O(n^{-c_0}),$$
where the implied constant in the $O()$ notation depends on $C,\eps,c_0$.
\end{proposition}

This is proven in exact analogy with Proposition \ref{clm} (with the quantity $|z_j|^2$ being replaced by $\frac{|z_j|^2}{1+|z_j|^2/n}$); we leave the details to the interested reader.

Applying Theorem \ref{replace}, we now obtain Theorem \ref{tmt-elliptic} as an immediate consequence.  To similarly use Theorem \ref{replace-real} to deduce Theorem \ref{tmt-elliptic-real}, we need the following analogue of Proposition \ref{level-repuls}:

\begin{proposition}[Level repulsion]\label{level-repuls-elliptic}
Let $\eps>0$, and let $C > 1$ be a sufficiently large constant. Let $n$ be a natural number, and let $x,y \in \R$ and $z \in \C$ be such that
\begin{equation}\label{xyz-elliptic}
n^\eps \leq |x|, |y|, |z| \leq C \sqrt{n}
\end{equation}
and
$$ |x-y|, |\Im z| \leq 1/C.$$
Let $f = f_{n,\xi}$ be a rescaled elliptic polynomial whose atom distribution is drawn from the real gaussian ensemble $N(0,1)_\R$.  Then we have the pointwise bounds
\begin{equation}\label{rho-20a-elliptic}
\rho^{(2,0)}_{\tilde f}(x,y)  \ll |x-y|
\end{equation}
and 
\begin{equation}\label{rho-01a-elliptic}
\rho^{(0,1)}_{\tilde f}(z)  \ll |\Im z|,
\end{equation}
where the implied constants depend on $C$.
\end{proposition}

\begin{proof}  
Applying Lemma \ref{lemma:RB1} with the holomorphic function $R(z) := (1+z^2/n)^{-n/2}$ on $B(x,1)$ (noting that we are well away from the poles $\pm \sqrt{-1} \sqrt{n}$ of this function), it suffices to establish the bounds
\begin{align}
|v(z)| &\ll 1 \label{vo-a-e}\\
\left|v(x) \wedge v'(x)\right| &\gg 1 \label{vos-a-e} 
\end{align}
for all $x \in \R$ with $n^\eps/2 < |x| \leq 2C\sqrt{n}$ and all $z \in B(x,1)$.

To prove \eqref{vo-a-e}, we compute
\begin{align*}
 |v(z)|^2 &:= |1+z^2/n|^{-n} \sum_{i=0}^n \binom{n}{i} n^{-i} |z|^{2i} \\
&= (\frac{1+|z|^2/n}{|1+z^2/n|})^n.
\end{align*}
But one can compute that $|1+z^2/n| = 1 + |z|^2/n + O(1/n)$, and the claim \eqref{vo-a-e} follows.

To obtain \eqref{vos-a-e}, we compute
$$ v'(x) = (1+x^2/n)^{-n/2} \left( (\frac{i-x^2/(1+x^2/n)}{x}) \sqrt{\binom{n}{i} n^{-i}} x^i \right)_{i=0}^n$$
and so
$$ |v(x) \wedge v'(x)|^2 = (1+x^2/n)^{-2n} \sum_{0 \leq i<j \leq n} \frac{|i-j|^2}{x^2} \binom{n}{i} n^{-i} x^{2i} \binom{n}{j} n^{-j} x^{2j}.$$
The expression on the right-hand side can in fact be computed exactly, but for the purposes of establishing the lower bound \eqref{vos-a-e}, we may appeal instead to Stirling's formula, which reveals that
$$ \binom{n}{i} n^{-i} x^{2i} \gg \frac{1}{x} (1+x^2/n)^n$$
when $i = \frac{x^2}{1+x^2/n} + O(x)$, and the claim follows much as in the analogous computation for flat polynomials in Proposition \ref{level-repuls}.
\end{proof}

As all of the hypotheses (i)-(iv) of Theorem \ref{replace-real} are obeyed, Theorem \ref{tmt-elliptic-real} is now established.

Finally, we establish Theorem \ref{rz-e}.  We need to show 

\begin{equation}\label{eni}
 \E N_I = \int_I \frac{1}{\pi} \frac{dx}{1+x^2/n} + O(n^{1/2-c})
\end{equation}
for all intervals $I$.

It will be convenient to use inversion symmetry to work in the region $I \subset [-\sqrt{n},\sqrt{n}]$.  Observe that if
$$ f_{n,\xi}(z) = \sum_{i=0}^n \xi_i \sqrt{\binom{n}{i} n^{-i}} z^i$$
is a rescaled elliptic polynomial, then
$$ n^{-n/2} z^n f_{n,\xi}(n/z) = \sum_{i=0}^n \xi_{n-i} \sqrt{\binom{n}{i} n^{-i}} z^i$$
is also a rescaled elliptic polynomial with the same distribution as $f_{n,\xi}$.  Thus the distribution of the zeroes of $f_{n,\xi}$ are invariant with respect to the inversion map $z \mapsto n/z$.  Among other things, this implies that if \eqref{eni} holds for an interval $I$ avoiding the origin, then it also holds for the inverse interval $\{ n/x: x \in I \}$.  From this (and \eqref{lcl-ellip}) it thus suffices to establish \eqref{eni} in the case $I \subset [-\sqrt{n},\sqrt{n}]$.

Covering the interval $[-n^{1/4},n^{1/4}]$ by $O(n^{1/4})$ balls of unit radius and then applying \eqref{lcl-ellip}, we see that this interval has $O(n^{1/4+o(1)})$ zeroes with overwhelming probability.  This establishes \eqref{eni} when $I \subset [-n^{1/4},n^{1/4}]$, so we may assume without loss of generality that $I \subset [-\sqrt{n},-n^{1/4}]$ or $I \subset [n^{1/4},\sqrt{n}]$.

By using Theorem \ref{tmt-elliptic-real} (and upper and lower bounding $1_I$ by smooth functions), it suffices to establish \eqref{eni} in the case when $I \subset \{ x: n^{1/4}/2 \leq |x| \leq 2\sqrt{n} \}$ and $\xi$ has the real gaussian distribution $N(0,1)_\R$.  By the Kac-Rice formula as in the previous section, we have
$$ \E N_I = \int_I \rho^{(1,0)}(x)\ dx$$
where
$$ \rho^{(1,0)}(x) =\frac{1}{\pi} \frac{|\operatorname{dist}(v'(x), v(x))|}{|v(x)|},$$
$$ v(x) := (1+x^2/n)^{-n/2} (\sqrt{\binom{n}{i}n^{-i}} x^i)_{i=0}^n$$
and
$$ v'(x) := (1+x^2/n)^{-n/2} \left(\frac{i-\frac{x^2}{1+x^2/n}}{x} \sqrt{\binom{n}{i}n^{-i}} x^i\right)_{i=0}^n.$$
We have
\begin{align*}
|v(x)|^2 &= (1+x^2/n)^{-n} \sum_{i=0}^n \binom{n}{i} n^{-i} x^{2i} \\
&= 1
\end{align*}
and thus on differentiation
$$ v(x) \cdot v'(x) = 0.$$
We also have
$$
|v'(x)|^2 = (1+x^2/n)^{-n} \sum_{i=0}^n \frac{(i-\frac{x^2}{1+x^2/n})^2}{x^2} \binom{n}{i} n^{-i} x^{2i}.$$
We can differentiate the binomial identity
$$ (1+x^2/n)^n = \sum_{i=0}^n \binom{n}{i} n^{-i} x^{2i}$$
to obtain
$$ \frac{x^2}{1+x^2/n} (1+x^2/n)^n = \sum_{i=0}^n i \binom{n}{i} n^{-i} x^{2i}$$
and
$$ \frac{x^4+1}{(1+x^2/n)^2})^2 = \sum_{i=0}^n i^2 \binom{n}{i} n^{-i} x^{2i}$$
(these are also the formulae for the mean and variance of a binomial random variable) and so
$$
|v'(x)|^2 = \frac{1}{(1+x^2/n)^2}.$$
This implies that
$$ \rho^{(1,0)}(x) =\frac{1}{\pi} \frac{1}{1+x^2/n}$$
and \eqref{eni} follows (indeed, the formula is even exact in this case). As a matter of fact, we can improve the error term $O(n^{1/2-\epsilon})$ to $O(| I| n^{-\epsilon})$ in this case. 

\section{Universality for Kac polynomials}\label{kac-sec} 

We now prove Theorems \ref{tmt-kac} and  \ref{tmt-kac-real}.
 As before, the first step is to obtain concentration results for the log-magnitude $\log |f(z)|$.  Here, a new difficulty arises: when $z$ is a root of unity, then $f(z)$ can vanish with polynomially small probability.  For instance, if $z=1$, $n$ is odd, and $\xi$ has the Bernoulli distribution (thus $P(\xi=+1)=P(\xi=-1)=1/2$), then $f(1) = \xi_0+\ldots+\xi_n$ vanishes with probability comparable to $1/\sqrt{n}$, as can be easily verified using Stirling's formula. The key new idea in  our proof is to show that this is  essentially  the only obstruction to concentration of the log magnitude, provided that $z$  stays away from zero and from infinity. 

\begin{lemma}[Concentration for log-magnitude]\label{cloc-kac}  Let $C, A, \eps > 0$ be constants, let $n$ be a natural number, and let $z$ be a complex number with
$$ \eps \le |z|  \le 1/\eps.$$
Let $f = f_{n,\xi}$ be a Kac polynomial whose atom distribution $\xi$ has mean zero and variance one with $\E |\xi|^{2+\eps} \leq C$.  Then one of the following holds:
\begin{itemize}
\item[(i)] $\eps \leq |z| \leq 1$, and one has $\log |f(z)| = O(n^{o(1)})$ with probability $1-O(n^{-A})$.
\item[(ii)] $1 \leq |z| \leq 1/\eps$, and one has $\log |f(z)| = n \log |z| + O(n^{o(1)})$ with probability $1-O(n^{-A})$.

\item[(iii)]  One has $z = \omega + O(n^{-A})$, where $\omega$ is a root of unity with $\omega^k=1$ for some $k = O(1)$.
\end{itemize}
The implied constants in the asymptotic notation can depend on $C,\eps,A$.
\end{lemma}

\begin{proof}  We may assume $n$ sufficiently large depending on $C,\eps,A$, as the claim is trivial otherwise.

Note that $z \mapsto z^n f(1/z)$ is a Kac polynomial with the same distribution as $f$, so the claims in this lemma for $|z| \geq 1$ will follow from the claims when $|z| \leq 1$.  Thus we may assume without loss of generality that $\eps \leq |z| \leq 1$.

If $\eps \leq |z| \leq 1 - \frac{\log^2 n}{n}$ (say) then the claim (i) follows easily from Lemma \ref{lemma:concentration-general} (and Lemma \ref{lemma:length}), so we may assume that
\begin{equation}\label{xip}
 1 - \frac{\log^2 n}{n} \leq |z| \leq 1.
\end{equation}

The random variable $f(z)$ has mean zero and variance
$$ \E |f(z)|^2 = \sum_{i=0}^n |z|^{2i} \leq n$$
and so by Chebyshev's inequality we certainly have the upper bound
$$ \log |f(z)| \leq n^{o(1)}$$
with overwhelming probability.  If we have
$$ \log |f(z)| \geq - \log^3 n$$
(say) with probability at least $1-n^{-A}$ we are done, so suppose instead that
$$ \P( \log |f(z)| \geq -\log^2 n ) < 1-n^{-A},$$
thus
\begin{equation}\label{fzz}
 \P( |f(z)| < \exp(-\log^3 n) ) > n^{-A}.
\end{equation}
The quantity $f(z) = \sum_{i=0}^n \xi_i z^i$ is a sum of independent random variables, and so \eqref{fzz} is an assertion that the \emph{small ball probability} of this random sum is large.  We can use this to constrain the coefficients $z^i$ of $f(z)$ by means of \emph{inverse Littlewood-Offord theorems}.  There are many such theorems in the literature; we will use \cite[Theorem 2.9]{nguyen}.  We first note from Lemma \ref{lam} that
$$ \P( B^{-1} \leq |\xi-\xi'| \leq B ) \gg 1$$
for some $B = O(1)$, if $\xi'$ is an independent copy of $\xi$.  This is essentially\footnote{In \cite{nguyen} the lower bound on $\P( B^{-1} \leq |\xi-\xi'| \leq B )$ is $1/2$ rather than $\gg 1$, but one can verify that the arguments in that paper are not significantly changed if one alters the lower bound, provided of course one allows all subsequent constants to depend on this new lower bound.} the hypothesis in \cite[(7)]{nguyen} up to some rescalings.  If one applies \cite[Theorem 2.9]{nguyen}, one can then complex numbers $v_1,\ldots,v_r$ for some $r=O(1)$ with the property that for all but at most $\sqrt{n}$ (say) of the numbers $z^i, 0 \leq i \leq n$, one has a representation of the form
\begin{equation}\label{zi-rep}
 z^i = a_{i,1} v_1 + \ldots + a_{i,r} v_r + O( n^{O(1)} \exp(-\log^3 n) )
\end{equation}
where $a_{i,1},\ldots,a_{i,r}$ are integers of magnitude $O(n^{O(1)})$; in particular, by the pigeonhole principle we can find $0 \leq i_0 \leq n-\sqrt{n}$ such that one has a representation \eqref{zi-rep} for all $i_0 \leq i \leq i_0+\sqrt{n}$.  Actually, the results in \cite{nguyen} provide significantly more precise results than this, but these bounds will suffice for our purposes.

It will be convenient to ensure that the generators $v_1,\ldots,v_r$ are approximately linearly independent in a certain sense.  Observe that if we have an approximate linear relation between the $v_1,\ldots,v_r$ of the form
\begin{equation}\label{bvv}
 b_1 v_1 + \ldots + b_r v_r = O( n^{O(1)} \exp(-\log^3 n) )
\end{equation}
for some integers $b_1,\ldots,b_r = O(n^{O(1)})$, not all zero, then after clearing denominators we can eliminate one of the $v_i$ from the basis $v_1,\ldots,v_r$ and divide all the other elements by integers of size $O(n^{O(1)})$ and obtain a new basis of $r-1$ elements for which one still has representations of the form \eqref{zi-rep} (with worse values of implied constants in the $O()$ notation) for all $i_0 \leq i \leq i_0+\sqrt{n}$.  Iterating this observation at most $r$ times, we may assume without loss of generality that there is no\footnote{Strictly speaking, one has to take some care with the asymptotic notation $O()$ in order to make this statement rigorous.  There are several (essentially equivalent) ways in which this can be achieved.  One is to reformulate the current argument (which is written in the context of a fixed $n$) in asymptotic fashion, involving a sequence of values of $n$ tending to infinity, with $O(1)$ now referring to a quantity that is bounded uniformly in $n$, at which point there is no difficulty interpreting the argument here rigorously.  Another approach, which we will not detail here, is to reformulate the argument in the language of nonstandard analysis via the device of forming an ultraproduct, so that $n$ is now a nonstandard natural number rather than a standard one, and the $O()$ notation again has a precise interpretation.  If instead one wishes to stay in the context of a fixed (standard) $n$, then one interprets \eqref{bvv} as the claim that there are no integers $b_1,\ldots,b_r$ of magnitude at most $F(C) n^{F(C)}$, not all zero, for which $|b_1 v_1 + \ldots + b_r v_r| \leq F(C) n^{F(C)} \exp(-\log^3 n)$, where $C$ bounds all the implied constants in previous usages of asymptotic notation (such as \eqref{zi-rep}) and $F: \R^+ \to \R^+$ is a sufficiently rapidly growing function (not depending on $n$) to be chosen later.} linear relation of the form \eqref{bvv}.

Among other things, this approximate linear independence shows (if the $O()$ notation is suitably interpreted) that the $r$-tuple $\vec a_i := (a_{i,1},\ldots,a_{i,r}) \in \Z^r$ appearing in \eqref{zi-rep} is uniquely defined for each $i_0 \leq i \leq i_0+\sqrt{n}$.  From \eqref{xip} and the approximate linear independence we also see that the $\vec a_i$ are all non-zero.

Note that for any $i_0 \leq i_1 \leq i_2 \leq i_0+\sqrt{n}$, the linear span $V_{i_1,i_2}$ of the vectors $\vec a_i$ for $i_1 \leq i \leq i_2$ has dimension between $1$ and $r$, and is non-decreasing in $i_2$ and non-increasing in $i_1$.  By the pigeonhole principle, one can thus find 
$$i_0+0.1 \sqrt{n} \leq i_1 \leq i_2 \leq i_0 + 0.9 \sqrt{n}$$
(say) such that
$$V_{i_1,i_2} = V_{i_1-n^{1/4},i_2+n^{1/4}}$$
(say).  In particular, we have $\vec a_i \in V_{i_1,i_2}$ for all $i_1-n^{1/4} \leq i \leq i_2 + n^{1/4}$.

By the Steinitz exchange lemma, we can find a sequence $\vec a_{j_1},\ldots, \vec a_{j_d}$ for $i_1 \leq j_1 < \ldots < j_d \leq i_2$ and some $d=O(1)$ that span $V_{i_1}$.  Then for any integer $m$ with $-n^{1/4} \leq m \leq n^{1/4}$, the vectors $\vec a_{j_1+m},\ldots, \vec a_{j_d+m}$ are linear combinations of $\vec a_{j_1},\ldots,\vec a_{j_d}$, thus we have
\begin{equation}\label{soi}
 \begin{pmatrix} \vec a_{j_1+m} \\ \ldots \\ \vec a_{j_d+m} \end{pmatrix} = T_m 
\begin{pmatrix} \vec a_{j_1} \\ \ldots \\ \vec a_{j_d} \end{pmatrix} 
\end{equation}
for some (unique) $d \times d$ matrix $T_m$ with rational entries.  From Cramer's rule we see that all entries of $T_m$ have height $O(n^{O(1)})$ (i.e. their numerator and denominator are $O(n^{O(1)})$) for any $-n^{1/4} \leq m \leq n^{1/4}$.

Clearly $T_0$ is the identity matrix.  We claim that
\begin{equation}\label{tim}
 T_{m+1} = T_m T_1
\end{equation}
for any $-n^{1/4} \leq m \leq n^{1/4}-1$, which implies by induction that $T_1$ is invertible and $T_m = T_1^m$ for all $-n^{1/4} \leq m \leq n^{1/4}$.  To see this, we note from \eqref{soi} and \eqref{zi-rep} that
\begin{equation}\label{dose}
 \begin{pmatrix} z^{j_1+m} \\ \ldots \\ z^{j_d+m} \end{pmatrix} = T_m 
\begin{pmatrix} z^{j_1} \\ \ldots \\ z^{j_d} \end{pmatrix} + O(n^{O(1)} \exp(-\log^3 n)).
\end{equation}
Multiplying by $z$, we see that
$$
 \begin{pmatrix} z^{j_1+m+1} \\ \ldots \\ z^{j_d+m+1} \end{pmatrix} = T_m 
\begin{pmatrix} z^{j_1+1} \\ \ldots \\ z^{j_d+1} \end{pmatrix} + O(n^{O(1)} \exp(-\log^3 n))$$
and by comparing this with \eqref{dose} with $m$ replaced by $m+1$ and by $1$, we conclude that
$$
T_{m+1} \begin{pmatrix} z^{j_1} \\ \ldots \\ z^{j_d} \end{pmatrix} = T_m T_1
\begin{pmatrix} z^{j_1} \\ \ldots \\ z^{j_d} \end{pmatrix} + O(n^{O(1)} \exp(-\log^3 n)).$$
Using \eqref{zi-rep} and the approximate linear independence of the $v_1,\ldots,v_r$, we conclude that
$$
T_{m+1} \begin{pmatrix} \vec a_{j_1} \\ \ldots \\ \vec a_{j_d} \end{pmatrix} = T_m T_1
\begin{pmatrix} \vec a_{j_1} \\ \ldots \\ \vec a_{j_d} \end{pmatrix}$$
and from the linear independence of the $\vec a_{j_1}, \ldots, \vec a_{j_d}$ we conclude \eqref{tim}.

We now see that $T_1$ is a matrix in $\operatorname{GL}_d(\Q)$ with the property that $T_1^m$ has entries of height $O(n^{O(1)})$ for all $-n^{1/4} \leq m \leq n^{1/4}$.  At this point we need to establish the claim that the minimal polynomial of $T_1$ is monic over the integers, or equivalently that the eigenvalues of $T_1$ are algebraic integers.  Suppose this were not the case, then we have a relation of the form
$$ a_{d'} T_1^{d'} = a_{d'-1} T_1^{d'-1} + \dots + a_0$$
for some $d'=O(1)$ and some integers $a_{d'},a_{d'-1},\dots,a_0$ with $1, T_1, \dots, T_1^{d'-1}$ linearly independent, with $a_{d'}$ divisible by some prime $p$, and at least one of the $a_0,\dots,a_{d'-1}$ not divisible by $p$.  By induction, one then sees that for any integer $j \geq 0$, one has
$$ a_{d'}^{j+1} T_1^{d'+j} = a_{d'-1,j} T_1^{d'-1} + \dots + a_{0,j}$$
for some integers $a_{d'-1,j},\dots,a_{0,j}$, with at least one of the $a_{i,j}$ not divisible by $p$.  In particular, one of the rational numbers $\frac{a_{i,j}}{a_{d'}^{j+1}}$ has height at least $p^{j+1}$.  On the other hand, from Cramer's rule we see that if $d'+j \leq n^{1/4}$, then these rational numbers must have height $O(n^{O(1)})$.  This leads to a contradiction if one sets $j$ comparable to a small multiple of $n^{1/4}$.  Thus all the eigenvalues of $T_1$ are algebraic integers, so that $T_1$ is conjugate to a matrix $T'_1$ in $\operatorname{SL}_d(\Z)$.  All the powers of $(T'_1)^m$ for $-n^{1/4} \leq m \leq n^{1/4}$ then have entries that are integers of magnitude $O(n^{O(1)})$.  This is a polynomial growth condition on $T'_1$, and one can use results related\footnote{The situation here does not require the full strength of Gromov's theorem from \cite{gromov} (or quantitative versions thereof), and is actually closer to the older work of \cite{milnor} and \cite{wolf} treating polynomial growth in solvable groups.} to quantitative versions \cite{shalom} of Gromov's theorem to then force $T'_1$ and hence $T_1$ to be virtually unipotent.  Indeed, if we apply \cite[Proposition 13.1]{shalom}, we conclude that there exists a non-zero vector $\vec c \in \Z^d$ and a natural number $k = O(1)$ such that $T_1^k \vec c = \vec c$.  From Cramer's rule we can take $\vec c = (c_1,\ldots,c_d)$ to have magnitude $O(n^{O(1)})$.  From \eqref{soi} and \eqref{zi-rep} we see that
\begin{equation}\label{panic}
\sum_{l=1}^d c_l z^{j_l+k} = \sum_{l=1}^d c_l z^{j_l} + O( n^{C} \exp(-\log^3 n) )
\end{equation}
for some $C = O(1)$.  If we had
$$ |\sum_{l=1}^d c_l z^{j_l}| \leq n^{A+C} \exp(-\log^3 n) $$
then by \eqref{zi-rep} this would contradict the approximate linear independence of the $v_1,\ldots,v_r$, the actual linear independence of the $\vec a_{j_1},\ldots, \vec a_{j_d}$, and the non-zero nature of $\vec c$, so we have
$$ |\sum_{l=1}^d c_l z^{j_l}| > n^{A+C} \exp(-\log^3 n) $$
and hence from \eqref{panic} one has $z^k = 1 + O(n^{-A})$, which gives the conclusion (iii).
\end{proof}

As the roots of unity are fairly sparse, they can be avoided for the purposes of obtaining non-concentration bounds on zeroes:

\begin{lemma}[Non-clustering bounds]\label{nclb}  Let $C, \eps > 0$ be constants, and let $n$ be a natural number.
Let $f = f_{n,\xi}$ be a Kac polynomial whose atom distribution $\xi$ has mean zero and variance one with $\E |\xi|^{2+\eps} \leq C$.  Let $B(z_0,r)$ be a ball in the complex plane.  Then one has
$$ N_{B(z_0,r)} \ll n^{o(1)} (1 + nr)$$
with overwhelming probability.  If in addition $B(z_0,2r)$ is disjoint from the unit circle $\{ z \in \C: |z| = 1 \}$, one can improve this bound to
$$ N_{B(z_0,r)} \ll n^{o(1)}$$
with overwhelming probability.  Furthermore, $f$ is non-vanishing with overwhelming probability.
\end{lemma}

\begin{proof}  We apply Proposition \ref{critloc} with the function $G: \C \to \R$ defined by setting $G(z) := 0$ for $|z| \leq 1$, and $G(z) := n \log |z|$ for $|z| > 1$.  Strictly speaking, $G$ is not smooth enough for Proposition \ref{critloc} to apply as stated, but this technical difficulty can be overcome by a routine infinitesimal mollification which we omit here.  One can compute from Green's theorem that
$$ \frac{1}{2\pi} \Delta G = n d\sigma$$
in the sense of distributions, where $d\sigma$ is the uniform probability measure on the unit circle $\{z: |z|=1\}$.  From Proposition \ref{critloc} (with Remark \ref{conc-strong}) and Lemma \ref{cloc-kac}, we conclude that $f$ is non-vanishing with overwhelming probability, and for any fixed $A>0$, any ball $B(z_0,r)$, and any $0 < c \leq r$ with $r \ll n^{O(1)}$ and $c \gg n^{-O(1)}$, one has
$$
 N_{B(z_0,r)}(f) \ll n \int_{B(z_0,r+c)} d\sigma + O( n^{o(1)} c^{-1} r ) $$
whenever the region $B(z_0,r+c) \backslash B(z_0,r+c)$ lies in the annulus $\{ z: 1/10 \leq |z| \leq 10 \}$ (say) and also avoids the disks $B(\omega,n^{-A})$ whenever $\omega^k = 1$ for some $k \leq C_A$, where $C_A$ depends only on $A$.  In particular, under these hypotheses we have
$$
 N_{B(z_0,r)}(f) \ll n^{o(1)} ( nr + c^{-1} r )$$
together with the stronger bound
$$
 N_{B(z_0,r)}(f) \ll n^{o(1)} c^{-1} r$$
when $B(z_0,r+c)$ is disjoint from the unit circle.

From these estimates, we see that
$$ N_{B(0,0.9)}(f) \ll n^{o(1)}$$
(say) with probability $1-O(n^{-A})$, and by setting $c=r$ we also see that
$$ N_{B(z_0, (1-|z_0|)/10)}(f) \ll n^{o(1)}$$
with probability $1-O(n^{-A})$ for any $z_0$ with $1/2 \leq |z_0| \leq 1 - 1/n$.  By letting $z_0$ lie on the unit circle, setting $c$ to be a small multiple of $r$, and enlarging $r$ as necessary in order for $B(z_0,r+c) \backslash B(z_0,r+c)$ to avoid the disks $B(\omega,n^{-A})$, we also conclude that
$$ N_{B(z_0,r)} \ll n^{o(1)} (1 + nr )$$
with probability $1-O(n^{-A})$ for any $0 < r < 1/10$ (say) and $z_0$ on the unit circle.  By diagonalization in $A$, these events in fact hold with overwhelming probability.  From these bounds and covering argument one obtains the required bounds for the contribution of the zeroes on or inside the unit circle; the contribution of the zeroes outside the unit circle can then be handled by exploiting the invariance of the distribution of the zeroes with respect to the transformation $z \mapsto 1/z$.
\end{proof}

Following our treatment of the flat and elliptic polynomials, the next step is to obtain comparability of log-magnitudes.

\begin{proposition}[Comparability of log-magnitudes]\label{clm-kac}
Let $C, \eps > 0$ be constants, and let $c_0 > 0$ be sufficiently small depending on $\eps$.  Let $n$ be a natural number, let $1 \leq k \leq n^{c_0}$ be another natural number, and let $z_1,\ldots,z_k$ be complex numbers such that
$$ 1-n^{-\eps} \le |z_j|  \le 1+n^{-\eps}$$
for $j=1,\ldots,k$.  Let $f_{n,\xi}, f_{n,\tilde \xi}$ be Kac polynomials whose atom distributions $\xi, \tilde \xi$ have mean zero and variance one matching moments to second order with $\E |\xi|^{2+\eps}, |\tilde \xi|^{2+\eps} \leq C$.  Let $F: \C^k \to \C$ be a smooth function obeying the bounds
$$ |\nabla^a F(z)| \leq C$$
for all $0 \leq a \leq 3$.  Then, if $c_0$ is sufficiently small, one has
$$ \E \Big( F(\log |f(z'_1)|, \ldots, \log |f(z'_{k'})|) - F(\log |\tilde f(z'_1)|, \ldots, \log |\tilde f(z'_{k'})|)  \Big) = O(n^{-c_0}),$$
where the implied constant in the $O()$ notation depends on $C,\eps,c_0$.
\end{proposition}

\begin{proof}  This is immediate from Theorem \ref{tmp-general}, noting that for $1-n^{-\eps} \leq |z_j| \leq 1$ one has $V(z_j)^{1/2} \gg n^{\eps}$ and for $1 \leq |z_j| \leq 1+n^{-\eps}$ one has $V(z_j)^{1/2} \gg |z_j|^n n^\eps$.
\end{proof}

To prove Theorem \ref{tmt-kac}, we would like to apply Theorem \ref{replace} (with $r_0 := 1$ and $a_0 := 3$), after first performing the rescaling
$$ f'(z) := f( 10^{-3} r z )$$
and replacing the $z_j$ by $z'_j := z_j / (10^{-3} r)$; note that the correlation functions rescale according to the law
$$ \rho^{(k)}_{f'}(w_1,\ldots,w_k) := (10^{-3} r)^{2k} \rho^{(k)}_f(10^{-3} r w_1,\ldots,10^{-3} r w_k).$$
Unfortunately, a difficulty arises: the bounds in Lemma \ref{nclb} on the zeroes of $f$, when rescaled to $f'$, do not quite give the non-clustering bounds 
\begin{equation}\label{ncl}
N_{B(z'_i,r')}(f') \leq C n^{1/A} (r')^2
\end{equation}
with probability $1-n^{-A}$ required for the hypotheses of Theorem \ref{replace}; more precisely, this bound is obtained in the range $1 \leq r' \leq 10^2$ (say), but not necessarily for larger values of $r'$.

However, if one inspects the proof of Theorem \ref{replace}, one sees that the only place in that argument in which the non-clustering bound \eqref{ncl} is needed in the range $r' > 100$ is in the proof of Lemma \ref{squarebound}.  Thus, if we can find an alternate proof of that lemma in this situation, we can still obtain the conclusion of Theorem \ref{replace}, which will give Theorem \ref{tmt-kac}.  Fortunately, in the case when all the $|z_j| < 1$, the bounds in Lemma \ref{cloc-kac} allow one to do this as follows.  Firstly, if $K_j$ is the function from Lemma \ref{squarebound} (applied to $f'$ and $z'_j$ instead of $f$ and $z_j$, of course), then from the triangle inequality one has the crude deterministic bound
$$ \|K_j \|_{L^2} \ll n^{O(1)}.$$
Thus by Lemma \ref{sampling}, if one selects $m := n^C$ points $w_1,\ldots,w_m \in B(z_j,50)$ uniformly at random for some sufficiently large fixed $C$ (independent of $A$), one has with probability $1-O(n^{-\delta})$ that
$$ \|K_j \|_{L^2}^2 \ll 1 + \frac{1}{m} \sum_{i=1}^m |K_j(w_i)|^2.$$
On the other hand, from Lemma \ref{nclb} and the union bound, we see that with probability $1-O(n^{-\delta})$, one has $K_j(w_i) = O(n^{o(1)})$ for all $i=1,\ldots,m$.  This gives the desired conclusion for Lemma \ref{squarebound}.  

Finally, we need to address the situation in which some of the $z_j$ are in the regime $|z_j|>1$ rather than $|z_j|<1$.  In such cases, observe that as $H_j$ is orthogonal to the function $n\log(10^{-3} r |z|)$, we may replace $K_j$ in the proof of Theorem \ref{replace} by $K_j - n \log(10^{-3} rz) H_j$ without affecting the rest of the argument. One may then use the second part of Lemma \ref{cloc-kac} rather than the first part, and the previous argument then goes through as before.   (As a matter of fact, we do not need to subtract off the function $L_j$ in this argument.)

To establish Theorem \ref{tmt-kac-real}, we need a rescaled level repulsion estimate:

\begin{proposition}[Level repulsion]\label{level-repuls-kac}
Let $\eps>0$, and let $C > 1$ be a sufficiently large constant.  Let $n$ be a natural number, let $r$ be a radius with
$$ \frac{1}{n} \leq r \leq n^{-\eps},$$
and let $x,y \in \R$ and $z \in \C$ be such that
\begin{equation}\label{xyz-kac}
r \leq \frac{1}{n} + ||x|-1|, \frac{1}{n} + ||y|-1|, \frac{1}{n} + ||z|-1| \leq 2r
\end{equation}
and
and
$$ |x-y|, |\Im z| \leq r/C.$$
Let $f = f_{n,\xi}$ be a Kac polynomial whose atom distribution is drawn from the real gaussian ensemble $N(0,1)_\R$.  Then we have the pointwise bounds
\begin{equation}\label{rho-20a-kac}
\rho^{(2,0)}_{\tilde f}(x,y)  \ll |x-y| / r^3
\end{equation}
and 
\begin{equation}\label{rho-01a-kac}
\rho^{(0,1)}_{\tilde f}(z)  \ll |\Im z| / r^3,
\end{equation}
where the implied constants depend on $C$.
\end{proposition}

\begin{proof}  We will work in the regime when $|x|, |y|, |z| \leq 1+\frac{1}{n}$; the opposing case $|x|, |y|, |z| \geq 1-\frac{1}{n}$ can be treated similarly, and in any event is essentially equivalent to the former case after using the symmetry $z \mapsto 1/z$ of the distribution of the zeroes of a Kac polynomial.

We work with the rescaled polynomials $\tilde f'(z) := \tilde f(rz)$, and note that it suffices to show that
\begin{equation}\label{rho-20a-kac-rescaled}
\rho^{(2,0)}_{\tilde f'}(x',y')  \ll |x'-y'| 
\end{equation}
and 
\begin{equation}\label{rho-01a-kac-rescaled}
\rho^{(0,1)}_{\tilde f'}(z')  \ll |\Im z'|
\end{equation}
where $x' := x/r$, $y' := y/r$, $z' := z/r$.

Applying Lemma \ref{lemma:RB1}, and then undoing the rescaling, it suffices to show that
\begin{align}
|v(z)| &\ll 1 \label{lave}\\
|v(x) \wedge v'(x)| &\gg 1/r \label{sticky}
\end{align}
whenever $z \in B(x,\frac{1}{Cr})$, $-1-\frac{1}{n} \leq x \leq 1+\frac{1}{n}$ is such that
$$ r \leq \frac{1}{n} + ||x|-1| \leq 2r,$$
and
$$ v(z) := R(z) (z^i)_{i=0}^n$$
where we can for instance take
$$ R(z) := (1-z^2+\frac{100}{n})^{1/2}.$$
To prove \eqref{lave}, we expand
$$ |v(z)|^2 \leq |1-z^2+\frac{100}{n}| \sum_{i=0}^n |z|^{2i}.$$
Bounding $|1-z^2+\frac{100}{n}| \ll 1-|z|^2$ and
$$ \sum_{i=0}^n |z|^{2i} \leq \sum_{i=0}^\infty |z|^{2i} = \frac{1}{1-|z|^2}$$
we obtain \eqref{lave}.  To obtain \eqref{sticky}, we compute
$$ v'(x) = R(x) \left( \left(\frac{i}{x} - \frac{x}{1-x^2+\frac{100}{n}}\right) x^i \right)_{i=0}^n$$
and so
$$ |v(x) \wedge v'(x)|^2 = \left|1-x^2+\frac{100}{n}\right|^2 \sum_{0 \leq i,j \leq n} \frac{|i-j|^2}{x^2} x^{2i} x^{2j}.$$
But $1-x^2+\frac{100}{n}$ is comparable to $r$, and $x^{2i}$ is comparable to $1$ when $i = O(1/r)$, and the claim \eqref{sticky} follows.
\end{proof}

Theorem \ref{tmt-kac-real} then follows from Theorem \ref{replace-real} after using the same rescaling used to establish Theorem \ref{tmt-kac},
 and after again using the alternate proof of Lemma \ref{squarebound} indicated above. To be more precise, we use a modification of Theorem \ref{replace-real}
 where   Lemma \ref{nclb} is used as a substitute for the non-clustering axiom. This lemma is sufficiently strong for the proof of Lemma \ref{wlr} that allows us to pass 
 from the real case to complex case.

\vskip2mm 

\acks We would like to thank O. Nguyen, M. Krishnapur and T. Reddy for useful comments. 
T. Tao is  partially supported by a Simons Investigator award from the Simons Foundation and by NSF grant DMS-0649473.  V. Vu is supported by research grants from the NSF and the Air Force.


\end{document}